\documentclass[11pt]{amsart}
\usepackage{amsmath}
\usepackage{amssymb}
\usepackage{amsthm}
\usepackage[mathscr]{eucal}
\usepackage{mathrsfs}
\usepackage{psfrag}
\usepackage{fullpage}
\usepackage{color}
\usepackage{eucal}
\usepackage[dvips]{graphicx}
\usepackage{amsfonts}%

\def\barr{\begin{array}}
\def\earr{\end{array}}
\def\bali{\begin{aligned}}
\def\eali{\end{aligned}}
\def\bearr{\begin{eqnarray}}
\def\eearr{\end{eqnarray}}

\providecommand{\play}{\displaystyle}

\providecommand{\li}{\limits}

\providecommand{\pt}{\partial}

\providecommand{\ra}{\rightarrow}

\providecommand{\da}{\downarrow}

\providecommand{\E}{\mathbb E}

\providecommand{\al}{\alpha}

\providecommand{\bt}{\beta}

\providecommand{\Gm}{\Gamma}

\providecommand{\dt}{\delta}

\providecommand{\Dt}{\Delta}

\providecommand{\ve}{\varepsilon}

\providecommand{\kp}{\kappa}

\providecommand{\lb}{\lambda}

\providecommand{\sm}{\sigma}

\providecommand{\R}{\mathbb R}

\providecommand{\T}{\mathbb T}

\providecommand{\cA}{\mathcal A}

\providecommand{\cB}{\mathcal B}

\providecommand{\cC}{\mathcal C}

\providecommand{\cD}{\mathcal D}

\providecommand{\cF}{\mathcal F}

\providecommand{\cL}{\mathcal L}

\providecommand{\cV}{\mathcal V}

\providecommand{\cY}{\mathcal Y}

\providecommand{\cZ}{\mathcal Z}

\providecommand{\Xbar}{\bar{X}}

\providecommand{\1}{\mathbf 1}

\providecommand{\grad}{\nabla}

\providecommand{\Xbar}{\overline{X}}

\providecommand{\qbar}{\overline{q}}

\providecommand{\pbar}{\overline{p}}

\providecommand{\occupation}{\mathrm{P}}

\newtheorem{theorem}{Theorem}[section]
\newtheorem{definition}{Definition}[section]

\newtheorem{remark}[theorem]{Remark}
\newtheorem{lemma}[theorem]{Lemma}

\newtheorem{condition}{Condition}[section]

\begin{document}

\title{Hypoelliptic multiscale Langevin diffusions: Large deviations, invariant measures and small mass asymptotics}

\author{Wenqing Hu}
\address{School of Mathematics\\
University of Minnesota, Twin Cities \\ 302 Vincent Hall,
Minneapolis, MN 55414.} \email{huxxx758@umn.edu }
\author{Konstantinos Spiliopoulos}
\thanks{Research of K.S. supported in part by the National Science Foundation (DMS 1312124 and DMS 1550918).}
\address{Department of Mathematics and Statistics\\
Boston University\\
111 Cummington Mall, Boston MA 02215.}
\email{kspiliop@math.bu.edu}

\date{\today.}

\maketitle

\begin{abstract}
We consider a general class of non-gradient hypoelliptic Langevin diffusions and study two related questions. The first one is large deviations for hypoelliptic multiscale diffusions. The second one is small mass asymptotics of the invariant measure corresponding to hypoelliptic Langevin operators and of related hypoelliptic Poisson equations. The invariant measure corresponding to the hypoelliptic problem and appropriate hypoelliptic Poisson equations enter the large deviations rate function due to the multiscale effects. Based on the small mass asymptotics we derive that the large deviations behavior of the multiscale hypoelliptic diffusion is consistent with the large deviations behavior of its overdamped counterpart. Additionally, we rigorously obtain an asymptotic expansion of the solution to the related density of the invariant measure and to hypoelliptic Poisson equations with respect to the mass parameter, characterizing the order of convergence. The proof of convergence of invariant measures is of independent interest, as it involves an improvement of the hypocoercivity result for the kinetic Fokker-Planck equation. We do not restrict attention to gradient drifts and our proof provides explicit information on the dependence of the bounds of interest in terms of the mass parameter.
\end{abstract}



\section{Introduction}

The second order Langevin equation
\begin{equation*}
\tau \ddot{q}_t=f(q_t)-\lb
\dot{q}_t+\sm(q_t)\dot{W}_t \ , \ q_0=q \in \R^n \ , \
\dot{q}_0=p\in \R^n \ ,
\end{equation*}
is one of the most classical equations in
probability theory as well as in mathematical physics
(\cite{Langevin, Einstein, Smoluchowski}). It describes, under Newton's law, the motion of a particle of mass $\tau$ in a force field
$f(q)$, $q\in \R^n$, subject to random fluctuations and to a
friction  proportional to the velocity. Here $W_t$ is the standard
Wiener process (Brownian motion) in $\R^n$, $\lb>0$ is the friction
coefficient. 

In this paper we are interested in the case where the force field $f(q)$ has multiscale structure and the magnitude of the random fluctuations are small when allowing for inhomogeneous friction coefficient. In particular, our starting object of interest is the second order
hypoelliptic multiscale Langevin equation,
\begin{equation}
\tau\ddot{q}^{\varepsilon}_{t}=\left[\frac{\varepsilon}{\delta}b\left(q^{\varepsilon}_{t},
\frac{q^{\varepsilon}_{t}}{\delta}\right)+c\left(q^{\varepsilon}_{t},\frac{q^{\varepsilon}_{t}}
{\delta}\right)-\lambda \left(q^{\varepsilon}_{t}\right) \dot{q}^{\varepsilon}_{t}\right]dt+
\sqrt{\varepsilon}\sigma\left(q^{\varepsilon}_{t},\frac{q^{\varepsilon}_{t}}{\delta}\right)\dot{W}_{t}
\label{Eq:SecondOrderLangevin}
 \ ,
\end{equation}
where $\varepsilon,\delta\ll 1$ and
$\delta=\delta(\varepsilon)\downarrow 0$ as $\varepsilon\downarrow 0$.
Here, $\lambda(q)>0$ is an inhomogeneous friction coefficient. Moreover, $\varepsilon$ represents the strength of the noise, whereas $\delta$ is the parameter that separates the scales.

It is well known that when $\tau\downarrow 0$, the solution to
(\ref{Eq:SecondOrderLangevin}) approximates that of a first order equation.
In particular, if $\lambda$ is a constant, then
in the overdamped case, i.e. when $\tau$ is small, the motion can be approximated
by the first order Langevin equation (see for example \cite{Freidlin SK approximation})
\begin{equation}
\dot{\tilde{q}}^{\varepsilon}_{t}=\frac{1}{\lambda}\left[\frac{\varepsilon}{\delta}
b\left(\tilde{q}^{\varepsilon}_{t},\frac{\tilde{q}^{\varepsilon}_{t}}{\delta}\right)
+c\left(\tilde{q}^{\varepsilon}_{t},\frac{\tilde{q}^{\varepsilon}_{t}}{\delta}\right)\right]
+\sqrt{\varepsilon}\frac{\sigma\left(q^{\varepsilon}_{t},
\frac{q^{\varepsilon}_{t}}{\delta}\right)}{\lambda}\dot{W}_{t}
\label{Eq:FirstOrderLangevin}
 \ .
\end{equation}

The situation is much more complex in the case that the friction coefficient
depends on the position too, see \cite{HottovyMcDanielVopeWehr2014,HuFreidlin2011}.
In particular, in the setting of (\ref{Eq:SecondOrderLangevin}),
the motion of $q^{\varepsilon}$ as $\tau\downarrow 0$ is approximated by
\begin{equation}
\dot{\tilde{q}}^{\varepsilon}_{t}=\left[\frac{\varepsilon}{\delta}
\frac{b\left(\tilde{q}^{\varepsilon}_{t},\frac{\tilde{q}^{\varepsilon}_{t}}{\delta}\right)}
{\lambda(\tilde{q}^{\varepsilon}_{t})}
+\frac{c\left(\tilde{q}^{\varepsilon}_{t},\frac{\tilde{q}^{\varepsilon}_{t}}{\delta}\right)}
{\lambda(\tilde{q}^{\varepsilon}_{t})}-
\varepsilon\frac{\nabla \lambda(\tilde{q}^{\varepsilon}_{t})}{2\lambda^{3}(\tilde{q}^{\varepsilon}_{t})}
\alpha \left(\tilde{q}^{\varepsilon}_{t},\frac{\tilde{q}^{\varepsilon}_{t}}{\delta}\right)\right]
+\sqrt{\varepsilon}\frac{\sigma\left(q^{\varepsilon}_{t},\frac{q^{\varepsilon}_{t}}{\delta}\right)}
{\lambda(\tilde{q}^{\varepsilon}_{t})}\dot{W}_{t}
\label{Eq:FirstOrderLangevin2}
 \ ,
\end{equation}
where $\alpha(q,r)=\sigma(q,r)\sigma^{T}(q,r)$. Clearly, when $\lambda(q)=\lambda=\text{constant}$,
 (\ref{Eq:FirstOrderLangevin2}) reduces to (\ref{Eq:FirstOrderLangevin}).

The first goal of this paper is to consider, in the homogenization regime where $\frac{\varepsilon}{\delta}\rightarrow \infty$
as $\varepsilon,\delta\downarrow 0$, the large deviations behavior of
the solution to (\ref{Eq:SecondOrderLangevin}) $q^{\varepsilon}$ in
such a way that, when the mass is small,   it is consistent with the large deviations behavior
of the solution to the overdamped counterpart
(\ref{Eq:FirstOrderLangevin2}), or equivalently (\ref{Eq:FirstOrderLangevin}). In particular we want to investigate the conditions under which the tail behavior of (\ref{Eq:SecondOrderLangevin}) and of (\ref{Eq:FirstOrderLangevin2}) agree, at least in a limiting sense.

It turns out that we get interesting
non-trivial behavior when the mass $\tau$ relates to
$\varepsilon,\delta$ in a specific way that will be explained in the
sequel. For this reason we shall write $\tau^{\varepsilon}$ in place
of $\tau$ when we want to emphasize this dependence. We prove that
if the mass of the particle $\tau$ scales appropriately with the order of the fluctuations and in particular if it is of  order
$\delta^{2}/\varepsilon$, i.e., if
$\tau=m\frac{\delta^{2}}{\varepsilon}$ with $m$ small but positive,
then the  large deviation behaviors of the overdamped and underdamped systems agree. The large deviations result for (\ref{Eq:SecondOrderLangevin}) is given in Theorem \ref{T:LDPexplcitRepresentation} and the agreement in terms of the large deviations behavior of (\ref{Eq:SecondOrderLangevin}) and (\ref{Eq:FirstOrderLangevin2}) is given in Theorem \ref{T:MainTheorem4}.

In order to derive the large deviations principle we follow the weak convergence
approach, \cite{DupuisEllis,DupuisSpiliopoulos}.
This framework transforms the large deviations problem to convergence of a hypoelliptic
stochastic control problem. Due to the hypoellipticity one needs certain
a-priori bounds that establish compactness, see \cite{HairerPavliotis2004}. We obtain an explicit form of the control (equivalently change of measure) that leads to the proof of large deviations upper bound in the multiscale hypoelliptic case. Even though we do not address this issue in the current paper, we mention that the explicit information on the optimal control can be used for the construction of provably-efficient Monte Carlo schemes in the spirit of the constructions done in \cite{DupuisSpiliopoulosWang,Spiliopoulos2014b} for the corresponding elliptic case.

Under the parameterizations $\tau=m\frac{\delta^{2}}{\varepsilon}$
and when $\delta\ll \varepsilon$ we derive the large deviations
principle for $\{q^{\varepsilon},\varepsilon>0\}$, where
$q^{\varepsilon}$ solves (\ref{Eq:SecondOrderLangevin}), see Theorem \ref{T:LDPexplcitRepresentation}. The large
deviations rate function is derived in closed form and it depends on
$m$. The next natural question is to derive that as $m\downarrow 0$
the large deviations rate function converges to that of the large
deviations principle for the overdamped case, i.e., for the solution to
(\ref{Eq:FirstOrderLangevin2}). However, to our surprise, we find
that even in the case of constant diffusion  the proof of such a
convergence is highly involved. We prove such a convergence in the
special case of  diffusion coefficient
$\sm(q,r)=\sqrt{2\beta\lambda(q)} I$, $\beta>0$ (which is the
parametrization of the fluctuation-dissipation theorem) and we
include a discussion for the general variable diffusion coefficient
case in Remark \ref{R:GeneralDiffusionCoeff}. 
This result
supports the claim that the large
deviations behavior of the multiscale second order Langevin diffusion (\ref{Eq:SecondOrderLangevin}) and of  its first
order counterpart (\ref{Eq:FirstOrderLangevin2}) agree, see Theorem \ref{T:MainTheorem4}. 

The second and related goal of this paper is to rigorously develop small mass asymptotics  for the invariant measure, see Theorem \ref{T:ConvergenceInvariantMeasures} and for certain Poisson equations, see Theorem \ref{T:ConvergenceCellProblems}, that appear in the rate function of
the large deviation principle (see Theorem \ref{T:LDPexplcitRepresentation}) due to the homogenization effects. Our proof of the convergence as $m\da 0$ of the large deviation rate
function requires a thorough analysis of the small mass asymptotic
for the invariant measure of the fast motion corresponding to (\ref{Eq:SecondOrderLangevin}). In
particular, since we will allow the drift term $b(q,r)$
to be a general vector field rather than a gradient field, our proof
of the convergence involves a non-trivial improvement of the hypocoercivity
result for linear Fokker--Planck equation (\cite[Section
1.7]{Villani2006}, see also \cite{DolbeaultMouhotSchmseiser2015}).  If $b(q,r)$ is not a gradient field, then certain operators that appear in the analysis are not anti-symmetric. This implies that extra terms
appear that need to be appropriately handled. Then making use and extending the hypocoercivity results of \cite{Villani2006},
we prove that the invariant measures corresponding to the $m>0$ case,
converges in $L^{2}$ to the invariant measure corresponding to the $m=0$ problem. Here we make use of the $((\cdot,\cdot))$ inner product introduced in \cite{Villani2006} and we combine the different terms in such a way that the desired bounds follow. To accomplish this goal in the general non-gradient case, we use the structure of the hypoelliptic operator in an effective way.

Using the convergence of the invariant measure  and Poincar\'{e} inequality, we also prove that the solutions to related Poisson  equations
(the so-called ``cell problems") that appear due to the homogenization effects of the drift $b(q,r)$, also converge in the appropriate $L^{2}$ sense. In addition to that, the proof provides a rigorous justification
of the corresponding multiscale expansion of the solutions of the corresponding equations in powers of $\sqrt{m}$.  Related heuristic, i.e., without proof, asymptotic expansions can be
also found in \cite{PavliotisStuart}.
We would like to emphasize that our method of proof allows to obtain  upper bounds for the norms of interest with detailed dependence on the parameters of interest, such as the mass of
the particle.

Partial motivation for our work comes from  chemical physics and biology, and in
particular from the dynamical behavior of proteins such as their folding and
binding kinetics. As it has been suggested long time ago (e.g., \cite{LifsonJackson, Zwanzig})
the potential surface of a protein might have a hierarchical structure with
potential minima within potential minima. 
As a consequence, the roughness of the energy
landscapes that describe proteins has numerous effects on their kinetic properties as well as on their behavior at equilibrium.

One of the first papers that used a simple model with two separated time scales to model diffusion in rough potentials is \cite{Zwanzig}. The situation usually investigated \cite{LifsonJackson,Zwanzig,DupuisSpiliopoulosWang2}
is  based on the first order equation (\ref{Eq:FirstOrderLangevin}) even though the physical model and what is many times used in molecular simulations is the more complex second order
Langevin equation that involves both position and velocity, see for example \cite{LelievreStoltzRousset2010}, and would also usually include more than two separated time scales. The usual  choice of coefficients is $\lambda(q)=\text{constant}$,
$b(q,q/\delta)=-\frac{2\beta}{k_{\beta}T}\nabla Q(q/\delta)$ ,
$c(q,q/\delta)=-\frac{2\beta}{k_{\beta}T}\nabla V(q)$ and
$\sigma(q,q/\delta)=\sqrt{2\beta\lambda} I$, where $k_{\beta}$ is the
Boltzmann constant and $T$ is the
temperature, in such a way that the fluctuations-dissipations theorem holds. We remark here that our formulation for the large deviations result is general and includes the parametrization suggested by the fluctuation-dissipation theorem as a special case. Notice that the choice of the separable drift
\begin{equation}
b(q,q/\delta)=-\nabla Q(q/\delta), \qquad  c(q,q/\delta)=-\nabla
V(q) \nonumber
\end{equation}
represents the motion of a massless particle in a rough potential
$\varepsilon Q(q/\delta)+ V(q)$. In particular, the model of interest
 in this case becomes
\begin{equation}
\dot{\tilde{q}}^{\varepsilon}_{t}=-\frac{\varepsilon}{\delta}
\frac{2\beta}{k_{\beta}T}\nabla Q\left(\frac{\tilde{q}^{\varepsilon}_{t}}{\delta}\right)
-\frac{2\beta}{k_{\beta}T}\nabla V\left(\tilde{q}^{\varepsilon}_{t}\right)+
\sqrt{\varepsilon}\sqrt{2\beta}\dot{W}_{t}
\label{Eq:OverdampedLangevinWithSeparableDrift}
 \ .
\end{equation}

The questions of interest in \cite{Zwanzig,DupuisSpiliopoulosWang2} are related to the
effect of taking $\delta\downarrow0$ with $\varepsilon$ small but fixed. This is
almost the same to requiring that $\delta$ goes to $0$ much faster than
$\varepsilon$ does, which is the regime that we study in this paper.

The related mathematical literature is quite rich. For the related hypocoercivity theory the reader is referred to \cite{Villani2006}.
For the case $\delta=1$, the large deviations principle of the solutions to (\ref{Eq:SecondOrderLangevin}) and (\ref{Eq:FirstOrderLangevin})
as $\varepsilon\downarrow 0$ is being compared in \cite{ChenFreidlin2005}. For the case $\varepsilon=1$, periodic homogenization for a special case of (\ref{Eq:SecondOrderLangevin})
(in particular when $c(q,r)=0$ and $b(q,r)=b(r)$) has been addressed in \cite{HairerPavliotis2004}. Also, when $\varepsilon=1$ random homogenization for
(\ref{Eq:SecondOrderLangevin}) when $c(q,r)=0$ and the special case of gradient drift $b(q,r)=-\nabla Q(r)$ has been addressed in \cite{Benabou2006, PapanicolaouVaradhan1985}. More is known about the overdamped
case (\ref{Eq:FirstOrderLangevin}), see \cite{DupuisSpiliopoulos, KosyginaRezakhanlouVaradhan, LionsSouganidis2006, Spiliopoulos2014a} where homogenization and large deviation results for
the solution to equations of the form (\ref{Eq:FirstOrderLangevin}) are obtained under different relations between $\varepsilon$ and $\delta$, in both periodic and random environments.

The rest of the paper is structured as follows. In Section \ref{S:MainResults} we formulate the problem, our assumptions and the main results of this paper in detail.
In Sections \ref{S:LLN}-\ref{S:LDP} we prove the large deviations principle for the hypoelliptic problem. In Sections \ref{S:RateFunctions}-\ref{S:CellProblems} and in the Appendix we exploit the small mass asymptotics.

In particular, using the weak convergence approach we turn the large
deviations principle into a law of large numbers for a stochastic control problem. Section \ref{S:LLN} proves the convergence of the controlled stochastic equation and Section \ref{S:LDP} proves the convergence of the cost functional, which is the Laplace principle.
 In Section \ref{S:RateFunctions} we prove the
small mass limit of the rate function in the  diffusion $\sigma(q,r)=\sqrt{2\beta\lambda(q)} I$ case, using the convergence of the invariant measures as $m\rightarrow 0$ (Section \ref{S:InvariantMeasures}) and of the related ``cell problems" that are auxiliary Poisson equations that appear in the rate functions due to homogenization effects (Section \ref{S:CellProblems}). We emphasize that  Section \ref{S:InvariantMeasures} is of
independent interest as it is an extension of the hypo--coercivity
result for the linear kinetic Fokker--Planck equation \cite[Section
1.7]{Villani2006}, since we do not restrict our attention to drifts that are of gradient form. The method of proofs also yields explicit decay rates of the norms of interest
with regards to parameters of interest such as the mass of the particle. Most of the proofs to technical lemmas are deferred to the Appendix.

\section{Problem formulation, assumptions and main results}\label{S:MainResults}

In this section, we formulate more precisely the problem that we are studying in this paper, we state our main assumptions and our main results.
In preparation for stating the main results, we recall the concept of a
Laplace principle.
\begin{definition}
\label{Def:LaplacePrinciple} Let $\{q^{\varepsilon},\varepsilon>0\}$ be a family of
random variables taking values on a Polish space $\mathcal{S}$ and let $I$ be a rate function
on $\mathcal{S}$. We say that $\{q^{\varepsilon},\varepsilon>0\}$ satisfies the
Laplace principle with rate function $I$ if for every bounded and continuous
function $h:\mathcal{S}\rightarrow\mathbb{R}$
\begin{equation*}
\lim_{\varepsilon\downarrow0}-\varepsilon\ln\E\left[  \exp\left\{
-\frac{h(q^{\varepsilon})}{\varepsilon}\right\}  \right]  =\inf_{x\in\mathcal{S}%
}\left[  I(x)+h(x)\right]  . \label{Eq:LaplacePrinciple}%
\end{equation*}
\end{definition}

If the rate function has compact level sets, then the Laplace
principle is equivalent to the corresponding large deviations
principle with the same rate function (see Theorems 2.2.1 and 2.2.3
in \cite{DupuisEllis}). Hence, instead of proving a large deviations
principle for $\{q^{\varepsilon}\}$ we prove a Laplace principle for
$\{q^{\varepsilon}\}$.

Our main regularity assumption in regards to the coefficients of (\ref{Eq:SecondOrderLangevin}) is given by Condition \ref{A:Assumption1}.
\begin{condition}
\label{A:Assumption1}
The functions $b(q,r),c(q,r),\sigma(q,r)$ are
\begin{enumerate}
\item   periodic with period $1$ in the second variable
in each direction, and
\item  $C^{1}(\mathbb{R}^{d})$ in $r$ and $C^{2}(\mathbb{R}^{d})$ in $q$ with all
partial derivatives continuous and globally bounded in $q$ and $r $.
\end{enumerate}
The diffusion matrix $\alpha(q,r)=\sigma(q,r)\sigma^{T}(q,r)$ is uniformly non-degenerate.
There exist constants $0<\underline{\lambda}<\overline{\lambda}$ such that for every
$q\in\mathbb{R}^{d}$,  $\underline{\lambda}<\lambda(q)<\overline{\lambda}$.
Moreover, the function $\lambda(q)$ is in $C^{1}(\mathbb{R}^{d})$ with bounded partial derivatives.
\end{condition}

Using the parametrization $\tau=m\frac{\delta^{2}}{\varepsilon}$, the system being considered is
\begin{align}
m\dfrac{\dt^2}{\ve}\ddot{q}_t^\ve&=\left[\dfrac{\ve}{\dt}
b\left(q_t^\ve,\dfrac{q_t^\ve}{\dt}\right)+c\left(q_t^\ve,\dfrac{q_t^\ve}{\dt}\right)-\lambda(q_t^\ve)\dot{q}_t^\ve\right]
+\sqrt{\ve}\sigma\left(q_t^\ve,\dfrac{q_t^\ve}{\dt}\right)\dot{W}_t  \ .\label{Eq:OverdampedSystem} 
\end{align}

Setting $p^{\varepsilon}_{t}=\sqrt{m}\frac{\delta}{\varepsilon}\dot{q}^{\varepsilon}_{t}$
we obtain the following system of equations which we also supplement with initial conditions
\begin{align}
\dot{q}_t^\ve&=\dfrac{1}{\sqrt{m}}\dfrac{\ve}{\dt}p_t^\ve \ ,\label{Eq:TransformedSecondOrderSystem}\\
\dot{p}_t^\ve&=\dfrac{1}{\sqrt{m}}\dfrac{1}{\dt}\left[\dfrac{\ve}{\dt}
b\left(q_t^\ve,\dfrac{q_t^\ve}{\dt}\right)+c\left(q_t^\ve,\dfrac{q_t^\ve}{\dt}\right)\right]
-\dfrac{\lambda(q_t^\ve)}{m}\dfrac{\ve}{\dt^2}p_t^\ve+
\dfrac{\sqrt{\ve}}{\dt}\frac{\sigma\left(q_t^\ve,\dfrac{q_t^\ve}{\dt}\right)}{\sqrt{m}}\dot{W}_t
\ ,\nonumber\\
q_0^\ve&=q_o\in \R^d \ , \ p_0=p_o\in \R^d \ .\nonumber
\end{align}

Condition \ref{A:Assumption1}, guarantees that (\ref{Eq:OverdampedSystem}) and (\ref{Eq:TransformedSecondOrderSystem}), have a unique strong solution; this is a classical result, see for example \cite{Freidlin SK approximation} or Theorem 5.2.1 of \cite{Oksendal}. The infinitesimal generator for the $(q,p)$ process satisfying (\ref{Eq:TransformedSecondOrderSystem}) is given by
\begin{align*}
\cL &=\dfrac{1}{\sqrt{m}}\left[\frac{\epsilon}{\delta}p \cdot \grad_q +\frac{\epsilon}{\delta^{2}}b(q,q/\delta)\cdot \grad_p+\frac{1}{\delta}c(q,q/\delta)\cdot \grad_p\right]\nonumber\\
&\qquad\qquad +\frac{1}{m}\frac{\epsilon}{\delta^{2}}\left[-\lb(q) p\cdot \grad_p
+\frac{1}{2}\alpha(q,q/\delta):\nabla^{2}_p\right] \ ,
\end{align*}
where we recall that $\alpha(q,r)=\sigma(q,r)\sigma^{T}(q,r)$.

We can assume that $p_{o}$ is a random variable, as long as it is
independent of the driving Wiener process $W_{t}$ and as long as
$\E\left[e^{\frac{1}{2}|\sigma_{\text{max}}^{-1}p_{o}|^{2}}\right]<\infty$
(see Appendix \ref{S:AppendixA}), where we have defined
$\sigma_{\text{max}}=\max\li_{i,j=1,\cdots
d}\sup\li_{(q,r)}\left|\sigma_{i,j}(q,r)\right|$.

Sometimes, we may write $X_t^\ve=(q_t^\ve, p_t^\ve)$. Let
$|\bullet|$ be the Euclidean norm in $\R^d$. We introduce the
control set
$$\cA=\left\{u=\{u_s\in \R^d: 0\leq s \leq T\} \text{ progressively }
\cF_s \text{--measurable and }
\E\play{\int_0^T|u_s|^2ds<\infty}\right\} \ .$$

The result in \cite{BoueDupuis} gives the following representation
$$-\ve\ln\E_{q_0}\left[\exp\left(-\dfrac{h(q_\bullet^\ve)}{\ve}\right)\right]
=\inf\li_{u\in \cA}\E_{q_0}\left[\dfrac{1}{2}\play{\int_0^T
|u_s|^2ds+h(\qbar_\bullet^\ve)}\right] \ . $$

Here the process $\qbar_t^\ve$ is the $q$--component of the
hypoelliptic controlled diffusion process $\Xbar_t^\ve=(\qbar_t^\ve,
\pbar_t^\ve)$:
\begin{align}
\dot{\qbar}_t^\ve&=\dfrac{1}{\sqrt{m}}\dfrac{\ve}{\dt}\pbar_t^\ve \ ,
\label{Eq:SecondOrderControlSystem}\\
\dot{\pbar}_t^\ve&=\dfrac{1}{\sqrt{m}}\dfrac{1}{\dt}\left[\dfrac{\ve}{\dt}
b\left(\qbar_t^\ve,\dfrac{\qbar_t^\ve}{\dt}\right)+c\left(\qbar_t^\ve,\dfrac{\qbar_t^\ve}{\dt}\right)\right]
-\dfrac{\lb(q_t^\ve)}{m}\dfrac{\ve}{\dt^2}\pbar_t^\ve+\dfrac{1}{\dt}\frac{\sigma\left(q_t^\ve,\dfrac{q_t^\ve}{\dt}\right)}{\sqrt{m}}u_t\nonumber\\
&\quad+\dfrac{\sqrt{\ve}}{\dt}\frac{\sigma\left(q_t^\ve,\dfrac{q_t^\ve}{\dt}\right)}{\sqrt{m}}\dot{W}_t
\ ,\nonumber\\
\qbar_0^\ve&=q_o\in \R^d \ , \ \pbar_0^\ve =p_o\in \R^d \ . \nonumber
\end{align}

Let $u^\ve_\bullet\in \cA$ and $\Xbar_s^\ve$ solves (\ref{Eq:SecondOrderControlSystem}) with $u^\ve$
in place of $u$. Let the control space be $\cZ=\R^d$, the fast variable
space be $\cY=\R^d\times \T^d$. We see that the fast variable is
actually $\left(\pbar_s^\ve, \dfrac{\qbar_s^\ve}{\dt}\right)$. Let us define the operator
\begin{align*}
\cL^{m}_q\Phi(p,r)&=\dfrac{1}{\sqrt{m}}\left[p \cdot \grad_r\Phi(p,r)+b(q,r)\cdot \grad_p\Phi(p,r)\right]\nonumber\\
&\qquad +\frac{1}{m}\left[-\lb(q) p\cdot \grad_p
\Phi(p,r)+\frac{1}{2}\alpha(q,r):\nabla^{2}_p\Phi(p,r)\right] \ .
\end{align*}
For each fixed $q$, the operator $\cL^{m}_q$ defines a hypoelliptic
diffusion process on $(p,r)\in\cY=\R^d\times \T^d$. Let $\mu(dpdr|q)$
be the  unique invariant measure for this process. Notice that $\cL^{m}_q$ is effectively the operator corresponding to the fast motion. The following centering condition is essential for the validity of  the results.
\begin{condition}
\label{A:Assumption2}
We assume that for every $q\in\mathbb{R}^{d}$
\begin{equation*}
\int_\cY b(q,r)\mu(dpdr|q)=0.
\end{equation*}
\end{condition}

Let us consider the preliminary cell problem
\begin{align}
\cL^{m}_q \Phi(p,r)&=-\dfrac{1}{\sqrt{m}}p \ , \qquad
\int_{\mathcal{Y}}\Phi(p,r)\mu(drdp|q)=0 \ .
\label{Eq:HypoellipticCellProblem}
\end{align}

It is clear that the solution to (\ref{Eq:HypoellipticCellProblem})
$\Phi$ depends also on $q$, but we sometimes suppress this in the
notation for convenience. By the work of \cite{HairerPavliotis2004},
we know that under Condition \ref{A:Assumption2}, the PDE
(\ref{Eq:HypoellipticCellProblem}) has a unique, smooth solution
that does not grow too fast at infinity, see Appendix
\ref{S:AppendixA} for more details.  Note that the function $\Phi$
is actually a vector valued function
$\Phi(p,r)=(\Phi_1(p,r),...,\Phi_d(p,r))$.

Then our first main result reads as follows.
\begin{theorem}
\label{T:LDPexplcitRepresentation} Let $\{(q^{\varepsilon},
p^{\varepsilon}),\varepsilon>0\}$ be  the unique  solution to
(\ref{Eq:TransformedSecondOrderSystem}). Under Conditions
\ref{A:Assumption1} and \ref{A:Assumption2},
$\{q^{\varepsilon},\varepsilon>0\}$ satisfies the large deviations
principle with rate function
\begin{equation*}
S_{m}(\phi)=\left\{
\begin{array}{ll}
\play{\dfrac{1}{2}\int_0^T
(\dot{\phi}_s-r_{m}(\phi_s))^TQ^{-1}_{m}(\phi_s)(\dot{\phi}_s-r_{m}(\phi_s))ds}&
\text{ if } \phi\in \mathcal{AC}([0,T]; \R^d),\phi_{0}=q_{o}\\
+\infty & \text{ otherwise \ .}
\end{array}\right.
\label{Eq:ActionFunctional1}%
\end{equation*}
where
$$r_{m}(q)=\dfrac{1}{\sqrt{m}}\int_{\cY}\grad_p\Phi(p,r)c(q,r)\mu(dpdr|q) \ ,$$
$$Q_{m}(q)=\frac{1}{m}\int_{\cY}\grad_p\Phi(p,r)\alpha(q,r)(\grad_p\Phi(p,r))^{T}\mu(dpdr|q) \ .$$
\end{theorem}

To support the claim that the particular parametrization is consistent
 with the large deviations principle of the overdamped case (\ref{Eq:FirstOrderLangevin2}),
 we need to prove
that $\lim\li_{m\rightarrow 0}S_{m}(\phi)=S_{0}(\phi)$, where
$S_{0}(\phi)$ is the rate function associated to
(\ref{Eq:FirstOrderLangevin2}). To that end, we recall the
corresponding large deviations result from
\cite{DupuisSpiliopoulos}.

Let $\mu_{0}(dr|q)$ be the unique invariant measure
corresponding to the operator
\begin{equation*}
\mathcal{L}_{q}^{0}=\frac{1}{\lambda(q)}b(q,r)\cdot\nabla_{r}+\frac{1}{2\lambda(q)}\alpha(q,r): \nabla^{2}_{r} \label{OperatorRegime1}%
\end{equation*}
equipped with periodic boundary conditions in $r$ ($q$ is being treated as a
parameter here). By Theorem \ref{T:ConvergenceInvariantMeasures}, Condition \ref{A:Assumption2} implies that
the following centering condition for the drift term $b$:
\[
\int_{\bar{\mathcal{Y}}}b(q,r)\mu_{0}(dr|q)=0,
\]
where $\bar{\mathcal{Y}}=\mathbb{T}^{d}$ denotes the $d$-dimensional torus. Under this centering condition,
 the cell problem
\begin{equation}\cL^{0}_q\chi_\ell(q,r)=-\dfrac{1}{\lb(q)}b_\ell(q,r) \ , \
\int_{\bar{\cY}}\chi_\ell(q,r)\mu_{0}(dr|q)=0 \ , \ell=1,2,...,d \ .
\label{Eq:CellProblemOverdampedCase}\end{equation} has a unique
bounded and sufficiently smooth solution $\chi=(\chi_1,...,\chi_d)$. After these definitions we
recall the result from \cite{DupuisSpiliopoulos} that will be of use to
us.

\begin{theorem}[Theorem 5.3 in \cite{DupuisSpiliopoulos}]
\label{T:MainTheorem3} Let $\{q^{\varepsilon},\varepsilon>0\}$ be the unique
solution to (\ref{Eq:FirstOrderLangevin}). Under Conditions
\ref{A:Assumption1} and \ref{A:Assumption2}, $\{q^{\varepsilon},\varepsilon>0\}$
satisfies a large deviations principle with rate function
\begin{equation*}
S_{0}(\phi)=%
\begin{cases}
\frac{1}{2}\int_{0}^{T}(\dot{\phi}_{s}-r_{0}(\phi_{s}))^{T}Q^{-1}_{0}(\phi_{s}%
)(\dot{\phi}_{s}-r_{0}(\phi_{s}))ds & \text{if }\phi\in\mathcal{AC}%
([0,T];\mathbb{R}^{d}), \phi_{0}=q_{o}\\
+\infty & \text{otherwise.}%
\end{cases}
\label{ActionFunctional1}%
\end{equation*}
where
\[
 r_{0}(q)=\frac{1}{\lambda(q)}\int_{\bar{\mathcal{Y}}}\left((I+\frac{\partial\chi}
 {\partial r}(q,r))\right)c(q,r)\mu_{0}(dr|q)
\]
and
\[
 Q_{0}(q)=\frac{1}{\lambda^{2}(q)}\int_{\bar{\mathcal{Y}}}\left(I+\frac{\partial\chi}
 {\partial r}(q,r)\right)\alpha(q,r)\left(I+\frac{\partial\chi}{\partial r}(q,r)\right)^{T}\mu_{0}(dr|q).
\]
\end{theorem}

In order now to show that $\lim\li_{m\rightarrow 0}S_{m}(\phi)=S_{0}(\phi)$, we need to study the limiting begavior of $\mu(dpdr|q)$ and of $\nabla_{p}\Phi(p,r)$  as $m\rightarrow 0$.
For this purpose, let us assume that  $\sm(q,r)=\sqrt{2\beta\lambda(q)}I, \beta>0$, i.e., we assume that the noise is such that we are in fluctuation-dissipation balance.  In this case, for a function  $f\in\mathcal{C}^{2}(\mathcal{Y})$, we have
\begin{align*}
\cL^{m}_q f(p,r)&=\frac{\lb(q)}{m}\cA f(p,r)+\dfrac{1}{\sqrt{m}}\cB f(p,r)  \ ,
\end{align*}
where $\cA f= -p\cdot \grad_p
f+\beta \Delta_p f$ and $\cB f(p,r)=p \cdot \grad_r f+b(q,r)\cdot \grad_p f$. Likewise, we have
\begin{equation*}
\mathcal{L}_{q}^{0}f(r)=\frac{1}{\lambda(q)}b(q,r)\cdot\nabla_{r}f(r)+\beta \Delta_{r}f(r) \ .
\end{equation*}

We denote by
$\mu(dpdr|q)=\rho^m(p,r|q)dpdr$ the invariant measure corresponding to the operator $\cL^{m}_{q}$. Also, let us write $\mu_{0}(dr|q)=\rho_0(r|q)dr$ for the invariant measure corresponding to the operator $\mathcal{L}_{q}^{0}$.

Let us also define $\pi(dp)=\rho^{\text{OU}}(p)dp$ to be the invariant measure on $\R^d$ for the Ornstein--Uhlenbeck process with generator $\cA$. With this notation, let us write
$\rho^m(p,r)=\widetilde{\rho}^m(p,r) \rho^0(p,r)$, where $\rho^0(p,r)=\rho^{\text{OU}}(p)\rho_0(r)$, suppressing the dependence on $q$.

Then, in  Sections \ref{S:InvariantMeasures} and \ref{S:CellProblems} respectively we prove the following Theorems which constitute the second main result of our paper.
\begin{theorem}\label{T:ConvergenceInvariantMeasures}
Let Condition \ref{A:Assumption1} hold and assume that $\sm(q,r)=\sqrt{2\beta\lambda(q)} I, \beta>0$. Then, for every $q\in\mathbb{R}^{d}$, we have
\begin{align*}
\lim_{m\rightarrow0}\left\|\tilde{\rho}^{m}(p,r)-1\right\|_{L^{2}(\mathcal{Y};
\rho^{0})}&=0 \ .
\end{align*}
\end{theorem}

\begin{theorem}\label{T:ConvergenceCellProblems}
Let Conditions \ref{A:Assumption1} and \ref{A:Assumption2} hold and assume that $\sm(q,r)=\sqrt{2\beta\lambda(q)}I, \beta>0$. Then, for every $q\in\mathbb{R}^{d}$, we have
\begin{align*}
\lim_{m\rightarrow0}\left\|\frac{1}{\sqrt{m}}\nabla_{p}\Phi-\frac{1}{\lambda(q)}(I+\nabla_{r}\chi)\right\|_{L^{2}(\mathcal{Y};
\rho^{0})}&=0 \ .
\end{align*}
\end{theorem}

Using then Theorems \ref{T:ConvergenceInvariantMeasures} and \ref{T:ConvergenceCellProblems} we prove in Section \ref{S:RateFunctions} that the
rate function $S_{m}(\phi)$ converges $S_{0}(\phi)$,  as $m\downarrow 0$.
\begin{theorem}
\label{T:MainTheorem4} Let Conditions \ref{A:Assumption1} and \ref{A:Assumption2} hold and assume that $\sm(q,r)=\sqrt{2\beta\lambda(q)}I, \beta>0$. Then, we have $$\lim\li_{m\ra 0}S_m(\phi)=S_0(\phi) \
.$$
\end{theorem}

\begin{remark}\label{R:GeneralDiffusionCoeff}
We believe that Theorems \ref{T:ConvergenceInvariantMeasures} and \ref{T:ConvergenceCellProblems} and as a consequence Theorem \ref{T:MainTheorem4} are true under
more general variable diffusion coefficients as long as Condition \ref{A:Assumption1} holds. When, the diffusion coefficient $\sigma$ is not a
multiple of the identity matrix, then the operator $\mathcal{A}$ is not the classical Ornstein-Uhlenbeck that has the Gaussian measure $\rho^{\text{OU}}(p)dp\sim e^{-\frac{|p|^{2}}{2\beta}}dp$ as its invariant measure.
Some of our technical lemmas use this explicit structure  in order to derive the necessary estimates.
However, since the spirit of the proof does not rely on this  structure, we believe that this is only a technical problem.
\end{remark}

\section{Law of large numbers}\label{S:LLN}
In this section we study the limiting behavior of the solution to the control
problem (\ref{Eq:SecondOrderControlSystem}). It turns out that  we need
to consider the solution to (\ref{Eq:SecondOrderControlSystem}) together with
an appropriate occupation measure and then consider the limit of the pair. Let us be more specific now.

Let $u^\ve_\bullet\in \cA$ and $\Xbar_s^\ve$ solves (\ref{Eq:SecondOrderControlSystem}) with $u^\ve$
in place of $u$. Let the control space be $\cZ=\R^d$ and the fast variable
space be $\cY=\R^d\times \T^d$. We see that the fast variable is
actually $\left(\pbar_s^\ve, \dfrac{\qbar_s^\ve}{\dt}\right)$. Let
$A\subseteq \cZ$, $B_1\times B_2\subseteq \cY$ and $\Gm\subseteq
[0,T]$. Let $\Dt=\Dt(\ve)>0$ be a separation of scales parameter. We
introduce the occupation measure
\begin{align}
\occupation^{\ve,\Dt}(A\times B_1\times B_2 \times \Gm)&=
\int_\Gm\left[\dfrac{1}{\Dt}\int_t^{t+\Dt}\1_A(u_s^\ve)\1_{B_1}(\pbar_s^\ve)\1_{B_2}\left(\dfrac{\qbar_s^\ve}{\dt}
\text{ mod } 1\right)ds\right]dt \ .\label{Eq:OccupationMeasure}
\end{align}

Let  us define the function
\begin{align}
\gamma(q, (p,r),z)&=\dfrac{1}{\sqrt{m}}\left[c(q,r)+\sigma(q,r)z\right]\cdot\grad_{p}\Phi(p,r)\ .  \label{Eq:AveragedIntegrant}
\end{align}

Definition \ref{Def:ViablePair} captures the notion of a viable pair as introduced in \cite{DupuisSpiliopoulos} which characterizes the required law of large numbers.
\begin{definition}\label{Def:ViablePair}
A pair $(\psi, \occupation)\in \cC([0,T]; \R^d)\times
\mathcal{P}(\cZ\times \cY\times[0,T])$ will be called viable with
respect to $(\gamma, \cL^{m}_q)$ or simply viable if there is no
confusion, if the following are satisfied. The function $\psi_t$ is
absolutely continuous, $\occupation$ is square integrable in the
sense that
\begin{equation}
\int_{\cZ\times\cY\times[0,T]}|z|^2\occupation(dz, dpdr, ds)<\infty
\ , \label{Eq:ViablePairRequirement1}
\end{equation}
 and
\begin{enumerate}
 \item{
 \begin{equation}
\psi_t=q_o+\int_0^t \int_{\cZ\times \cY}\gamma(\psi_s, (p,r),
z)\occupation(dz,dpdr,ds)  \ ;  \label{Eq:ViablePairRequirement2}
 \end{equation}}
\item{For any $g(p,r)\in \cD(\cL^{m}_q)$,
\begin{equation}
 \int_0^t\int_{\cZ\times \cY}\cL^{m}_{\psi_s}g(p,r)\occupation(dz,dpdr,ds)=0 \ ; \label{Eq:ViablePairRequirement3}
\end{equation}}
\item{\begin{equation}
       \occupation(\cZ\times \cY\times [0,t])=t  \ . \label{Eq:ViablePairRequirement4}
      \end{equation}}
\end{enumerate}
We write $(\psi, \occupation)\in \cV_{(\gamma, \cL_q)}$.
\end{definition}

\begin{theorem}\label{T:LLN}
Consider any family
$\{u^\ve, \ve>0\}$ of controls in $\cA$ satisfying
$$\sup\li_{\ve>0}\E\int_0^T|u_t^\ve|^2dt<\infty \ .$$
Let Conditions \ref{A:Assumption1} and \ref{A:Assumption2} be satisfied. Then the family
$\{(\qbar_\bullet^\ve, \occupation^{\ve,\Dt}), \ve>0\}$ is tight.
Hence, given any subsequence of $\{(\qbar_\bullet^\ve,
\occupation^{\ve,\Dt}), \ve>0\}$, there exists a subsequence that
converges in distribution with limit $(\qbar_\bullet, \occupation)$.
With probability $1$, the accumulation point $(\qbar_{\bullet},
\occupation)$ is a viable pair with respect to $(\gamma, \cL_q)$:
$(\qbar_\bullet, \occupation)\in \cV_{(\gamma, \cL_q)}$.
\end{theorem}
\begin{proof}[Proof of Theorem \ref{T:LLN}]

Part 1. [Tightness]. For a smooth function $g\left(\pbar_t^\ve,
\dfrac{\qbar_t^\ve}{\dt}\right)$ we can apply It\^{o}'s formula and
get
\begin{align}
g\left(\pbar_t^\ve,
\dfrac{\qbar_t^\ve}{\dt}\right)-g\left(\pbar_o^\ve,\dfrac{\qbar_o^\ve}{\dt}\right)
&=\dfrac{\ve}{\dt^2}\int_0^t
\cL^{m}_{\qbar_s^\ve}g\left(\pbar_s^\ve,
\dfrac{\qbar_s^\ve}{\dt}\right)ds\nonumber\\
&+\dfrac{1}{\dt}\dfrac{1}{\sqrt{m}}\int_0^t
\left[c\left(\qbar_s^\ve,\dfrac{\qbar_s^\ve}{\dt}\right)+\sigma\left(\qbar_s^\ve,
\dfrac{\qbar_s^\ve}{\dt}\right)u_s\right]\cdot
\grad_pg\left(\pbar_s^\ve, \dfrac{\qbar_s^\ve}{\dt}\right)ds
\nonumber\\
& +\dfrac{\sqrt{\ve}}{\dt}\frac{1}{\sqrt{m}}\int_0^t \grad_p g\left(\pbar_s^\ve,
\dfrac{\qbar_s^\ve}{\dt}\right)\cdot \sigma\left(\qbar_s^\ve,
\dfrac{\qbar_s^\ve}{\dt}\right) dW_s \ .\label{Eq:ItoFormulaToControlSystem}
\end{align}

Let us apply It\^{o}'s formula to $\Phi\left(\pbar_t^\ve,
\dfrac{\qbar_t^\ve}{\dt}\right)$ in (\ref{Eq:HypoellipticCellProblem}) and we use (\ref{Eq:ItoFormulaToControlSystem}) to get a
representation formula for $\qbar_t^\ve$ as follows:
\begin{align*}
\qbar_t^\ve&=q_o+\int_0^t \dfrac{1}{\sqrt{m}}\left[
c\left(\qbar_s^\ve,\dfrac{\qbar_s^\ve}{\dt}\right)+\sigma\left(\qbar_s^\ve,\dfrac{\qbar_s^\ve}{\dt}\right) u_s\right]\cdot\grad_p\Phi\left(\pbar_s^\ve,
\dfrac{\qbar_s^\ve}{\dt}\right)ds\nonumber\\
& -\dt\left[\Phi\left(\pbar_t^\ve,
\dfrac{\qbar_t^\ve}{\dt}\right)-\Phi\left(\pbar_o^\ve,
\dfrac{\qbar_o^\ve}{\dt}\right)\right]+\frac{\sqrt{\ve}}{\sqrt{m}}\int_0^t
\grad_p \Phi\left(\pbar_s^\ve, \dfrac{\qbar_s^\ve}{\dt}\right)\cdot \sigma\left(\qbar_s^\ve,\dfrac{\qbar_s^\ve}{\dt}\right)dW_s \
.\nonumber
\end{align*}

Using this representation formula, Condition \ref{A:Assumption1} and Theorem 3.3
of \cite{HairerPavliotis2004} (see also Appendix \ref{S:AppendixA}), we can then establish that
for every $\eta>0$
\[
\lim_{\rho\downarrow0}\limsup_{\varepsilon\downarrow0}\mathbb{P}\left[
\sup_{|t_{1}-t_{2}|<\rho,0\leq t_{1}<t_{2}\leq1}|\bar{q}_{t_{1}}^{\varepsilon
}-\bar{q}_{t_{2}}^{\varepsilon}|\geq\eta\right]  =0.
\]

This implies the tightness of the family $\{\qbar_\bullet^\ve\}$.
Tightness of the occupation measures $\{\mathrm{P}%
^{\varepsilon,\Delta},\varepsilon>0\}$ follows from the bound
\begin{equation}
\sup_{\varepsilon\in(0,1]}\E\left[  g(\mathrm{P}^{\varepsilon
,\Delta})\right]= \sup_{\varepsilon\in(0,1]}\E\int_{0}^{T}\frac{1}{\Delta}%
\int_{t}^{t+\Delta}| u^{\ve}(s)|^{2}dsdt
<\infty.\label{Eq:TightnessP}
\end{equation}
for the tightness function
$g(r)=\int_{\mathcal{Z}\times\mathcal{Y}\times\lbrack0,T]}| z|
^{2}r(dz,dpdr,dt),\hspace{0.2cm}r\in\mathcal{P}(\mathcal{Z}\times\mathcal{Y}\times\lbrack0,T])$,
see Theorem A.19 in \cite{DupuisEllis}. Notice that the last inequality in
(\ref{Eq:TightnessP}) follows by the uniform $L^{2}$ bound on the
family of controls $\{u^{\ve},\ve>0\}$.

Hence, the family $\{(\bar{q}^{\varepsilon},\mathrm{P}^{\varepsilon,\Delta}),\epsilon>0\}$ is tight.
Due to tightness, for any
subsequence of $\varepsilon>0$ there exists subsubsequence
that converges, in distribution, to some limit
$(\bar{q},\mathrm{P})$ such that
\[
(\bar{q}^{\epsilon},\mathrm{P}^{\epsilon,\Delta})\rightarrow(\bar
{q},\mathrm{P}) \ .
\]
Next, we prove that any accumulation point will be a viable pair according to Definition \ref{Def:ViablePair}.
\

Part 2. [Proof of (\ref{Eq:ViablePairRequirement1})].  By Fatou's
Lemma we have
\begin{equation*}
\E\int_{\mathcal{Z}\times\mathcal{Y}\times[ 0,T]}| z|
^{2}\mathrm{P}(dz,dpdr,dt)<\infty \ ,
\label{A:Assumption2_1}%
\end{equation*}
which then implies that $\int_{\mathcal{Z}\times\mathcal{Y}\times[0,
T]}| z| ^{2}\mathrm{P}(dz,dpdr,dt)<\infty$ \ w.p.1.

\

Part 3. [Proof of (\ref{Eq:ViablePairRequirement2})]. Consider a
test function $f=f(q)$ on $\R^d$. Let $\Psi(p,r)=\Phi(p,r)\cdot
\grad_q f(q)$ which satisfies the cell problem
$$\cL^{m}_q \Psi(p,r)=-\dfrac{1}{\sqrt{m}}p\cdot \grad_q f(q) \ .$$

Making use of (\ref{Eq:HypoellipticCellProblem}) and (\ref{Eq:ItoFormulaToControlSystem}) we get
\begin{align}
\Psi\left(\pbar_t^\ve,
\dfrac{\qbar_t^\ve}{\dt}\right)-\Psi\left(\pbar_o^\ve,\dfrac{\qbar_o^\ve}{\dt}\right)
&= -\dfrac{\ve}{\dt^2}\int_0^t
\dfrac{1}{\sqrt{m}}\pbar_s^\ve\cdot\grad_qf(\qbar_s^\ve)ds\nonumber\\
&\quad+\dfrac{1}{\dt}\int_0^t
\dfrac{1}{\sqrt{m}}\left[c\left(\qbar_s^\ve,
\dfrac{\qbar_s^\ve}{\dt}\right)+\sigma\left(\qbar_s^\ve,
\dfrac{\qbar_s^\ve}{\dt}\right)u_s\right]\cdot
\grad_p\Psi\left(\pbar_s^\ve, \dfrac{\qbar_s^\ve}{\dt}\right)ds
\nonumber\\
&\quad+\dfrac{\sqrt{\ve}}{\dt}\frac{1}{\sqrt{m}}\int_0^t \grad_p\Psi\left(\pbar_s^\ve,
\dfrac{\qbar_s^\ve}{\dt}\right)\cdot \sigma\left(\qbar_s^\ve,
\dfrac{\qbar_s^\ve}{\dt}\right)dW_s \ .\label{Eq:ItoFormulaWithHypoellipticCellProblem}
\end{align}

Let us now choose $S,\tau\geq 0$ such that $S\leq S+\tau\leq T$. We have
\begin{align*}
f(\qbar_{S+\tau}^\ve)-f(\qbar_S^\ve)&=\int_S^{S+\tau}\dfrac{1}{\sqrt{m}}
\dfrac{\ve}{\dt}\pbar_t^\ve\cdot\grad_q f(\qbar_t^\ve)dt \ .
\end{align*}

Combining the latter expression with  (\ref{Eq:ItoFormulaWithHypoellipticCellProblem}) we get
\begin{align*}
&f(\qbar_{S+\tau}^\ve)-f(\qbar_S^\ve)-\int_S^{S+\tau}\gamma\left(\qbar_t^\ve,\left(\pbar_t^\ve, \dfrac{\qbar_t^\ve}{\dt}\right),u_t\right)\cdot\grad_q f(\qbar_t^\ve)dt\nonumber\\
&\quad =-\delta \left(\Psi\left(\pbar_t^\ve,
\dfrac{\qbar_t^\ve}{\dt}\right)-\Psi\left(\pbar_o^\ve,\dfrac{\qbar_o^\ve}{\dt}\right)\right)+\sqrt{\ve}\frac{1}{\sqrt{m}}\int_0^t \grad_p\Psi\left(\pbar_s^\ve,
\dfrac{\qbar_s^\ve}{\dt}\right)\cdot \sigma\left(\qbar_s^\ve,
\dfrac{\qbar_s^\ve}{\dt}\right)dW_s \ .
\end{align*}

Due to the a-priori bounds from Appendix \ref{S:AppendixA} the right hand side of the last display goes to zero in $L^{2}$, which means that
\begin{align*}
 &\left|f(\qbar_{S+\tau}^\ve)-f(\qbar_S^\ve)-\int_S^{S+\tau}\gamma\left(\qbar_t^\ve,
\left(\pbar_t^\ve, \dfrac{\qbar_t^\ve}{\dt}\right),
u_t\right)\cdot\grad_q f(\qbar_t^\ve)dt\right|\ra 0 
\end{align*}
 as $\ve\da 0$ in means square sense. By Condition \ref{A:Assumption1},  Lemma 3.2 of \cite{DupuisSpiliopoulos}
 guarantees that
\begin{align*}
&\left|\int_S^{S+\tau}\gamma\left(\qbar_t^\ve, \left(\pbar_t^\ve,
\dfrac{\qbar_t^\ve}{\dt}\right), u_t\right)\cdot\grad_q
f(\qbar_t^\ve)dt\right.
\nonumber\\
&\left.-\int_{\cZ\times\cY\times [S,
S+\tau]}\gamma(\qbar_t^\ve, (p,r), z)\cdot \grad_q
f(\qbar_t^\ve)\occupation^{\ve,\Dt}(dz,dpdr,dt)\right|\ra 0
\end{align*}
and
\begin{align*}
&\left|\int_{\cZ\times\cY\times [S, S+\tau]}\gamma(\qbar_t^\ve,
(p,r), z)\cdot \grad_q
f(\qbar_t^\ve)\occupation^{\ve,\Dt}(dz,dpdr,dt)\right.
\nonumber\\
 &\left.-\int_{\cZ\times\cY\times [S, S+\tau]}\gamma(\qbar_t,
(p,r), z)\cdot \grad_q f(\qbar_t)\occupation(dz,dpdr,dt)\right|\ra 0
\end{align*}
as $\ve\da 0$. Therefore, by defining
\begin{equation*}\bar{\mathcal{A}}_{t}^{\varepsilon,\Delta}f(q)=
\int_{\mathcal{Z}\times\mathcal{Y}}\gamma(q,(p,r),z)\nabla
f(q)\mathrm{P}_{t}^{\epsilon,\Delta}(dz,dpdr) \ ,
\label{Eq:PrelimitOperator}%
\end{equation*}
where
\[
\mathrm{P}_{t}^{\varepsilon,\Delta}(dz,dpdr)=\frac{1}{\Delta}\int_{t}^{t+\Delta
}1_{dz}(u^{\varepsilon}_{s})1_{dp}\left(
\bar{p}^{\varepsilon}_{s}\right) 1_{dr}\left(
\frac{\bar{q}^{\varepsilon}_{s}}{\delta}\text{ mod } 1\right)  ds \
,
\]
we get that, as $\varepsilon\downarrow0$,
\begin{equation}
\E\left[  f(\bar{q}^{\varepsilon}_{S+\tau})-f(\bar{q}^{\varepsilon}_{S})-\int_{S}^{S+\tau}%
\bar{\mathcal{A}}_{t}^{\varepsilon,\Delta}f(\bar{q}^{\varepsilon}_{t})dt\right]
  \rightarrow0  \ , \label{Eq:MartingaleProblemRegime1_1}%
\end{equation}
and, in probability,
\begin{equation}
\int_{S}^{S+\tau}\bar{\mathcal{A}}_{s}^{\epsilon,\Delta}f(\bar{q}^{\varepsilon}_{s})ds-\int_{\mathcal{Z}\times\mathcal{Y}\times\lbrack S,S+\tau
]}\gamma(\bar{q}_{s},(p,r),z)\nabla f(\bar{q}_{s})\mathrm{P}(dz,dpdr,ds)\rightarrow0.
\label{Eq:MartingaleProblemRegime1_2a}%
\end{equation}

Relations (\ref{Eq:MartingaleProblemRegime1_1}) and (\ref{Eq:MartingaleProblemRegime1_2a}) imply that  the pair $(\bar{q},\mathrm{P})$ solves the martingale problem associated with  (\ref{Eq:ViablePairRequirement2}), which then proves that (\ref{Eq:ViablePairRequirement2}) holds.

\

Part 4. [Proof of (\ref{Eq:ViablePairRequirement3})]. 
For functions $f\in C^{2}(\mathcal{Y})$, let us introduce the auxiliary operator
$$\cA_{z,q}^\ve f(p,r)=\dfrac{\ve}{\dt^2}\cL^{m}_q f(p,r)+
\dfrac{1}{\dt}\dfrac{1}{\sqrt{m}}\left[c(q,r)+\sigma(q,r)z\right]\cdot
\grad_p f(p,r) \ ,$$
and define the $\cF_t$--martingale
\begin{align*}
M_t^\ve&=f\left(\pbar_t^\ve, \dfrac{\qbar_t^\ve}{\dt}\right)-
f\left(\pbar_o^\ve, \dfrac{\qbar_o^\ve}{\dt}\right)-\int_0^t
\cA_{u_s^\ve,\qbar_s^\ve}^\ve f\left(\pbar_s^\ve,
\dfrac{\qbar_s^\ve}{\dt}\right)ds\nonumber\\
&=\dfrac{\sqrt{\ve}}{\dt}\frac{1}{\sqrt{m}}\int_0^t
\grad_p f\left(\pbar_s^\ve,\dfrac{\qbar_s^\ve}{\dt}\right)\sigma\left(\qbar_s^\ve,\dfrac{\qbar_s^\ve}{\dt}\right)dW_s \ .
\end{align*}

Let us furthermore set $\mathcal{G}_{q,z}^\ve f(p,r)=\dfrac{1}{\sqrt{m}}\left[c(q,r)+\sigma(q,r)z\right]\cdot\grad_p f(p,r)$ and define $g(\ve)=\frac{\delta^{2}}{\ve}$. Then, we have that
\begin{align}
& g(\varepsilon)M_{t}^{\varepsilon}-g(\varepsilon)\left[  f\left(\bar{p}^{\varepsilon}_{t},\frac{\bar{q}
^{\varepsilon}_{t}}{\delta}\right)-f\left(\bar{p}^{\varepsilon}_{0},\frac{\bar{q}^{\varepsilon}_{0}}{\delta}\right)\right] \label{Eq:MartingaleProperty_1}\\
&\quad \mbox{} +g(\varepsilon
)\left[\int_{0}^{t}\frac
{1}{\Delta}\left[  \int_{s}^{s+\Delta}\mathcal{A}_{u^{\varepsilon}_{\rho},\bar
{q}^{\varepsilon}_{\rho}}^{\varepsilon}f\left(\bar{p}^{\varepsilon}_{\rho},\frac{\bar{q}
^{\varepsilon}_{\rho}}{\delta}\right)d\rho\right]  ds-\int_{0}^{t}\mathcal{A}_{u^{\varepsilon}_{s},\bar{q}^{\varepsilon}_{s}}^{\varepsilon
}f\left(\bar{p}^{\varepsilon}_{s},\frac{\bar{q}
^{\varepsilon}_{s}}{\delta}\right)ds\right]\nonumber\\
&  =-\frac{\delta}{\varepsilon}\left(  \int_{0}^{t}\frac{1}{\Delta}\left[
\int_{s}^{s+\Delta}\left[  \mathcal{G}_{\bar{q}^{\varepsilon}_{\rho},u^{\varepsilon}_{\rho}}f\left(\bar{p}^{\varepsilon}_{\rho},\frac{\bar{q}
^{\varepsilon}_{\rho}}{\delta}\right)
-\mathcal{G}_{\bar{q}^{\varepsilon}_{s},u^{\varepsilon}_{\rho}
}f\left(\bar{p}^{\varepsilon}_{\rho},\frac{\bar{q}
^{\varepsilon}_{\rho}}{\delta}\right)\right]  d\rho\right]  ds\right) \nonumber\\
& \quad \mbox{}-\frac{\delta}{\varepsilon}\left(  \int_{\mathcal{Z}\times
\mathcal{Y}\times[0,t]}\mathcal{G}_{\bar{q}^{\varepsilon}_{s},z}f\left(p,r\right)\bar{\mathrm{P}}^{\varepsilon,\Delta}(dz,dpdr,ds)\right) \nonumber\\
& \quad \mbox{}-\int_{0}^{t}\frac{1}{\Delta}  \int_{s}^{s+\Delta}\left[
\mathcal{L}^{m}_{\bar{q}^{\varepsilon}_{\rho}}f\left(\bar{p}^{\varepsilon}_{\rho},\frac{\bar{q}
^{\varepsilon}_{\rho}}{\delta}\right)-\mathcal{L}^{m}_{\bar{q}^{\varepsilon}_{s}}f\left(\bar{p}^{\varepsilon}_{\rho},\frac{\bar{q}
^{\varepsilon}_{\rho}}{\delta}\right)  d\rho\right]  ds\nonumber\\
&\quad  \mbox{}-\int_{\mathcal{Z}\times\mathcal{Y}\times[0,t]}\mathcal{L}^{m}
_{\bar{q}^{\varepsilon}_{s}}f\left(p,r\right)\mathrm{P}^{\varepsilon,\Delta}(dz,dpdr,dt).\nonumber
\end{align}

Let us now analyze the different terms in (\ref{Eq:MartingaleProperty_1}). We start by observing that $\E\left[
M_{T}^{\varepsilon}\right]  ^{2}\leq C_{0}\frac{1}{g(\varepsilon)}$, which then implies that $g(\varepsilon)M^{\varepsilon}_{t}\downarrow 0$ in probability, as $\varepsilon\downarrow0$. Moreover, boundedness of $f$ implies that
$g(\varepsilon)\left[  f\left(\bar{p}^{\varepsilon}_{t},\frac{\bar{q}
^{\varepsilon}_{t}}{\delta}\right)-f\left(\bar{p}^{\varepsilon}_{o},\frac{\bar{q}^{\varepsilon}_{o}}{\delta}\right)\right]$ converges to zero uniformly. Hence, the left hand side of (\ref{Eq:MartingaleProperty_1}) converges to zero in probability as $\varepsilon\downarrow 0$.

Let us next study the right hand side of (\ref{Eq:MartingaleProperty_1}). We have the following
\begin{enumerate}
\item{Conditions \ref{A:Assumption1}, the $L^{2}$ uniform bound on the controls and tightness of $\left\{\bar{q}^{\varepsilon},\varepsilon>0\right\}$, imply that the
first and the third term in the right hand side of
(\ref{Eq:MartingaleProperty_1}) converge to zero in probability as $\delta/\epsilon\downarrow0$.}
\item{The
second term on the right hand side of (\ref{Eq:MartingaleProperty_1}) also
converges to zero in probability, by the fact that $\delta/\epsilon
\downarrow0$ and uniform integrability of
$\mathrm{P}^{\epsilon,\Delta}$.}
\end{enumerate}

Thus, by combining the behavior of the different terms on the left and on the right hand side of (\ref{Eq:MartingaleProperty_1}), we obtain that we should necessarily have that
\[\int_{\mathcal{Z}\times\mathcal{Y}\times[0,T]}\mathcal{L}^{m}
_{\bar{q}_{t}^{\varepsilon}}f(p,r)\mathrm{P}^{\epsilon,\Delta}(dz,dpdr,dt)\rightarrow 0, \quad \textrm{ in probability.}
\]
which by continuity in $t\in[0,T]$ gives (\ref{Eq:ViablePairRequirement3}).
\

Part 5. [Proof of (\ref{Eq:ViablePairRequirement4}).] Finally $\occupation(\cZ\times\cY\times[0,t])=t$ follows from the
fact that analogous property holds at the prelimit level,
$\occupation(\cZ\times\cY\times\{t\})=0$ and the continuity of $t\ra
\occupation(\cZ\times\cY\times [0,t])$ and (\ref{Eq:ViablePairRequirement4}) follows.
\end{proof}

\section{Laplace principle}\label{S:LDP}
The main result of this section is the following Laplace principle. During the proof of  Theorem \ref{T:MainLaplacePrinciple} we also
establish the alternative representation of Theorem
\ref{T:LDPexplcitRepresentation}.

\begin{theorem}\label{T:MainLaplacePrinciple}
 Let $\{q_\bullet^\ve, \ve>0\}$ be the
unique strong solution to (1). Assume Conditions \ref{A:Assumption1} and \ref{A:Assumption2}. Define
\begin{equation}
S_{m}(\phi)=\inf\li_{(\phi,\occupation)\in \cV_{(\gamma,
\cL^{m}_q)}}
\left[\dfrac{1}{2}\int_{\cZ\times\cY\times[0,T]}|z|^2\occupation(dz,dpdr,dt)\right]\label{Eq:GeneralizedRateFunction}
\end{equation}
 with the convention that the infimum over the
empty set is $\infty$. Then for every bounded and continuous
function $h$ mapping $\cC([0,T]; \R^d)$ into $\R$ we have
\begin{align*}
\lim\li_{\ve\da
0}-\ve\ln\E_{q_0}\left[\exp\left(-\dfrac{h(q_\bullet^\ve)}{\ve}\right)\right]
=\inf\li_{\phi\in \cC([0,T];\R^d)}[S_{m}(\phi)+h(\phi)] \ .
\label{Eq:SpecificLaplacePrinciple}
\end{align*}
Moreover, for each $s<\infty$, the set $$\Phi_s=\{\phi\in
\cC([0,T];\R^d): S_{m}(\phi)\leq s\} $$ is a compact subset of
$\cC([0,T]; \R^d)$.
\end{theorem}

In other words, $\{q^\ve_\bullet, \ve>0\}$ satisfies the Laplace
principle with rate function $S(\bullet)$.
\begin{proof}[Proof of Theorem \ref{T:MainLaplacePrinciple}]
The proof of this theorem borrows some of the arguments of the
related proof of the LDP for the elliptic overdamped case of Theorem 2.6
in \cite{DupuisSpiliopoulos}. We present here the main arguments,
emphasizing the differences.
\

Part 1. [Laplace principle lower bound]. Theorem \ref{T:LLN} and
Fatou's lemma, guarantee the validity of the following chain of inequalities.
\begin{align*}
&\liminf_{\varepsilon\downarrow0}\left(  -\varepsilon\ln\E\left[
\exp\left\{  -\frac{h(q^{\varepsilon})}{\varepsilon}\right\}  \right]  \right)
\geq\liminf_{\varepsilon\downarrow0}\left(  \E\left[  \frac{1}%
{2}\int_{0}^{T}\left| u_{t}^{\varepsilon}\right| ^{2}dt+h(\bar
{q}^{\varepsilon})\right]  -\varepsilon\right) \nonumber\\
& \qquad \geq\liminf_{\varepsilon\downarrow0}\left(  \E\left[  \frac
{1}{2}\int_{0}^{T}\frac{1}{\Delta}\int_{t}^{t+\Delta}\left| u_{s}
^{\varepsilon}\right| ^{2}dsdt+h(\bar{q}^{\varepsilon})\right]
\right)
\nonumber\\
& \qquad =\liminf_{\varepsilon\downarrow0}\left(  \E\left[  \frac{1}%
{2}\int_{\mathcal{Z}\times\mathcal{Y}\times\lbrack0,T]}\left|
z\right|
^{2}\mathrm{P}^{\varepsilon,\Delta}(dz,dpdr,dt)+h(\bar{q}^{\varepsilon})\right]
\right)
\nonumber\\
& \qquad \geq\E\left[  \frac{1}{2}\int_{\mathcal{Z}\times
\mathcal{Y}\times[0,T]}\left| z\right|
^{2}\bar{\mathrm{P}}%
(dz,dpdr,dt)+h(\bar{q})\right] \nonumber\\
& \qquad \geq\inf_{(\phi,\mathrm{P})\in\mathcal{V}_{(\lambda,\mathcal{L}^{m}_{q})}}\left\{  \frac{1}{2}%
\int_{\mathcal{Z}\times\mathcal{Y}\times[0,T]}\left| z\right|
^{2}\mathrm{P}(dz,dpdr,dt)+h(\phi)\right\}  \nonumber\\
& \qquad =\inf_{\phi\in
\mathcal{C}([0,T];\mathbb{R}^{d})}\left[  S_{m}(\phi)+h(\phi)\right].
\end{align*}
Hence, the lower bound has been established.

\
Part 2. [Laplace principle upper bound and alternative
representation].  We
first observe that one can write (\ref{Eq:GeneralizedRateFunction}) in terms of a local rate
function
$$S_{m}(\phi)=\int_0^T L^r(\phi_s, \dot{\phi}_s)ds  \ .$$

Here we set
$$L^r(x,\nu)=\inf\li_{\occupation\in \cA_{q,\nu}^r}\int_{\cZ\times\cY}\dfrac{1}{2}|z|^2\occupation(dz,dpdr) \ ,$$
where
$$\cA_{q,\nu}^r=\left\{
\begin{array}{l}
\play{\occupation\in \mathcal{P}(\cZ\times \cY):
\int_{\cZ\times\cY}\cL^{m}_q f(p,r)\occupation(dz,dpdr)=0, \forall f\in \cC^2_{\text{loc}}(\cY) \ , }
\\
\play{\ \ \ \ \ \ \
\int_{\cZ\times\cY}|z|^2\occupation(dz,dpdr)<\infty \text{ and }
\nu=\int_{\cZ\times\cY}\gamma(q,(p,r),z)\occupation(dz,dpdr)}
\end{array}\right\} \ .$$

We can decompose the measure $\occupation\in
\mathcal{P}(\cZ\times\cY)$ into the form
\begin{equation*}
\occupation(dz, dpdr)=\eta(dz|p,r)\mu(dpdr|q) \ ,
\end{equation*}
where $\mu$ is a
probability measure on $\cY$ and $\eta$ is a stochastic kernel on
$\cZ$ given $\cY$. This is referred to as the ``relaxed" formulation
because the control is characterized as a distribution on $\cZ$
(given $q$ and $(p,r)$) rather than as an element of $\cZ$. We now
have, for every $f\in \cC_{\text{loc}}^2(\cY)$ and for every $q\in\mathbb{R}^{d}$, that
\begin{equation*}
\int_{\cY}\cL^{m}_q f(p,r)\mu(dpdr)=0 \ .
\end{equation*}

Here we have used the independence of $\cL^{m}_q$ on the control
variable $z$ to eliminate the stochastic kernel $\eta$. Thus $\mu(dpdr)$ is the unique
corresponding to the operator $\cL^{m}_q$, written as $\mu(dpdr|q)$.

Since the cost is convex in $z$ and $\gamma$ is affine in $z$, the
relaxed control formulation is equivalent to the following ordinary
control formulation of the local rate function
$$L^o(q,\nu)=\inf\li_{(v,\mu)\in \cA_{q,\nu}^o}
\dfrac{1}{2}\int_{\cY}|v(p,r)|^2\mu(dpdr) \ ,$$ where
$$\begin{array}{l}
\play{\cA_{q,\nu}^o=\left\{v(\bullet): \cY\ra \R^d, \mu\in
\mathcal{P}(\cY) \ , \ (v,\mu) \text{ satisfy } \int_{\cY}\cL^{m}_q
f(p,r)\mu(dpdr)=0, \forall f\in \cC^2_{\text{loc}}(\cY) \ ,
\ \right.}
\\
\play{\ \ \ \ \ \ \ \ \ \ \ \left.
\int_{\cY}|v(p,r)|^2\mu(dpdr)<\infty \text{ and }
\nu=\int_{\cY}\gamma(q,(p,r),v(p,r))\mu(dpdr) \right\} \ .}
\end{array}$$

One can show as in \cite[Section 5]{DupuisSpiliopoulos} that $L^r(q,\nu)=L^o(q,\nu)$. 
Let us recall now the definitions of $r_{m}(q)$ and $Q_{m}(q)$ from Theorem \ref{T:LDPexplcitRepresentation}.
For any $v\in \cA_{q,\nu}^o$ we can write
$$\begin{array}{ll}
\nu & \play{=\int_{\cY}\gamma(q,(p,r),v(p,r))\mu(dpdr|q)}
\\
&
\play{=\int_{\cY}\dfrac{1}{\sqrt{m}}\left[c(q,r)+\sigma(q,r)v(p,r)\right]\cdot \grad_p\Phi(p,r)\mu(dpdr|q)}
\\
&
\play{=r_{m}(q)+\int_{\cY}\frac{1}{\sqrt{m}}\grad_p\Phi(p,r)\sigma(q,r)(v(p,r))^T\mu(dpdr|q)
\ .}
\end{array}$$

Then, $\nu-r_{m}(q)$ can be treated as $\nu$, and
$\kp(q,(p,r))=\frac{1}{\sqrt{m}}(\grad_p\Phi(p,r))^{T}(\sigma(q,r))^T$,
$u(p,r)=(v(p,r))^T$ in \cite[Lemma 5.1]{DupuisSpiliopoulos}. We
apply this lemma and then we get that for all $v\in \cA_{q,\nu}^o$,
\begin{equation*}
 \int_{\cY}|v(p,r)|^2\mu(dpdr|q)\geq (\nu-r_{m}(q))^TQ^{-1}_{m}(q)(\nu-r_{m}(q)).
\end{equation*}

Moreover, if we take
\begin{equation}
v(p,r)=\bar{u}_\nu(q,(p,r))=\frac{1}{\sqrt{m}}\sigma^{T}(q,r)(\grad_p\Phi(p,r))^{T}Q^{-1}_{m}(q)(\nu-r_{m}(q)) \ ,\label{Eq:OptimalControl}
\end{equation}
we will have
$$\int_{\cY}|\bar{u}_\nu(q,(p,r))|^2\mu(dpdr|q)=(\nu-r_{m}(q))^TQ^{-1}_{m}(q)(\nu-r_{m}(q)) \ .$$

This shows that
$$L^o(x,\nu)=\dfrac{1}{2}(\nu-r_{m}(q))^TQ^{-1}_{m}(q)(\nu-r_{m}(q)) \ ,$$
and the minimum is achieved in (\ref{Eq:OptimalControl}).

Now, that we have identified that the action functional can be written in the proceeding form we can proceed in proving the Laplace principle upper bound.  We must show that for all bounded,
continuous functions $h$ mapping $\mathcal{C}([0,T];\mathbb{R}^{d})$ into
$\mathbb{R}$
\begin{equation*}
\limsup_{\varepsilon\downarrow0}-\varepsilon\ln\E\left[
\exp\left\{  -\frac{h(q^{\varepsilon})}{\varepsilon}\right\}  \right]  \leq
\inf_{\phi\in\mathcal{C}([0,T];\mathbb{R}^{d})}\left[  S_{m}(\phi)+h(\phi)\right]. 
\end{equation*}

By the variational representation formula, it is enough to prove that
\begin{equation}
\limsup_{\varepsilon\downarrow0}\inf_{u\in\mathcal{A}}\E\left[
\frac{1}{2}\int_{0}^{T}\left|u_{s}\right|
^{2}ds+h(\bar{q}^{\varepsilon})\right]  \leq
\inf_{\phi\in\mathcal{C}([0,T];\mathbb{R}^{d})}\left[  S_{m}(\phi)+h(\phi)\right]. \label{Eq:LaplacePrincipleUpperBoundRegime1}%
\end{equation}

To be precise, we consider for the limiting variational problem in the Laplace
principle a nearly optimal control pair $(\psi,\mathrm{P})$. In particular, let $\eta>0$ be given and consider $\psi\in\mathcal{C}([0,T];\mathbb{R}^{d})$
with $\psi_{0}=q_{o}$ such that
\begin{equation*}
S_{m}(\psi)+h(\psi)\leq\inf_{\phi\in\mathcal{C}([0,T];\mathbb{R}^{d})}\left[
S_{m}(\phi)+h(\phi)\right]  +\eta<\infty.
\label{Eq:NearlyOptimalTrajectoryRegime1}%
\end{equation*}

It is clear now that $L^o(x,\nu)$ is continuous and finite at each pair $(x,\nu)\in\mathbb{R}^{2d}$. Hence, a standard mollification argument, allows us to assume that $\dot{\psi}$ is piecewise constant, see Lemmas 6.5.3 and 6.5.5 in Subsection $6.5$ of \cite{DupuisEllis}.  The control in feedback from used to prove (\ref{Eq:LaplacePrincipleUpperBoundRegime1}) is then given by (\ref{Eq:OptimalControl}), i.e,
\[
\bar{u}_{t}=\bar{u}_{\dot{\psi}_{t}}\left(\bar{q}^{\varepsilon}_{t},
\left(\bar{p}^{\varepsilon}_{t},\frac{\bar{q}^{\varepsilon}_{t}}{\delta}\right)\right)
\ .
\]

It is easy to see that Condition \ref{A:Assumption1} guarantees that $\bar{u}_{t}$ is continuous in all of its arguments  and that (\ref{Eq:SecondOrderControlSystem}) has a unique strong solution with $u_{t}=\bar{u}_{t}$. Then, by Theorem \ref{T:LLN}, we obtain that in distribution $\bar{q}^{\varepsilon}\overset{\mathcal{D}}{\rightarrow}\bar{q}$, where
\begin{equation*}
\bar{q}_{t}=q_{o}+\int_{0}^{t}\int_{\mathcal{Y}}\gamma\left(  \bar{q}_{s},(p,r),\bar{u}_{\dot{\psi}_{s}}(\bar{q}_{s},(p,r))\right)  \mu(dpdr|\bar{q}_{s})ds.\label{Eq:LimitingXReg2}%
\end{equation*}

Keeping in mind  the definition of $\mathcal{A}_{q,\dot{\psi}_{t}}^{o}$ and that
$\psi_{0}=q_{o}$,   we obtain that
\[
\bar{q}_{t}=q_{o}+\int_{0}^{t}\dot{\psi}_{s}ds=\psi_{t}\hspace{0.2cm}\text{
for any }t\in[0,T]\text{, with probability } 1 \ .
\]

Therefore,  we finally obtain that
\begin{align}
\limsup_{\varepsilon\downarrow0}\left[  -\varepsilon\ln\E\left[
\exp\left\{  -\frac{h(q^{\varepsilon})}{\varepsilon}\right\}  \right]  \right]   &
=\limsup_{\varepsilon\downarrow0}\inf_{u}\E\left[  \frac{1}%
{2}\int_{0}^{T}\left| u_{t}\right| ^{2}dt+h(\bar{q}^{\varepsilon})\right]  \nonumber\\
&  \leq\limsup_{\varepsilon\downarrow0}\E\left[  \frac{1}{2}%
\int_{0}^{T}\left| \bar{u}_{t}\right| ^{2}dt+h(\bar
{q}^{\varepsilon})\right]  \nonumber\\
&  =\E\left[  S_{m}(\bar{X})+h(\bar{X})\right]  \nonumber\\
&  \leq\inf_{\phi\in\mathcal{C}([0,T];\mathbb{R}^{d})}\left[  S_{m}(\phi
)+h(\phi)\right]  +\eta.\nonumber
\end{align}

Since $\eta$ is arbitrary, we are done with the proof of the Laplace principle upper bound. At the same time we get the explicit form of the rate function
\begin{align}
S_{m}(\phi)&=\left\{
\begin{array}{ll}
\play{\dfrac{1}{2}\int_0^T
(\dot{\phi}_s-r_{m}(\phi_s))^TQ^{-1}_{m}(\phi_s)(\dot{\phi}_s-r_{m}(\phi_s))ds}&
\text{ if } \phi\in \mathcal{AC}([0,T]; \R^d), \phi_{0}=q_{o}
\\
+\infty & \text{ otherwise \ .}
\end{array}\right.\label{Eq:RateFcnAlternativeRepresentation}
\end{align}

Part 3. [Compactness of level sets]. This follows directly from the alternative representation (\ref{Eq:RateFcnAlternativeRepresentation}), as it is in the standard quadratic form, see for example \cite{DupuisSpiliopoulos}.

\

This concludes the proof of the theorem as well as of the alternative representation of Theorem \ref{T:LDPexplcitRepresentation}.
\end{proof}

\section{Convergence of the action functional as $m\rightarrow
0$}\label{S:RateFunctions}

Let $\beta>0$ and set $\sigma(q,r)=\sqrt{2\beta\lambda(q)} I$. Recall the definitions of the operators $\cL^{m}_q$ and $\mathcal{L}_{q}^{0}$ and of the corresponding invariant measures from Section \ref{S:MainResults}.

Theorem \ref{T:MainTheorem3} follows directly from Lemma \ref{L:ConvergenceOfQandr} below, whose proof is based on Theorems \ref{T:ConvergenceInvariantMeasures} and \ref{T:ConvergenceCellProblems}.
\begin{lemma}\label{L:ConvergenceOfQandr}
Assume that Conditions \ref{A:Assumption1} and \ref{A:Assumption2} hold. Let $Q_{m}(q), r_{m}(q)$ and $Q_{0}(q), r_{0}(q)$ be as in Theorems \ref{T:LDPexplcitRepresentation} and \ref{T:MainTheorem3} respectively. Then, for any $\ve>0$, there exist some $m_0>0$ such that for every $q\in
\R^d$ and every $0<m<m_0$ we have
$$|Q_m(q)-Q_0(q)|<\ve \ , \ |r_m(q)-r_0(q)|<\ve \ .$$
\end{lemma}
\begin{proof}
For notational convenience and without loss of generality, we shall set $\sigma(q,r)=2 I, \beta=\lambda(q)=1$. Since $q$ is viewed as a parameter, we do not mention it explicitly in the formulas.
We have
\begin{align*}
&Q_{m}-Q_{0}=2\int_{\mathcal{Y}}\left|\frac{1}{\sqrt{m}}\partial_{p}\Phi(p,r)\right|^{2}\rho^{0}(p,r)\tilde{\rho}^{m}(p,r)dpdr-2\int_{\mathcal{Y}}\left|I+\partial_{r}\chi(r)\right|^{2}\rho^{0}(p,r)dpdr\nonumber\\
&\quad=2\int_{\mathcal{Y}}\left[\left|\frac{1}{\sqrt{m}}\partial_{p}\Phi(p,r)\right|^{2}-\left|I+\partial_{r}\chi(r)\right|^{2}\right]\rho^{0}(p,r)\tilde{\rho}^{m}(p,r)dpdr\nonumber\\
&\qquad+2\int_{\mathcal{Y}}\left|I+\partial_{r}\chi(r)\right|^{2}\left(\tilde{\rho}^{m}(p,r)-1\right)\rho^{0}(p,r)dpdr\nonumber\\
&\quad=2\int_{\mathcal{Y}}\left[\left(\frac{1}{\sqrt{m}}\partial_{p}\Phi(p,r)-\left(I+\partial_{r}\chi(r)\right)\right)\left(\frac{1}{\sqrt{m}}\partial_{p}\Phi(p,r)+\left(I+\partial_{r}\chi(r)\right)\right)^{T}\right]\rho^{m}(p,r)dpdr\nonumber\\
&\qquad+2\int_{\mathcal{Y}}\left|I+\partial_{r}\chi(r)\right|^{2}\left(\tilde{\rho}^{m}(p,r)-1\right)\rho^{0}(p,r)dpdr
\ .
\end{align*}

Taking absolute value and using Cauchy--Schwarz inequality we obtain
\begin{align*}
|Q_{m}-Q_{0}|
&\leq 2\left\|\frac{1}{\sqrt{m}}\partial_{p}\Phi-\left(I+\partial_{r}\chi\right)\right\|_{L^{2}(\mathcal{Y}; \rho^{m})}
\left\|\frac{1}{\sqrt{m}}\partial_{p}\Phi+\left(I+\partial_{r}\chi\right)\right\|_{L^{2}(\mathcal{Y}; \rho^{m})}\nonumber\\
&\qquad+2\left\|\left(I+\partial_{r}\chi\right)^{2}\right\|_{L^{2}(\mathcal{Y};
\rho^{0})}\left\|\tilde{\rho}^{m}-1\right\|_{L^{2}(\mathcal{Y};
\rho^{0})}
\end{align*}

By Theorem \ref{T:ConvergenceInvariantMeasures} we have
\begin{align}
\lim_{m\rightarrow0}\left\|\tilde{\rho}^{m}-1\right\|_{L^{2}(\mathcal{Y};
\rho^{0})}&=0 \ . \label{Eq:RelationToShow3}
\end{align}

By Theorem \ref{T:ConvergenceCellProblems} we have
\begin{align}
\lim_{m\rightarrow0}\left\|\frac{1}{\sqrt{m}}\partial_{p}\Phi-\left(I+\partial_{r}\chi\right)\right\|_{L^2(\mathcal{Y};
\rho^{m})}&=0 \ .\label{Eq:RelationToShow1}
\end{align}

The results (\ref{Eq:RelationToShow3}) and
(\ref{Eq:RelationToShow1}) imply that there exists a uniform
constant $C$ such that
\begin{align}
\sup_{m\in(0,1)}\left\|\frac{1}{\sqrt{m}}\partial_{p}\Phi\right\|_{L^{2}(\mathcal{Y};
\rho^{m})}+\sup_{m\in(0,1)}\left\|\left|I+\partial_{r}\chi\right|^{2}\right\|_{L^{2}(\mathcal{Y};
\rho^{m})}\leq C \ , \label{Eq:RelationToShow2}
\end{align}
and by classical elliptic regularity theory there exists a uniform
constant $C$, clearly independent of $m$, such that
\begin{align}
\left\|\left(I+\partial_{r}\chi\right)^{2}\right\|_{L^{2}(\mathcal{Y};
\rho^{0})}\leq C \ .
\label{Eq:RelationToShow4}
\end{align}
From (\ref{Eq:RelationToShow3})--(\ref{Eq:RelationToShow4}) we infer
the first inequality of this Lemma. In a similar way from
(\ref{Eq:RelationToShow3}) and (\ref{Eq:RelationToShow2}) we also
derive the second estimate of this Lemma.
\end{proof}

\section{$L^{2}$ Convergence of the invariant density}\label{S:InvariantMeasures}
In this section we prove Theorem \ref{T:ConvergenceInvariantMeasures}. For notational convenience and without loss of generality, let us
assume in this Section that $\alpha(q,r)=2I$ and that
$\beta=\lambda=1$ (recall $\sm(q,r)=\sqrt{2\lb(q)\bt} I$). Since $q\in\mathbb{R}^{d}$ is viewed as a parameter, it will not be mentioned explicitly.

We want to show that
\begin{align}
\lim_{m\rightarrow0}\left\|\tilde{\rho}^{m}-1\right\|_{L^{2}(\mathcal{Y};
\rho^{0})}&=0 \ . \label{Eq:RelationToShow3a}
\end{align}
where we recall that $\rho^{0}(p,r)=\rho^{\text{OU}}(p)\rho_{0}(r)$.

Notice that in the case of gradient potential, i.e., when
$b(q,r)=-\nabla_{r}V(q,r)$ then (\ref{Eq:RelationToShow3a}) is
immediately true even without the limit. In fact in this case we
have that the invariant density is basically
$\rho^{m}(p,r)=\rho^{\text{OU}}(p)\rho_{0}(r)$  for every finite
$m\in\mathbb{R}_{+}$ which implies that $\tilde{\rho}^{m}(p,r)=1$
completing the proof of (\ref{Eq:RelationToShow3a}). Our goal here
is to show that this true in the more general setting of not
potential drifts.

By Condition \ref{A:Assumption1} the drift $b(q,r)$ and
its partial derivatives are uniformly bounded with respect to $q$. For
this reason we sometimes suppress the dependence on $q$ and write
$b(q,r)=b(r)$. Also, for notational convenience, let us set
\[
h(r)= b(r)-\nabla_{r}\log\rho_{0}(r) \ .
\]
This definition for $h(r)$ will also be used throughout the rest of the
paper.

Notice that in the gradient case, i.e, when $b(r)=-\nabla V(r)$, we
have that $h(r)=0$, but in the general case one has $h(r)\neq 0$. Let us next establish some useful relations
\begin{lemma}\label{L:L2_Bounds1}
 Let $f,g$ be two functions that belong in the domain of definition of $\cL^{m}_{q}$. Then, we have the identity
 \begin{align*}
  &\int_{\mathcal{Y}}\left[\left(\cL^{m}_{q}f(p,r)\right) g(p,r) +\left(\cL^{m}_{q}g(p,r)\right) f(p,r) \right]\rho^{0}(p,r)dpdr=\nonumber\\
  &\quad -\frac{2}{m}\int_{\mathcal{Y}}\nabla_{p}f(p,r)\nabla_{p}g(p,r)\rho^{0}(p,r)dpdr+\frac{1}{\sqrt{m}}
  \int_{\mathcal{Y}}f(p,r)g(p,r)h(r)p\rho^{0}(p,r)dpdr \ .
 \end{align*}
 In particular, we have that
 \begin{align*}
  &\int_{\mathcal{Y}}\left(\cL^{m}_{q}f(p,r)\right) f(p,r) \rho^{0}(p,r)dpdr=\nonumber\\
  &\qquad=-\frac{1}{m}\int_{\mathcal{Y}}\left|\nabla_{p}f(p,r)\right|^{2}\rho^{0}(p,r)dpdr
  +\frac{1}{2\sqrt{m}}\int_{\mathcal{Y}}\left|f(p,r)\right|^{2}h(r)p\rho^{0}(p,r)dpdr
  \ .
 \end{align*}
\end{lemma}

\begin{lemma}\label{L:L2_Bounds3}
Let $f,g$ be two functions that are in
$\mathcal{W}^{1,0}_{2}(\mathcal{Y})$, i.e., the functions and their derivatives with respect to $p$ are in $L^{2}(\mathcal{Y})$. Then, there exists a finite
constant $K<\infty$ that depends only on
$\sup_{r\in\mathbb{T}^{d}}\left|h(r)\right|$ such that
\begin{align*}
\left|\left<h(r)p,f g\right>_{L^{2}(\mathcal{Y};
\rho^{0})}\right|&\leq K \left[\left\|f\right\|_{L^{2}(\mathcal{Y};
\rho^{0})} \left\|\nabla_{p}g\right\|_{L^{2}(\mathcal{Y};
\rho^{0})}+\left\|\nabla_{p}f\right\|_{L^{2}(\mathcal{Y}; \rho^{0})}
\left\|g\right\|_{L^{2}(\mathcal{Y}; \rho^{0})}\right] \ .
\end{align*}
\end{lemma}

\begin{lemma}\label{L:L2_Bounds4}
For every $\eta>0$, there exists  constant  constant $K<\infty$ that
depends only on $\sup_{r\in\mathbb{T}^{d}}\left|h(r)\right|$ such
that
\begin{align*}
\left<f,\mathcal{B} f\right>_{L^{2}(\mathcal{Y};
\rho^{0})}&=\frac{1}{2}\left<p h(r),
|f|^{2}\right>_{L^{2}(\mathcal{Y}; \rho^{0})}\geq -K \left[\eta
\left\|f\right\|^{2}_{L^{2}(\mathcal{Y}; \rho^{0})}
+\frac{1}{4\eta}\left\|\nabla_{p}f\right\|^{2}_{L^{2}(\mathcal{Y};
\rho^{0})}\right] \ ,
\end{align*}
where we recall that $\mathcal{B}f=p\cdot\nabla_{r}f+b(q,r)\nabla_{p}f$.
\end{lemma}

The proof of Lemmas \ref{L:L2_Bounds1}--\ref{L:L2_Bounds4} are in
Appendix B. Let us now define
\begin{equation*}
 \delta^{m}(p,r)=\tilde{\rho}^{m}(p,r)-1 \ .
\end{equation*}

Recall that our goal is to prove Theorem \ref{T:ConvergenceInvariantMeasures}, i.e. that (\ref{Eq:RelationToShow3a}) holds.
The next lemmas are towards this direction. The proof of Lemmas
\ref{L:L2_Bounds2}--\ref{L:L2_Bounds2b} are in Appendix B.
\begin{lemma}\label{L:L2_Bounds2}
 For every $m>0$ we have the following equality
\begin{align*}
&\left\|\nabla_{p}\delta^{m}\right\|^{2}_{L^{2}(\mathcal{Y};
\rho^{0})}=\frac{\sqrt{m}}{2}\left<h(r)p,\left|\delta^{m}\right|^{2}\right>_{L^{2}(\mathcal{Y};
\rho^{0})}+\sqrt{m}\left<h(r)p,\delta^{m}\right>_{L^{2}(\mathcal{Y};
\rho^{0})} \ .
\end{align*}
\end{lemma}

\begin{lemma}\label{L:L2_Bounds2a}
 There is a universal constant $K>0$ that depends  on
 $\sup_{r\in\mathbb{T}^{d}}\left|h(r)\right|$, but not on $m>0$, such that for all $m$ sufficiently small
\begin{align*}
&(1-\sqrt{m})\left\|\nabla_{p}\nabla_{p}\delta^{m}\right\|^{2}_{L^{2}(\mathcal{Y};
\rho^{0})}\leq \sqrt{m} K\left[1+
\left\|\delta^{m}\right\|^{2}_{L^{2}(\mathcal{Y}; \rho^{0})}+
\left\| \nabla_{p}\delta^{m}\right\|^{2}_{L^{2}(\mathcal{Y};
\rho^{0})}\right] \ .
 \end{align*}
\end{lemma}

 \begin{lemma}\label{L:L2_Bounds2b}
 There is a universal constant $K>0$ that depends
 on $\sup_{r\in\mathbb{T}^{d}}\max(\left|h(r)\right|, |\grad_r h(r)|)$, but not on $m>0$, such that for all $m$ sufficiently small
\begin{align*}
&(1-\sqrt{m})\left\|\nabla_{p}\nabla_{r}\delta^{m}\right\|^{2}_{L^{2}(\mathcal{Y}; \rho^{0})}\leq\nonumber\\
&\qquad \leq\sqrt{m}
K\left[1+\left\|\delta^{m}\right\|^{2}_{L^{2}(\mathcal{Y};
\rho^{0})}+ \left\|
\nabla_{p}\delta^{m}\right\|^{2}_{L^{2}(\mathcal{Y}; \rho^{0})}+
\left\| \nabla_{r}\delta^{m}\right\|^{2}_{L^{2}(\mathcal{Y};
\rho^{0})}\right] \ .
 \end{align*}
\end{lemma}

Let us define $\cL^{1}$ to be the operator $\cL^{m}_q$ with $m=1$.
We recall that
\[
\cL^{1}=\mathcal{A}+\mathcal{B} \ ,
\]
where $\mathcal{A}=-p\cdot \grad_p+\Dt_p $ and $\mathcal{B}=p \cdot
\grad_r+b(q,r)\cdot\grad_p$. It is easy to check that, with respect to the measure $\rho^{0}(p,r)dpdr$ we can actually write that
\[
\cL^{1}=-AA^{*}+\mathcal{B}
\]
where
\[
A=\nabla_{p} \ , \quad \text{ and }\quad A^{*}=-(\nabla_{p}-p) \ .
\]

One can also check that the adjoint operator of $\mathcal{B}$ is
formally given by
\[
\mathcal{B}^{*}=-\mathcal{B}+p h(r) \ .
\]

Notice that the latter relation implies that $\mathcal{B}$ is
antisymmetric only if $h(r)=0$ which essentially is the case of
gradient drift. However, in the general case $h(r)\neq 0$ which
would imply that $\mathcal{B}$ is not antisymmetric. Next, we introduce the operator
\[
\mathcal{C}=[A,\mathcal{B}]=[\nabla_{p},
p\nabla_{r}+b(r)\nabla_{p}]=\nabla_{r} \ .
\]

A word on notation now. In order to make the notation lighter we
will write from now on
\[
\left\|\cdot\right\|=\left\|\cdot\right\|_{L^{2}(\mathcal{Y};
\rho^{0})}, \text{ and
}\left<\cdot,\cdot\right>=\left<\cdot,\cdot\right>_{L^{2}(\mathcal{Y};
\rho^{0})} \ ,
\]
for the norm and for the inner product in the space
${L^{2}(\mathcal{Y}; \rho^{0})}$.

In order to show that (\ref{Eq:RelationToShow3a}) holds, we  use
the work of \cite{Villani2006}. In particular, as in
\cite{Villani2006}, let $a,b,c$ be constants to be chosen such that
$1>a> b>c>0$ and let us define the norm
\begin{align*}
((f,f))&=\left\|f\right\|^{2}+\alpha\left\|A
f\right\|^{2}+2b\mathcal{R}\left<Af,\mathcal{C}f\right>+c\left\|\mathcal{C}f\right\|^{2}
\ .
\end{align*}

In fact, as it is argued in \cite{Villani2006}, the norms $((f,f))$
and $\left\|f\right\|^{2}_{H^{1}(\mathcal{Y}; \rho^{0})}$ are
equivalent as soon as $b<\sqrt{ac}$, in that
\[
\min\{1,a,c\}\left(1-\frac{b}{\sqrt{ac}}\right)\left\|f\right\|^{2}_{H^{1}(\mathcal{Y};
\rho^{0})}\leq ((f,f))\leq
\max\{1,a,c\}\left(1+\frac{b}{\sqrt{ac}}\right)\left\|f\right\|^{2}_{H^{1}(\mathcal{Y};
\rho^{0})} \ .
\]

Since, we are dealing with a real Hilbert space, all the inner
products are real. By polarization we have
\begin{align*}
((f,\cL^{1}f))&=\left<f,\cL^{1}f\right>+a\left<A f,
A\cL^{1}f\right>+b\left[\left<A\cL^{1}f,\mathcal{C}f\right>+
\left<Af,\mathcal{C}\cL^{1}f\right>\right]+c\left<\mathcal{C}f,\mathcal{C}\cL^{1}f\right>
 \ .
\end{align*}

One important difference between the current setup and the setup of
\cite{Villani2006} is that there $\mathcal{B}^{*}=-\mathcal{B}$
whereas here that is not the case, as we have
$\mathcal{B}^{*}=-\mathcal{B}+p h(r)$. Keeping that in mind and
repeating the argument of the proof of Theorem 18 in
\cite{Villani2006}, we obtain that there are constants $a,b,c$ that
are sufficiently small such that $1\gg a\gg b\gg 2c$ with $b^{2}<ac$
(the exact same constants as in \cite{Villani2006}) such that
\begin{align}
((f,\cL^{1}f))&\geq K\left[\left\|Af\right\|^{2}+\left\|\mathcal{C}f\right\|^{2} \right]\nonumber\\
&+\left\{\left<f,\mathcal{B}f\right>+a\left<A f, \mathcal{B}Af\right>+b\left<p h(r)Af, \mathcal{C}f\right>+c\left<\mathcal{C}f,\mathcal{B}\mathcal{C}f\right>\right\}\nonumber\\
&\geq K
\left\|f\right\|^{2}_{H^{1}}+\left\{\left<f,\mathcal{B}f\right>+a\left<A
f, \mathcal{B}Af\right>+b\left<p h(r),Af
\mathcal{C}f\right>+c\left<\mathcal{C}f,\mathcal{B}\mathcal{C}f\right>\right\}
\ . \label{Eq:NormToBound1}
\end{align}

The  bracket term of the right hand side of the inequality is due to
the fact that in our case $h(r)\neq 0$ and thus $\mathcal{B}$ is not
anti-symmetric. The bracket term is equal to zero in \cite{Villani2006}.

Let us now choose $f=\delta^{m}$ in (\ref{Eq:NormToBound1}). The
strategy of the proof is: (a) bound from below the bracket term on
the right hand side of (\ref{Eq:NormToBound1}) using Lemmas
\ref{L:L2_Bounds3}-\ref{L:L2_Bounds2b} and the equation that
$\delta^{m}$ satisfies, and (b) bound from above the left hand side
of (\ref{Eq:NormToBound1}) using Lemmas
\ref{L:L2_Bounds3}-\ref{L:L2_Bounds2b} and the equation that
$\delta^{m}$ satisfies. Putting the two bounds together one will
then obtain a bound for $\|\delta^{m}\|^{2}_{H^{1}}$ which will give
the convergence to zero of (\ref{Eq:RelationToShow3a}) that we need,
combined with Poincar\'{e} inequality for the measure
$\rho^{0}(p,r)dpdr$.

We would like to highlight here that one of the obstacles in putting the lower and upper bounds together, are the order one terms
$\left<f,\cL^{1}f\right>$ in the definition of $((f,\cL^{1}f))$ and $\left<f,\mathcal{B}f\right>$ in the lower bound (\ref{Eq:NormToBound1}). However, as it turns out, see (\ref{Eq:EquationInvariantDensity3}),
for $f=\delta^{m}$, we actually have that  $\left<\cL^{1}\delta^{m},\delta^{m}\right>-\left<\mathcal{B}\delta^{m},\delta^{m}\right>=o(\sqrt{m})$ which then allows us to proceed with the bounds. The rest of the terms are being handled via Lemmas \ref{L:L2_Bounds3}-\ref{L:L2_Bounds2b}.

We start with obtaining a lower bound for the bracket term on the
right hand side of (\ref{Eq:NormToBound1}) using Lemmas
\ref{L:L2_Bounds2}--\ref{L:L2_Bounds2b} and the equation that
$\delta^{m}$ satisfies.  For this purpose, let us define
\begin{align*}
R(\delta^{m})&=\left<\delta^{m},\mathcal{B}\delta^{m}\right>+a\left<A \delta^{m}, \mathcal{B}A\delta^{m}\right>+b\left<p h(r)A\delta^{m}, \mathcal{C}\delta^{m}\right>+c\left<\mathcal{C}\delta^{m},\mathcal{B}\mathcal{C}\delta^{m}\right>\nonumber\\
&=\left<\delta^{m},\mathcal{B}\delta^{m}\right>+R_{1}(\delta^{m}) \ .
\end{align*}

Let $\eta>0$ to be chosen. By Lemmas
\ref{L:L2_Bounds3}-\ref{L:L2_Bounds4},  recalling that
$A\delta^{m}=\nabla_{p}\delta^{m}$ and
$\mathcal{C}\delta^{m}=\nabla_{r}\delta^{m}$ and using the generalized Cauchy inequality $ab\leq \eta a^{2}+\frac{1}{4\eta}b^{2}$ we have that
\begin{align*}
R_{1}(\delta^{m})&\geq - K\left\{ a\left[\eta \|\nabla_{p}\delta^{m}\|^{2}
+\frac{1}{4\eta}\|\nabla_{p}\nabla_{p}\delta^{m}\|^{2}\right]+
b\left[\eta \|\nabla_{r}\delta^{m}\|^{2}+\frac{1}{4\eta}\| \nabla_{p}\nabla_{p}\delta^{m} \|^{2}\right]\right.\nonumber\\
&\left.+b\left[\eta \|\nabla_{p}\delta^{m}\|^{2}+\frac{1}{4\eta}\|
\nabla_{p}\nabla_{r}\delta^{m}
\|^{2}\right]+c\left[\eta\|\nabla_{r}\delta^{m}\|^{2}+\frac{1}{4\eta}\|\nabla_{p}\nabla_{r}\delta^{m}\|^{2}\right]\right\}
\ .
\end{align*}

Next, using Lemmas \ref{L:L2_Bounds2}, \ref{L:L2_Bounds2a} and
\ref{L:L2_Bounds2b} we subsequently obtain
\begin{align*}
R_{1}(\delta^{m}) &\geq -K\left\{ a\left[\eta \|\nabla_{p}\delta^{m}\|^{2}+\frac{\sqrt{m}}{4\eta}\left(1+\|\delta^{m}\|^{2}+\|\nabla_{p}\delta^{m}\|^{2}\right)\right]\right.\nonumber\\
&\left.\qquad +b\left[\eta \|\nabla_{r}\delta^{m}\|^{2}+\frac{\sqrt{m}}{4\eta}\left(1+\|\delta^{m} \|^{2}+\| \nabla_{p}\delta^{m} \|^{2}\right)\right]\right.\nonumber\\
&\left.\qquad +b\left[\eta \|\nabla_{p}\delta^{m}\|^{2}+\frac{\sqrt{m}}{4\eta}\left(\|\delta^{m} \|^{2}+\|\nabla_{p}\delta^{m} \|^{2}+\|\nabla_{r}\delta^{m} \|^{2}\right)\right]\right.\nonumber\\
&\left.\qquad +c\left[\eta\|\nabla_{r}\delta^{m}\|^{2}+\frac{\sqrt{m}}{4\eta}\left(\|\delta^{m} \|^{2}+\|\nabla_{p}\delta^{m} \|^{2}+\|\nabla_{r}\delta^{m} \|^{2}\right)\right]\right\}\nonumber\\
&\geq -K\left[\eta\left(\|\nabla_{p}\delta^{m}\|^{2}+\|\nabla_{r}\delta^{m}\|^{2}\right)+\frac{\sqrt{m}}{4\eta}\left(1+\|\delta^{m} \|^{2}+\|\nabla_{p}\delta^{m} \|^{2}+\|\nabla_{r}\delta^{m} \|^{2}\right)\right]\nonumber\\
&\geq
-K\left[\eta\|\delta^{m}\|^{2}_{H^{1}}+\frac{\sqrt{m}}{4\eta}\left(1+\|\delta^{m}
\|^{2}+\|\delta^{m} \|^{2}_{H^{1}}\right)\right] \ ,
\end{align*}
where the positive constant $K<\infty$ may change from line to line
but it is always independent of $m$. Choosing now $\eta=\eta(m)$
such that $\lim_{m\downarrow 0}\eta(m)=\lim_{m\downarrow
0}\frac{\sqrt{m}}{\eta(m)}=0$, we obtain for
$\hat{\eta}(m)=\max\{\eta(m),\frac{\sqrt{m}}{\eta(m)}\}\downarrow 0$, that
\begin{align*}
R_{1}(\delta^{m}) &\geq  -K\hat{\eta}(m)\left[1+\|\delta^{m}
\|^{2}+\|\delta^{m} \|^{2}_{H^{1}}\right] \ .
\end{align*}

So, overall we have that for $m$ sufficiently small there is
$\hat{\eta}(m)\downarrow 0$ as $m\downarrow 0$ such that
\begin{align*}
R(\delta^{m})&\geq
\left<\delta^{m},\mathcal{B}\delta^{m}\right>-K\hat{\eta}(m)\left[1+\|\delta^{m}
\|^{2}+\|\delta^{m} \|^{2}_{H^{1}}\right] \ ,
\end{align*}
or in other words by (\ref{Eq:NormToBound1}) with $f=\delta^{m}$ we
have that for $m$ sufficiently small there is
$\hat{\eta}(m)\downarrow 0$ as $m\downarrow 0$ such that
\begin{align}
((\delta^{m},\cL^{1}\delta^{m})) &\geq K (1-\hat{\eta}(m))
\left\|\delta^{m}\right\|^{2}_{H^{1}}+\left<\delta^{m},\mathcal{B}\delta^{m}\right>-K\hat{\eta}(m)\left[1+\|\delta^{m}
\|^{2}\right]  \ . \label{Eq:NormToBound2}
\end{align}

Hence, recalling the definition of the inner product
$((\cdot,\cdot))$, using (\ref{Eq:NormToBound2}) and rearranging the
expression a little bit we have obtained the following bound
\begin{align}
&K (1-\hat{\eta}(m)) \left\|\delta^{m}\right\|^{2}_{H^{1}}\leq \left<\delta^{m},\cL^{1}\delta^{m}\right>-\left<\delta^{m},\mathcal{B}\delta^{m}\right>+K\hat{\eta}(m)\left[1+\|\delta^{m} \|^{2}\right]\nonumber\\
&\quad+a\left<A \delta^{m},
A\cL^{1}\delta^{m}\right>+b\left[\left<A\cL^{1}\delta^{m},\mathcal{C}\delta^{m}\right>+\left<A\delta^{m},\mathcal{C}\cL^{1}\delta^{m}\right>\right]+c\left<\mathcal{C}\delta^{m},\mathcal{C}\cL^{1}\delta^{m}\right>
 \ . \label{Eq:NormToBound3}
\end{align}

The next goal is to derive an appropriate upper bound for the left
hand side of (\ref{Eq:NormToBound3}). First, we need to obtain the
equation that $\delta^{m}$ satisfies.
 By factoring out $\rho^{m}(p,r)=\rho^{0}(p,r)\tilde{\rho}^{m}(p,r)$ where $\rho^{0}(p,r)=\rho^{\text{OU}}(p)\rho_{0}(r)$,
 we obtain the following equation for $\tilde{\rho}^{m}(p,r)$:
\begin{equation*}
\cL^{m}_{q}\tilde{\rho}^{m}(p,r)=\frac{2}{\sqrt{m}}\mathcal{B}\tilde{\rho}^{m}(p,r)-\frac{1}{\sqrt{m}}
p h(r) \tilde{\rho}^{m}(p,r) \ .
\end{equation*}
where we recall that $h(r)= b(r)-\nabla_{r}\log\rho_{0}(r)$. Hence, the equation for $\delta^{m}(p,r)=\tilde{\rho}^{m}(p,r)-1$ is
\begin{equation*}
\cL^{m}_{q}\delta^{m}(p,r)=\frac{2}{\sqrt{m}}\mathcal{B}\delta^{m}(p,r)-\frac{1}{\sqrt{m}}
p h(r) \left[\delta^{m}(p,r)+1\right] \ ,
\label{Eq:EquationInvariantDensity}
\end{equation*}
or in terms of the operator $\cL^{1}=\cL^{m=1}_{q}$ we have
\begin{equation}
\cL^{1}\delta^{m}(p,r)=\left(1+\sqrt{m}\right)\mathcal{B}\delta^{m}(p,r)-\sqrt{m}
p h(r) \left[\delta^{m}(p,r)+1\right] \ .
\label{Eq:EquationInvariantDensity2}
\end{equation}

By multiplying both sides of (\ref{Eq:EquationInvariantDensity2}) by
$\delta^{m}$ and integrating over $\mathcal{Y}$ with respect to the
measure $\rho^{0}(p,r)dpdr$ we then obtain that
\begin{equation}
\left<\cL^{1}\delta^{m},\delta^{m}\right>-\left<\mathcal{B}\delta^{m},\delta^{m}\right>=\sqrt{m}\left<\mathcal{B}\delta^{m},\delta^{m}\right>-\sqrt{m}
\left<p h(r), \left(\delta^{m}+1\right)\delta^{m}\right> \ .
\label{Eq:EquationInvariantDensity3}
\end{equation}

Hence,  using (\ref{Eq:NormToBound3}) and
(\ref{Eq:EquationInvariantDensity3}) we have the following bound
\begin{align}
&K (1-\hat{\eta}(m)) \left\|\delta^{m}\right\|^{2}_{H^{1}}\leq \sqrt{m}\left<\mathcal{B}\delta^{m},\delta^{m}\right>-\sqrt{m} \left<p h(r), \left(\delta^{m}+1\right)\delta^{m}\right>+K\hat{\eta}(m)\left[1+\|\delta^{m} \|^{2}\right]\nonumber\\
&\quad+a\left<A \delta^{m}, A\cL^{1}\delta^{m}\right>+b\left[\left<A\cL^{1}\delta^{m},\mathcal{C}\delta^{m}\right>+\left<A\delta^{m},\mathcal{C}\cL^{1}\delta^{m}\right>\right]+c\left<\mathcal{C}\delta^{m},\mathcal{C}\cL^{1}\delta^{m}\right>\nonumber\\
&\quad\leq K\hat{\eta}(m)\left[1+\|\delta^{m} \|^{2}\right]\nonumber\\
&\quad+\sqrt{m}\left[\left<\mathcal{B}\delta^{m},\delta^{m}\right>-\left<p h(r), \left(\delta^{m}+1\right)\delta^{m}\right>\right]\nonumber\\
&\quad+a\left<A \delta^{m}, A\cL^{1}\delta^{m}\right>+b\left[\left<A\cL^{1}\delta^{m},\mathcal{C}\delta^{m}\right>+\left<A\delta^{m},\mathcal{C}\cL^{1}\delta^{m}\right>\right]+c\left<\mathcal{C}\delta^{m},\mathcal{C}\cL^{1}\delta^{m}\right>\nonumber\\
&\quad\leq K\hat{\eta}(m)\left[1+\|\delta^{m} \|^{2}\right]\nonumber\\
&\quad+T_{1}(\delta^{m})+a T_{2}(\delta^{m})+b T_{3}(\delta^{m})+c
T_{4}(\delta^{m})  \ . \label{Eq:NormToBound4}
\end{align}

Our next goal is to derive upper bounds for the terms
$T_{i}(\delta^{m})$ for $i=1,2,3,4$. For better readability, we
collect the required bounds in the following lemma, which we also
prove in Appendix B.
\begin{lemma}\label{L:AuxiliaryBounds}
Let the terms $T_{i}(\delta^{m})$ for $i=1,2,3,4$ be defined as in
(\ref{Eq:NormToBound4}). Then, there exists a constant $K<\infty$
that does not depend on $m$, and a sequence $\eta(m),\frac{\sqrt{m}}{\eta(m)}\downarrow
0$ as $m\rightarrow 0$ such that for $m$ sufficiently small the
following bounds holds
\begin{align}
\left| T_{1}(\delta^{m})\right|&\leq\sqrt{m} \frac{K+2}{2}\left\|\delta^{m}\right\|^{2} +\sqrt{m}\frac{K}{2}
\left\|\nabla_{p}\delta^{m}\right\|^{2} \  ; \nonumber\\
\left|T_{2}(\delta^{m})\right|&\leq \left(\eta(m)+\sqrt{m}+\frac{\sqrt{m}}{\eta(m)} \right)K \|\delta^{m}\|^{2}_{H^{1}}+
\left(\frac{\sqrt{m}}{\eta(m)}+\sqrt{m}\right)K\left(1+ \left\|\delta^{m}\right\|^{2}\right)\ ; \nonumber\\ 
|T_{3}(\delta^{m})|
&\leq \left(\eta(m)+\sqrt{m}+\frac{\sqrt{m}}{\eta(m)}\right)K\left\|\delta^{m}\right\|^{2}_{H^{1}}+\frac{\sqrt{m}}{\eta(m)}K \left[1+\left\|\delta^{m}\right\|^{2}\right]+
(1+\sqrt{m}) \left\|\nabla_{r}\delta^{m}\right\|^{2} \ ; \nonumber\\
\left|T_{4}(\delta^{m})\right| &\leq \left(\eta(m)+\frac{\sqrt{m}}{\eta(m)}\right)
K\|\delta^{m}\|^{2}_{H^{1}}+\frac{\sqrt{m}}{\eta(m)}\left\|\delta^{m}\right\|^{2} \ . \nonumber
\end{align}
\end{lemma}

Now that we have obtained the desired bounds for the terms
$T_{i}(\delta^{m})$ for $i=1,2,3,4$ let us put them together. 
There are some constants $K_{1},K_{2}<\infty$, and a sequence
$\hat{\eta}(m)=\max\{\eta(m),\frac{\sqrt{m}}{\eta(m)}\}\downarrow 0$ such that for $m$ sufficiently small
\begin{align}
(1-\hat{\eta}(m)) \left\|\delta^{m}\right\|^{2}_{H^{1}}
&\leq \hat{\eta}(m)K_{1}\left[1+\|\delta^{m} \|^{2}\right]+\left[\sqrt{m} K_{1}\left\|\delta^{m}\right\|^{2} +\sqrt{m}K_{1}\left\|\nabla_{p}\delta^{m}\right\|^{2}\right]\nonumber\\
&\quad
+a\left[\hat{\eta}(m) K_{1} \|\delta^{m}\|^{2}_{H^{1}}+ \hat{\eta}(m)K_{1}\left(1+ \left\|\delta^{m}\right\|^{2}\right)\right] \nonumber\\
&\quad +b \left[\hat{\eta}(m)K_{1}\left[1+\left\|\delta^{m}\right\|^{2} +\left\|\delta^{m}\right\|^{2}_{H^{1}}\right]+ (1+\sqrt{m}) \left\|\nabla_{r}\delta^{m}\right\|^{2}\right]\nonumber\\
&\quad +c  \left[\hat{\eta}(m) K_{1}\left[\|\delta^{m}\|^{2}_{H^{1}} +\left\|\delta^{m}\right\|^{2}\right]\right]\nonumber\\
&\leq (\hat{\eta}(m)+\sqrt{m}) K_{2}\left[1+\|\delta^{m} \|^{2}+\left\|\delta^{m}\right\|^{2}_{H^{1}}\right] +b
(1+\sqrt{m}) \left\|\nabla_{r}\delta^{m}\right\|^{2} \ .
\label{Eq:NormToBound5}
\end{align}

Now we choose $m$ small enough such that $\hat{\eta}(m)<1$,
$(\hat{\eta}(m)+\sqrt{m})K_{2}<1/2$. Moreover, we also note that
since by construction $b\ll1$ we can write for $m$ small enough
  $b(1+\sqrt{m})\ll 1/2$. In fact
the proof of \cite{Villani2006} shows that we can choose $a,b,c$ to
be positive but as small as we want, as long we choose the constants $a,b,c$ to
be ordered appropriately. Putting these estimates together, we get
that there is some constant $K_{3}<\infty$ such that for $m$ small enough,
one has
\begin{align}
 \left\|\delta^{m}\right\|^{2}_{H^{1}}
&\leq K_{3} \frac{\hat{\eta}(m)+\sqrt{m}}{1/2-\hat{\eta}(m)}
\left[1+\|\delta^{m} \|^{2}\right] \ . \label{Eq:NormToBound6}
\end{align}

In order now to close the estimate we need to use Poincar\'{e}
inequality. Here we make the assumption that the drift $b(r)$ is
such that the invariant measure $\rho^{0}(p,r)dpdr$ satisfies the
Poincar\'{e} inequality with constant $\kappa>0$ . In particular, for
a function $Q(p,r)$, we have that the Poincar\'{e} inequality in the
following form holds
 \begin{align*}
\left\|Q-\int_{\mathcal{Y}}Q\right\|^{2}_{L^{2}(\mathcal{Y};
\rho^{0})}
 &\leq \kappa \left\|Q\right\|^{2}_{H^{1}(\mathcal{Y}; \rho^{0})} \
 .
\end{align*}

Let us set now $Q(p,r)=\delta^{m}(p,r)$. Notice that by definition
of $\delta^{m}(p,r)$ we have
\[
 \int_{\mathcal{Y}}\delta^{m}(p,r)\rho^{0}(p,r)dpdr=0 \ .
\]

Therefore, we have obtained
\begin{align}
\left\|\delta^{m}\right\|^{2} &\leq \kappa
\left\|\delta^{m}\right\|^{2}_{H^{1}} \ .
\label{Eq:PoincareInequality}
\end{align}

Inserting now (\ref{Eq:PoincareInequality}) into
(\ref{Eq:NormToBound6}), we finally obtain that for $m$ small enough
\begin{align}
\left\|\delta^{m}\right\|^{2} &\leq \frac{K_{3}}{\kappa}
\frac{\hat{\eta}(m)+\sqrt{m}}{1/2-\hat{\eta}(m)}
\left[1+\|\delta^{m} \|^{2}\right]  \ , \label{Eq:NormToBound7}
\end{align}
from which the desired result finally follows:
\begin{align*}
\left\|\delta^{m}\right\|^{2}\leq K_{4}\left(\hat{\eta}(m)+\sqrt{m}\right)\rightarrow 0 \ .
\end{align*}

This concludes the $L^{2}(\mathcal{Y}; \rho^{0})$ convergence
of the invariant measures.

\section{Convergence of the solution to the cell
problem}\label{S:CellProblems}

The goal of this section is to analyze the cell problem
(\ref{Eq:HypoellipticCellProblem}) that $\Phi(p,r)$ satisfies and we want to prove Theorem \ref{T:ConvergenceCellProblems}. As it will become clear from the proof below, we prove even more. We rigorously derive an asymptotic expansion of $\Phi(p,r)$ in terms of powers of $\sqrt{m}$.

Let us recall our assumption $\al(q,r)=2\beta\lambda(q) I$.
Let
$\ell=1,2,...,d$ be a given direction and let us define
\begin{equation*}
\Psi_{\ell}(p,r)=\Phi_{\ell}(p,r)-\frac{\sqrt{m}}{\lambda(q)}p\cdot e_{\ell} \ ,
\end{equation*}
where $e_{\ell}$ is the unit vector in direction $\ell$. Then,
bearing in mind (\ref{Eq:HypoellipticCellProblem}) the equation that
$\Psi_{\ell}(p,r)$ satisfies is given by
\begin{align}
\cL^{m}_q \Psi_{\ell}(p,r)&=- \frac{b_{\ell}(q,r)}{\lambda(q)} \ .
\label{Eq:HypoellipticCellProblemAlternative}
\end{align}

Moreover, by Condition \ref{A:Assumption2} we have that that for
every $m>0$
\[
\int_{\mathcal{Y}}\Psi_{\ell}(p,r)\mu(drdp|q)=0 \ .
\]

Let us write for notational convenience the hypoelliptic operator
\begin{equation*}
 \cL^{m}_q=\frac{\lambda(q)}{m}\mathcal{A}+\frac{1}{\sqrt{m}}\mathcal{B} \ ,
\end{equation*}
where we have already defined $\mathcal{A}=-p\cdot \grad_p +\beta \Dt_p $
and $\mathcal{B}=p \cdot \grad_r+b(q,r)\cdot \grad_p \ $.

Let us now write the expansion
\begin{equation*}
 \Psi_{\ell}(p,r)=\Psi_{\ell,0}(p,r)+\sqrt{m}\Psi_{\ell,1}(p,r)+m
 \Psi_{\ell,2}(p,r)+\Psi^{m}_{\ell,3}(p,r) \
 .
\end{equation*}

Assume that the functions $\Psi_{\ell,0}, \Psi_{\ell,1},
\Psi_{\ell,2}$ and $\Psi_{\ell,3}^m$ satisfy the following equations
\begin{align}
 \mathcal{A}\Psi_{\ell,0}(p,r)&=0 \ , \label{Eq:Term1}\\
 \mathcal{B}\Psi_{\ell,0}(p,r)+\lambda(q)\mathcal{A}\Psi_{\ell,1}(p,r)&=0 \ , \label{Eq:Term2}\\
 \mathcal{B}\Psi_{\ell,1}(p,r)+\lambda(q)\mathcal{A}\Psi_{\ell,2}(p,r)&=-b_{\ell}(q,r) \ ,
 \label{Eq:Term3}\\
 \cL^{m}_q\Psi^{m}_{\ell,3}(p,r)&=-\sqrt{m}\mathcal{B}\Psi_{\ell,2}(p,r) \
. \label{Eq:Termk}
\end{align}
and that for $i=0,1,2,3$, we have that
$\play{\int_{\mathcal{Y}}\Psi_{\ell,i}(p,r)\rho^{m}(p,r|q)=0}$. The
next step is to analyze the solutions to
(\ref{Eq:Term1})--(\ref{Eq:Termk}). First we notice that
(\ref{Eq:Term1}) basically implies that
$\Psi_{\ell,0}(p,r)=\Psi_{\ell,0}(r)$, i.e., function
$\Psi_{\ell,0}(r)$ is a function of $r$ alone. Then, using this we
get by (\ref{Eq:Term2}) that it has to be the case that
\[
 \Psi_{\ell,1}(p,r)=\frac{1}{\lambda(q)}\nabla_{r}\Psi_{\ell,0}(r)\cdot p+ \hat{\Psi}_{\ell,0}(r)
\]
for some function $\hat{\Psi}_{\ell,0}(r)$. From equation
(\ref{Eq:HypoellipticCellProblemAlternative}) and (\ref{Eq:Term1}),
(\ref{Eq:Term2}), (\ref{Eq:Term3}), (\ref{Eq:Termk}) we see that up
to an additive constant we can assume that
$\hat{\Psi}_{\ell,0}(r)=0$. Lastly, we notice that the solvability
condition for (\ref{Eq:Term3}) is
\begin{align*}
&\int_{\R^d} \left[\mathcal{B}\Psi_{\ell,1}(p,r)+b_{\ell}(q,r)\right] \pi(dp)=0\Rightarrow\nonumber\\
&\int_{\R^d} \left[\mathcal{B}(\nabla_{r}\Psi_{\ell,0}(r)\cdot p)+ \mathcal{B}\hat{\Psi}_{\ell,0}(r)+b_{\ell}(q,r)\right] \pi(dp)=0\Rightarrow\nonumber\\
&\int_{\R^d} \left[\Delta_r \Psi_{\ell,0}(r)
|p|^{2}+b(q,r)\cdot\nabla_{r}\Psi_{\ell,0}(r)+ p\cdot
\grad_{r}\hat{\Psi}_{\ell,0}(r)
+b_{\ell}(q,r)\right] \pi(dp)=0\Rightarrow\nonumber\\
&\beta \Delta_r \Psi_{\ell,0}(r)+ b(q,r)\cdot
\nabla_{r}\Psi_{\ell,0}(r)=-b_{\ell}(q,r) \ , \label{Eq:Solvability
condition}
\end{align*}
where the Gaussian structure of the invariant measure $\pi(dp)\sim
e^{-\frac{|p|^2}{2\beta}}dp$ and integration by parts were used. Notice
that this is exactly the solution to
(\ref{Eq:CellProblemOverdampedCase}) with $\al=2\beta\lambda(q)I$.
Thus, by uniqueness of the solution to
(\ref{Eq:CellProblemOverdampedCase}) we basically have that for
every $\ell=1,\cdots, d$
\begin{equation*}
 \Psi_{\ell,0}(r)=\chi_{\ell}(r).\label{Eq:FirstTermInExpansionHypoellipticCellProblem}
\end{equation*}

Hence, we have that
\begin{align*}
 \nabla_{p}\Phi(p,r)&=\nabla_{p}\left(\frac{\sqrt{m}}{\lambda(q)}p +\Psi(p,r)\right)\nonumber\\
 &=\nabla_{p}\left(\frac{\sqrt{m}}{\lambda(q)}p +\Psi_{0}(p,r)+\sqrt{m}\Psi_{1}(p,r)+m\Psi_{2}(p,r)+
    \Psi^{m}_{3}(p,r)\right)\nonumber\\
 &=\frac{\sqrt{m}}{\lambda(q)} \left[I+\nabla_{r}\chi(r)\right]+
 m\nabla_{p}\Psi_{2}(p,r)+\nabla_{p}\Psi^{m}_{3}(p,r) \ .
\end{align*}

Having established the last display, it is easy to see that in order
to show (\ref{Eq:RelationToShow1}), we basically need to show  that
\begin{align*}
\lim_{m\rightarrow0}\left\|\sqrt{m}\nabla_{p}\Psi_{2}(p,r)+\frac{1}{\sqrt{m}}\nabla_{p}\Psi^{m}_{3}(p,r)\right\|_{L^{2}
(\mathcal{Y}; \rho^{m})}&=0 \ , 
\end{align*}
or, in other words, it is sufficient to show
\begin{align}
\lim_{m\rightarrow0}\left\|\sqrt{m}\nabla_{p}\Psi_{2}(p,r)\right\|_{L^{2}(\mathcal{Y};
\rho^{m})}&=0
 \ , \label{Eq:RelationToShow1a1}
\end{align}
and
\begin{align}
\lim_{m\rightarrow0}\left\|\frac{1}{\sqrt{m}}\nabla_{p}\Psi^{m}_{3}(p,r)\right\|_{L^{2}(\mathcal{Y};
\rho^{m})}&=0
 \ . \label{Eq:RelationToShow1a2}
\end{align}


Relation (\ref{Eq:RelationToShow1a1})  can be claimed to be true by
the fact that $\Psi_{2}(p,r)$ is solution to the elliptic problem
(\ref{Eq:Term3})  and Theorem \ref{T:ConvergenceInvariantMeasures}.

So, it remains to prove (\ref{Eq:RelationToShow1a2}). At this point
let us recall that $\Psi^{m}_{\ell,3}(p,r)$ is solution to
(\ref{Eq:Termk}), i.e., it solves
\begin{align}
\cL^{m}_q\Psi^{m}_{\ell,3}(p,r)&=-\sqrt{m}\mathcal{B}\Psi_{\ell,2}(p,r)
\ . \label{Eq:ErrorEquation}
\end{align}

Notice that for the purposes of this section $q$ is seen as a fixed parameter by the operators and recall that we have already assumed $\al(q,r)=2\beta\lambda(q) I$. Namely $\beta\lambda(q)$ is seen as a fixed constant. Hence, from now on and for notational convenience, we shall assume without loss of generality that $\alpha(q,r)=2I$, i.e., that $\beta=\lambda(q)=1$.
Let us first apply Lemma \ref{L:L2_Bounds1} and we get
\begin{align}
\int_{\cY}(\cL_q^m \Psi_{\ell,3}^m)\Psi_{\ell,3}^m\rho^0(p,r)dpdr &
=-\dfrac{1}{m}\int_{\cY}|\grad_p\Psi_{\ell,3}^m|^2\rho^0(p,r)dpdr+\dfrac{1}{2\sqrt{m}}\int_{\cY}(\Psi_{\ell,3}^m)^2h(r)p\rho^0(p,r)dpdr
\ . \label{Eq:mZeroIntegrationByParts}
\end{align}

Lemmas \ref{mZeroPoincareInequality}-\ref{mNonZeroIntegrationByParts} that follow are proven in Appendix \ref{S:CellProblemsProof}.
\begin{lemma}\label{mZeroPoincareInequality}
We have the Poincar\'{e} inequality
\begin{align}
\left\|f-\int_{\mathcal{Y}}f(p,r)\rho^0(p,r)dpdr\right\|_{L^2(\mathcal{Y};
\rho^{0})}\leq \kappa\left\|\grad_p f\right\|_{L^2(\mathcal{Y};
\rho^{0})} \ ,
\end{align}
for some constant $\kappa>0$ independent of $m$.
\end{lemma}
\begin{lemma}\label{mZeroL2bound1}
We have \begin{align}\limsup\li_{m\ra
0}\dfrac{\|\Psi_{\ell,3}^m\|_{L^2(\cY;\rho^0)}}{m^{3/2}}\leq
C<\infty \ ,
 \limsup\li_{m\ra 0}\dfrac{\left\|\dfrac{1}{\sqrt{m}}\grad_p \Psi_{\ell,3}^m\right\|_{L^2(\cY;\rho^0)}}{m}\leq C<\infty\end{align}
for some constant $C>0$ independent of $m$.
\end{lemma}
\begin{lemma}\label{mZeroL4bound}
We have
\begin{align}
\lim\li_{m\ra 0}\|\Psi_{\ell,3}^m\|_{L^4(\cY;\rho^0)}=0 \ .
\end{align}
\end{lemma}
\begin{lemma}\label{mNonZeroIntegrationByParts}
We have \begin{align}
\int_{\mathcal{Y}}\left(\cL^{m}_{q}f(p,r)\right) g(p,r) \rho^{m}(p,r)dpdr&+\int_{\mathcal{Y}}f(p,r) \left(\cL^{m}_{q}g(p,r)\right)\rho^{m}(p,r)dpdr=\nonumber\\
&\quad=-\frac{2}{m}\int_{\mathcal{Y}}\left[\nabla_{p}f(p,r)\cdot
\alpha(q,r)\nabla_{p}g(p,r)\right]\rho^{m}(p,r)dpdr \ .
\label{Eq:mNonZeroIntegrationByParts}
\end{align}
\end{lemma}

We set in particularly in (\ref{Eq:mZeroIntegrationByParts})
$f=g=\Psi^{m}_{\ell,3}$, then we will have that
%
\begin{align*}
&\int_{\mathcal{Y}}\left(\cL^{m}_{q}\Psi^{m}_{\ell,3}(p,r)\right)
\Psi^{m}_{\ell,3}(p,r)
\rho^{m}(p,r)dpdr=-\frac{2}{m}\int_{\mathcal{Y}}|\nabla_{p}\Psi^{m}_{\ell,3}(p,r)|^{2}\rho^{m}(p,r)dpdr
\ . 
\end{align*}

But, we also know that $\Psi^{m}_{\ell,3}(p,r)$ satisfies
(\ref{Eq:ErrorEquation}). Therefore, multiplying both sides of
(\ref{Eq:ErrorEquation}) by $\Psi^{m}_{\ell,3}(p,r)$ and integrating
against the invariant density $\rho^{m}(p,r)$ gives us the identity
\begin{align*}
\frac{2}{m}\int_{\mathcal{Y}}|\nabla_{p}\Psi^{m}_{\ell,3}(p,r)|^{2}\rho^{m}(p,r)dpdr&=
\sqrt{m}\int_{\mathcal{Y}}\left(\mathcal{B}\Psi_{\ell,2}(p,r)\right)\Psi^{m}_{\ell,3}(p,r)\rho^{m}(p,r)dpdr
\ ,
\end{align*}
or, in other words
\begin{align}
\left\|\frac{1}{\sqrt{m}}\nabla_{p}\Psi^{m}_{\ell,3}\right\|^{2}_{L^{2}(\mathcal{Y};
\rho^{m})}& =\dfrac{\sqrt{m}}{2}\left<\mathcal{B}\Psi_{\ell,2},
\Psi^{m}_{\ell,3}\right>_{L^{2}(\mathcal{Y}; \rho^{m})} \ .
\label{Eq:GradientCorrector1}
\end{align}

We now have the estimate
\begin{align}
\|\Psi_{\ell,3}^m\|^2_{L^2(\cY; \rho^m)} &=\int_{\cY}(\Psi_{\ell,3}^m)^2 \rho^m(p,r)dpdr \nonumber
\\
&=\int_{\cY}(\Psi_{\ell,3}^m)^2
\dt^m(p,r)\rho^0(p,r)dpdr+\int_{\cY}(\Psi_{\ell,3}^m)^2
\rho^0(p,r)dpdr \nonumber
\\
&\leq \|\Psi_{\ell,3}^m\|^4_{L^4(\cY; \rho^0)}\|\dt^m\|^2_{L^2(\cY;
\rho^0)}+\|\Psi_{\ell,3}^m\|_{L^2(\cY; \rho^0)}^2 \ . \nonumber
\end{align}

Applying Lemma \ref{mZeroL4bound} and Lemma \ref{mZeroL2bound1} and the
fact that $\lim\li_{m\ra 0}\|\dt^m\|_{L^2(\cY; \rho^0)}=0$ we see
that
\begin{align*}
\lim\li_{m\ra 0}\|\Psi_{\ell,3}^m\|_{L^2(\cY;
\rho^m)}=0 \ .
\end{align*}

Thus we have by (\ref{Eq:GradientCorrector1})
\begin{align*}
\left\|\dfrac{1}{\sqrt{m}}\grad_p \Psi_{\ell,3}^m\right\|_{L^2(\cY;
\rho^m)}^2\leq \dfrac{\sqrt{m}}{2}\|\cB \Psi_{\ell,2}\|_{L^2(\cY,
\rho^m)}\|\Psi_{\ell,3}^m\|_{L^2(\cY; \rho^m)}\ra 0
\end{align*}
as $m\ra 0$. This is (\ref{Eq:RelationToShow1a2}), completing the proof of Theorem \ref{T:ConvergenceCellProblems}.

\appendix

\section{On properties of the solution to the hypoelliptic cell problem}\label{S:AppendixA}
In this section we recall some results on the solution to the hypoelliptic Poisson equation (\ref{Eq:HypoellipticCellProblem}) from \cite{HairerPavliotis2004}. Since the set-up of the current paper has some differences from the setup in \cite{HairerPavliotis2004}, we formulate the results that we need in the current setup, even though we emphasize that the derivation follows basically from \cite{HairerPavliotis2004}.

Under the assumptions made in this paper, Theorem $3.3$ from \cite{HairerPavliotis2004} guarantees that,
(\ref{Eq:HypoellipticCellProblem}) has a smooth solution that does not grow too fast at infinity.
In particular, we have that for every $\eta>0$, we can write
\[
\Phi(p,r)=e^{\frac{\eta}{2}|p|^{2}}\tilde{\Phi}(p,r)
\]
where $\tilde{\Phi}\in\mathcal{S}$, the Schwartz space of smooth
functions with fast decay. Furthermore, as it can be derived from
the proof of Theorem $3.3$ of  \cite{HairerPavliotis2004}, if we let
$\sigma_{\text{max}}=\max_{i,j=1,\cdots
d}\sup_{(q,r)}|\sigma_{i,j}(q,r)|$, then we have that for every
$\eta\in(0,2\sigma_{\text{max}}^{-2})$ the solution $\Phi$ is unique
(up to additive constants) in the space
$L^{2}\left(\mathcal{Y},e^{-\eta|p|^{2}}dpdr\right)$.

Moreover, it is clear that for each fixed $q$, the operator $\cL^{m}_q$ defines a hypoelliptic
diffusion process on $(p,r)\in\cY=\R^d\times \T^d$. Let us define this process by $(p_\cdot,r_\cdot)$.
We recall then the following useful bounds from \cite{HairerPavliotis2004}.
\begin{enumerate}
\item{There exists a constant $C$ such that
\[
\E\left[e^{\frac{\sigma_{\text{max}}^{-2}}{2}|p_{t}|^{2}}\right]<
\E\left[e^{\frac{\sigma_{\text{max}}^{-2}}{2}|p_{o}|^{2}+Ct}\right]
\]}
\item{For every $T>0$, there exist constants $\eta,C>0$ such that
\[
\E\left[\sup_{t\in[0,T]}e^{\eta|p_{t}|^{2}}\right]<C
\E\left[e^{\eta|p_{o}|^{2}}\right]
\]
}
\end{enumerate}

Based then on these bounds, the computations of  \cite{HairerPavliotis2004} reveal
that the following bounds for the solution to (\ref{Eq:HypoellipticCellProblem}) are true.
In particular we have that for every $T,p>0$ there exists a constant $C>0$ that is independent of $\varepsilon,\delta$ such that
$\E\left[\sup_{t\in[0,T]}\left|\Phi\left(\bar{p}^{\varepsilon}_{t},\frac{\bar{q}^{\varepsilon}_{t}}{\delta}\right)\right|^{p}\right]\leq C \delta^{-p/2},$
and
$\E\left[\sup_{t\in[0,T]}\left|\nabla_{p}\Phi\left(\bar{p}^{\varepsilon}_{t},\frac{\bar{q}^{\varepsilon}_{t}}{\delta}\right)\right|^{p}\right]\leq
C.$

These bounds are used in the proofs of this paper and more specifically in the derivation of Theorems \ref{T:LLN} and \ref{T:MainLaplacePrinciple}.

\section{Proofs of Lemmas in Section \ref{S:InvariantMeasures}.}\label{S:InvariantMeasuresProof}

\begin{proof}[Proof of Lemma \ref{L:L2_Bounds1}]
 The proof goes in a standard way using integration by parts. We present the main steps for completeness.
\begin{align}
&\int_{\mathcal{Y}}\left(\cL^{m}_{q}f(p,r)\right) g(p,r) \rho^{0}(p,r)dpdr=\int_{\mathcal{Y}}f(p,r) \left(\cL^{m}_{q}\right)^{*}\left(g(p,r) \rho^{0}(p,r)\right)dpdr\nonumber\\
&=\int_{\mathcal{Y}}f(p,r) \left[\frac{1}{m}\left(\mathcal{A}\right)^{*}\left(\rho^{0}(p,r)\right)g(p,r)+ \left(-\cL^{m}_{q}g(p,r)\right)\rho^{0}(p,r)\right]dpdr\nonumber\\
&\quad+\frac{2}{m}\int_{\mathcal{Y}}f(p,r) \left[\left(\rho^{0}(p,r) I:\nabla^{2}_{p}g(p,r)+\nabla_{p}\rho^{0}(p,r)\cdot I\nabla_{p}g(p,r)\right)\right]dpdr\nonumber\\
&\qquad- \frac{1}{\sqrt{m}}\int_{\mathcal{Y}}f(p,r)g(p,r)\left[b(r)\nabla_{p}\rho^{0}(p,r)+p\nabla_{r}\rho^{0}(p,r)\right]dpdr\nonumber\\
&=\int_{\mathcal{Y}}f(p,r) \left(-\cL^{m}_{q}g(p,r)\right)\rho^{0}(p,r)dpdr\nonumber\\
&\quad+\frac{2}{m}\int_{\mathcal{Y}}f(p,r)\left(\rho^{0}(p,r) I:\nabla^{2}_{p}g(p,r)+\nabla_{p}\rho^{0}(p,r)\cdot I \nabla_{p}g(p,r)\right)dpdr\nonumber\\
&\qquad+ \frac{1}{\sqrt{m}}\int_{\mathcal{Y}}f(p,r)g(p,r)p\cdot h(r)\rho^{0}(p,r)dpdr\nonumber\\
&=\int_{\mathcal{Y}}f(p,r) \left(-\cL^{m}_{q}g(p,r)\right)\rho^{m}(p,r)dpdr\nonumber\\
&\quad-\frac{2}{m}\int_{\mathcal{Y}}\left[\nabla_{p}f(p,r)\cdot I
\nabla_{p}g(p,r)\right]\rho^{0}(p,r)dpdr+
\frac{1}{\sqrt{m}}\int_{\mathcal{Y}}f(p,r)g(p,r)p\cdot
h(r)\rho^{0}(p,r)dpdr \ . \nonumber
\end{align}
To derive the last line, we used integration by parts as well as the
definition $h(r)= b(r)-\nabla_{r}\log\rho_{0}(r)$. The statement of
the lemma  follows.
\end{proof}

\begin{proof}[Proof of Lemma \ref{L:L2_Bounds3}]
We start with the following calculation
\begin{align*}
\nabla_{p}\left(e^{-\frac{1}{2}|p|^{2}}f(p,r)g(p,r)\right)&=-pe^{-\frac{1}{2}|p|^{2}}f(p,r)g(p,r)+e^{-\frac{1}{2}|p|^{2}}\nabla_{p}\left(f(p,r)g(p,r)\right)
 \ . \end{align*}

Therefore, we obtain
\begin{align*}
\int_{\mathbb{R}^{d}}pf(p,r)g(p,r) \rho^{\text{OU}}(p)dp&=\int_{\mathbb{R}^{d}}\nabla_{p}(f(p,r)g(p,r)) \rho^{\text{OU}}(p)dp\nonumber\\
&=\int_{\mathbb{R}^{d}}\left(\nabla_{p}f(p,r) g(p,r)+f(p,r)
\nabla_{p}g(p,r) \right)\rho^{\text{OU}}(p)dp  \ .
\end{align*}

Multiplying both sides by $h(r)\rho_{0}(r)$ and integrating over
$r\in \mathbb{T}^{d}$ we then obtain after using H\"{o}lder
inequality
\begin{align*}
\left<h(r)p,fg\right>_{L^{2}(\mathcal{Y}; \rho^{0})}&=\int_{\mathcal{Y}}\left(\nabla_{p}f(p,r) g(p,r)+f(p,r) \nabla_{p}g(p,r) \right) \rho^{0}(p,r)dpdr\nonumber\\
&\leq K \left[\left\|f\right\|_{L^{2}(\mathcal{Y};
\rho^{0})}\left\|\nabla_{p}g\right\|_{L^{2}(\mathcal{Y};
\rho^{0})}+\left\|\nabla_{p}f\right\|_{L^{2}(\mathcal{Y};
\rho^{0})}\left\|g\right\|_{L^{2}(\mathcal{Y}; \rho^{0})}\right]
 \ . \end{align*}
This completes the statement of the lemma.
\end{proof}

\begin{proof}[Proof of Lemma \ref{L:L2_Bounds4}]
We notice that
\begin{align*}
 \left<f,\mathcal{B}f\right>_{L^{2}(\mathcal{Y}; \rho^{0})}&=\int_{\mathcal{Y}}p\nabla_{r}f(p,r) f(p,r)\rho^{0}(p,r)dpdr\nonumber\\
 &\quad+
 \int_{\mathcal{Y}}b(r)\nabla_{p}f(p,r) f(p,r)\rho^{0}(p,r)dpdr\nonumber\\
 &=\text{Term1}_{m}+\text{Term2}_{m} \ .
\end{align*}

By integration by parts, we have
\begin{align*}
\text{Term1}_{m}&= \int_{\mathcal{Y}}p\nabla_{r}f(p,r) f(p,r)\rho^{0}(p,r)dpdr\nonumber\\
&=-\int_{\mathcal{Y}}pf(p,r)
\nabla_{r}f(p,r)\rho^{0}(p,r)dpdr-\int_{\mathcal{Y}} p
\nabla_{r}\log\rho_{0}(r) \left|f(p,r)\right|^{2}  \rho^{0}(p,r)dpdr
\ .
\end{align*}

Thus, we get
\begin{align*}
\text{Term1}_{m}&= \int_{\mathcal{Y}}p\nabla_{r}f(p,r) f(p,r)\rho^{0}(p,r)dpdr=-\frac{1}{2}\int_{\mathcal{Y}} p \nabla_{r}\log\rho_{0}(r)
\left|f(p,r)\right|^{2}  \rho^{0}(p,r)dpdr \ .
\end{align*}

Similarly, we have
\begin{align*}
\text{Term2}_{m}&= \int_{\mathcal{Y}}b(r)\nabla_{p}f(p,r) f(p,r)\rho^{0}(p,r)dpdr\nonumber\\
&=-\int_{\mathcal{Y}}b(r)f(p,r)
\nabla_{p}f(p,r)\rho^{0}(p,r)dpdr+\int_{\mathcal{Y}} b(r) p
\left|f(p,r)\right|^{2}  \rho^{0}(p,r)dpdr \ .
\end{align*}

Thus, we get
\begin{align*}
\text{Term2}_{m}&= \int_{\mathcal{Y}}b(r)\nabla_{p}f(p,r) f(p,r)\rho^{0}(p,r)dpdr=\frac{1}{2}\int_{\mathcal{Y}} b(r) p \left|f(p,r)\right|^{2}
\rho^{0}(p,r)dpdr \ .
\end{align*}

Putting the representations of $\text{Term1}_{m}$ and
$\text{Term2}_{m}$ together, we have in fact obtained
\begin{align*}
\left<f,\mathcal{B}f\right>_{L^{2}(\mathcal{Y};
\rho^{0})}&=\frac{1}{2} \int_{\mathcal{Y}} p h (r)
\left|f(p,r)\right|^{2}  \rho^{0}(p,r)dpdr \ .
\end{align*}

Hence, by Lemma \ref{L:L2_Bounds3} we have that there exists a
constant $K<\infty$ that depends on
$\sup_{r\in\mathbb{T}^{d}}\left|h(r)\right|$ such that
\begin{align*}
\left<f,\mathcal{B} f\right>_{L^{2}(\mathcal{Y}; \rho^{0})}=\frac{1}{2}\left<p h(r), |f|^{2}\right>_{L^{2}(\mathcal{Y}; \rho^{0})}&\geq -K\left\|f\right\|_{L^{2}(\mathcal{Y}; \rho^{0})}\left\|\nabla_{p}f\right\|_{L^{2}(\mathcal{Y}; \rho^{0})}\nonumber\\
&\geq -K \left[\eta \left\|f\right\|^{2}_{L^{2}(\mathcal{Y};
\rho^{0})}
+\frac{1}{4\eta}\left\|\nabla_{p}f\right\|^{2}_{L^{2}(\mathcal{Y};
\rho^{0})}\right] \ .
\end{align*}
where we use the generalized Cauchy-Schwarz inequality $ab\leq \eta
|a|^{2}+\frac{1}{4\eta}|b|^{2}$ for any $\eta\in(0,\infty)$. This
concludes the proof of the lemma.
\end{proof}

\begin{proof}[Proof of Lemma \ref{L:L2_Bounds2}]
Recall that by (\ref{Eq:EquationInvariantDensity}), the equation for
$\delta^{m}(p,r)=\tilde{\rho}^{m}(p,r)-1$ is
\begin{align}
\cL^{m}_{q}\delta^{m}(p,r)=\frac{2}{\sqrt{m}}\mathcal{B}\delta^{m}(p,r)-\frac{1}{\sqrt{m}}
p h(r) \left[\delta^{m}(p,r)+1\right] \ . \label{Eq:deltam equation}
\end{align}

Let us multiply now the last equation by $\delta^{m}(p,r)$ and
integrate over $\mathcal{Y}$ against $\rho^{0}(p,r)$. Doing so, we
get
\begin{align*}
 \left<\cL^{m}_{q}\delta^{m},\delta^{m}\right>_{L^{2}(\mathcal{Y}; \rho^{0})}&=\frac{2}{\sqrt{m}}\left<\mathcal{B}\delta^{m},\delta^{m}\right>_{L^{2}(\mathcal{Y}; \rho^{0})} -\frac{1}{\sqrt{m}} \left<p h(r) \left[\delta^{m}(p,r)+1\right],\delta^{m}\right>_{L^{2}(\mathcal{Y}; \rho^{0})}\label{Eq:RepresentationInvMeas1}
 \ . \end{align*}

The next step is to rewrite the term
$\left<\mathcal{B}\delta^{m},\delta^{m}\right>_{L^{2}(\mathcal{Y};
\rho^{0}(p,r))}$. By Lemma \ref{L:L2_Bounds4}
 we have
\begin{align*}
\left<\mathcal{B}f,f\right>_{L^{2}(\mathcal{Y};
\rho^{0})}&=\frac{1}{2} \int_{\mathcal{Y}} p h (r)
\left|f(p,r)\right|^{2}  \rho^{0}(p,r)dpdr \ .
\end{align*}

Inserting the latter expression into
(\ref{Eq:RepresentationInvMeas1}) we obtain
\begin{align*}
 \left<\cL^{m}_{q}f,f\right>_{L^{2}(\mathcal{Y}; \rho^{0})}&=\frac{2}{\sqrt{m}}\left<\mathcal{B}f,f\right>_{L^{2}(\mathcal{Y}; \rho^{0})}\nonumber\\
 &\quad -\frac{1}{\sqrt{m}} \left<p h(r) \left[f(p,r)+1\right],f\right>_{L^{2}(\mathcal{Y}; \rho^{0})}\nonumber\\
 &=\frac{1}{\sqrt{m}}\left<p h(r),|f|^{2}\right>_{L^{2}(\mathcal{Y}; \rho^{0})}\nonumber\\
 &\quad -\frac{1}{\sqrt{m}} \left<p h(r) \left[f(p,r)+1\right],f\right>_{L^{2}(\mathcal{Y}; \rho^{0})}\nonumber\\
  &= -\frac{1}{\sqrt{m}} \left<p h(r),f\right>_{L^{2}(\mathcal{Y};
  \rho^{0})} \ .
\end{align*}

Next step is to apply Lemma \ref{L:L2_Bounds1} with
$f(p,r)=g(p,r)=\delta^{m}(p,r)$ to get
 \begin{align*}
  &\left<\cL^{m}_{q}\delta^{m},\delta^{m}\right>_{L^{2}(\mathcal{Y}; \rho^{0})}=-\frac{1}{m}\left\|\nabla_{p}\delta^{m}\right\|^{2}_{L^{2}(\mathcal{Y}; \rho^{0})}+
  \frac{1}{2\sqrt{m}}\left<h(r)p,\left|\delta^{m}\right|^{2}\right>_{L^{2}(\mathcal{Y}; \rho^{0})}
    \ . \end{align*}

Combining the last two expressions, we obtain
\begin{align}
-\frac{1}{m}\left\|\nabla_{p}\delta^{m}\right\|^{2}_{L^{2}(\mathcal{Y};
\rho^{0})}&+
  \frac{1}{2\sqrt{m}}\left<h(r)p,\left|\delta^{m}\right|^{2}\right>_{L^{2}(\mathcal{Y}; \rho^{0})}= -\frac{1}{\sqrt{m}} \left<p h(r),\delta^{m}\right>_{L^{2}(\mathcal{Y};
  \rho^{0})}   \ ,
\end{align}
and after rearranging, we obtain
\begin{align*}
\left\|\nabla_{p}\delta^{m}\right\|^{2}_{L^{2}(\mathcal{Y};
\rho^{0})}&=
  \frac{\sqrt{m}}{2}\left<h(r)p,\left|\delta^{m}\right|^{2}\right>_{L^{2}(\mathcal{Y}; \rho^{0})}  +\sqrt{m} \left<p h(r),\delta^{m}\right>_{L^{2}(\mathcal{Y};
  \rho^{0})} \ .
\end{align*}

This concludes the proof of the lemma.
\end{proof}

\begin{proof}[Proof of Lemma \ref{L:L2_Bounds2a}]
The proof goes along the same lines of Lemma \ref{L:L2_Bounds2}. We
take $\pt_{p_i}$ on both sides of the equation (\ref{Eq:deltam
equation}) and we get the following equation
\begin{equation*}
\cL_q^m\pt_{p_i}\dt^m= \dfrac{2}{\sqrt{m}}\cB\pt_{p_i}\dt^m-
\dfrac{1}{\sqrt{m}}h_i(r)[\dt^m+1]- \dfrac{1}{\sqrt{m}}p\cdot
h(r)\pt_{p_i}\dt^m \ .
\end{equation*}

Multiplying both sides of the above equation by $\pt_{p_i}\dt^m$ and
integrate with respect to $L^2(\cY; \rho^0)$--inner product we get
\begin{equation*}\begin{array}{ll} \langle\cL_q^m \pt_{p_i}\dt^m,
\pt_{p_i}\dt^m\rangle_{L^2(\cY;
\rho^0)}=&\dfrac{2}{\sqrt{m}}\langle\cB \pt_{p_i}\dt^m,
\pt_{p_i}\dt^m\rangle_{L^2(\cY; \rho_0)}\\&-\dfrac{1}{\sqrt{m}}\langle
h_i(r)(\dt^m+1), \pt_{p_i}\dt^m\rangle_{L^2(\cY; \rho_0)}\\ &
-\dfrac{1}{\sqrt{m}}\langle p\cdot h(r)\pt_{p_i}\dt^m,
\pt_{p_i}\dt^m\rangle_{L^2(\cY; \rho_0)}
 \ .
\end{array} \end{equation*}

We apply Lemma \ref{L:L2_Bounds1} with
$f(p,r)=g(p,r)=\pt_{p_i}\dt^m(p,r)$ to get
\begin{equation*}
\begin{array}{ll} \langle\cL_q^m \pt_{p_i}\dt^m,
\pt_{p_i}\dt^m\rangle_{L^2(\cY; \rho^0)}=&
-\dfrac{1}{m}\|\grad_p\pt_{p_i}\dt^m\|_{L^2(\cY;
\rho^0)}^2+\dfrac{1}{2\sqrt{m}}\langle h(r)p,
|\pt_{p_i}\dt^m|^2\rangle_{L^2(\cY; \rho^0)} \ .
\end{array}
\end{equation*}

We now apply Lemma \ref{L:L2_Bounds4} and we have
\begin{equation*}\begin{array}{ll}
\langle\cB \pt_{p_i}\dt^m, \pt_{p_i}\dt^m\rangle_{L^2(\cY; \rho^0)}
=&\play{\dfrac{1}{2}\int_{\cY}p\cdot
h(r)|\pt_{p_i}\dt^m|^2\rho^0(p,r)dpdr} \ .
\end{array} \end{equation*}

Furthermore, we can calculate
\begin{equation*}\begin{array}{ll}
\langle h_i(r)(\dt^m+1), \pt_{p_i}\dt^m\rangle_{L^2(\cY; \rho^0)}
=&\play{\int_{\cY}h_i(r)(\dt^m+1)\pt_{p_i}\dt^m \rho^0(p,r)dpdr} \ ,
\end{array}
\end{equation*}
\begin{equation*}\begin{array}{ll}
\langle p\cdot h(r)\pt_{p_i}\dt^m, \pt_{p_i}\dt^m\rangle_{L^2(\cY;
\rho^0)} =& \play{\int_{\cY} p\cdot h(r) |\pt_{p_i}\dt^m|^2
\rho^0(p,r) dpdr} \ .
\end{array}
\end{equation*}

Thus, we get the identity
\begin{equation*}\begin{array}{ll}
-\dfrac{1}{m}\|\grad_p\pt_{p_i}\dt^m\|^2_{L^2(\cY;\rho^0)}=&
-\dfrac{1}{2\sqrt{m}}\langle ph(r),
|\pt_{p_i}\dt^m|^2\rangle_{L^2(\cY;
\rho^0)}\nonumber\\
&\quad-\dfrac{1}{\sqrt{m}}\langle h_i(r)(\dt^m+1),
\pt_{p_i}\dt^m\rangle_{L^2(\cY;\rho^0)} \ .
\end{array}
\end{equation*}

Making use of Lemma \ref{L:L2_Bounds3} and Young's inequality we estimate
\begin{equation*}\begin{array}{ll}
\|\grad_p\pt_{p_i}\dt^m\|^2_{L^2(\cY;\rho^0)}\leq &
K\sqrt{m}\|\grad_p \dt^m\|^2_{L^2(\cY;
\rho^0)}+\sqrt{m}\|\grad_p\pt_{p_i}\dt^m\|^2_{L^2(\cY; \rho^0)}\\ &
+ K\sqrt{m} (\|\dt^m\|^2_{L^2(\cY;
\rho^0)}+\|\grad_p\dt^m\|^2_{L^2(\cY; \rho^0)}+1) \ ,
\end{array}
\end{equation*} where $K>0$ is a constant that depends only on $\sup\li_{r\in
\T^d}|h(r)|$. This implies the lemma.
\end{proof}

\begin{proof}[Proof of Lemma \ref{L:L2_Bounds2b}]
The proof goes again along the same lines of Lemma
\ref{L:L2_Bounds2}. We take $\pt_{r_i}$ on both sides of the
equation (\ref{Eq:deltam equation}) and we get the following
equation
\begin{equation*}
\cL_q^m\pt_{r_i}\dt^m= \dfrac{2}{\sqrt{m}}\cB\pt_{r_i}\dt^m-
\dfrac{1}{\sqrt{m}}p\cdot \pt_{r_i}h(r)[\dt^m+1]-
\dfrac{1}{\sqrt{m}}p\cdot h(r)\pt_{r_i}\dt^m \ .
\end{equation*}

Multiplying both sides of the above equation by $\pt_{r_i}\dt^m$ and
integrate with respect to $L^2(\cY; \rho^0)$--inner product we get
\begin{equation*}\begin{array}{ll} \langle\cL_q^m \pt_{r_i}\dt^m,
\pt_{r_i}\dt^m\rangle_{L^2(\cY;
\rho^0)}=&\dfrac{2}{\sqrt{m}}\langle\cB \pt_{r_i}\dt^m,
\pt_{r_i}\dt^m\rangle_{L^2(\cY; \rho_0)}\\
&-\dfrac{1}{\sqrt{m}}\langle
p\cdot \pt_{r_i}h(r)[\dt^m+1], \pt_{r_i}\dt^m\rangle_{L^2(\cY; \rho_0)}\\
& -\dfrac{1}{\sqrt{m}}\langle p\cdot h(r)\pt_{r_i}\dt^m,
\pt_{r_i}\dt^m\rangle_{L^2(\cY; \rho_0)}
 \ .
\end{array} \end{equation*}

We apply Lemma \ref{L:L2_Bounds1} with
$f(p,r)=g(p,r)=\pt_{r_i}\dt^m(p,r)$ to get
\begin{equation*}\begin{array}{ll} \langle\cL_q^m \pt_{r_i}\dt^m,
\pt_{r_i}\dt^m\rangle_{L^2(\cY; \rho^0)}=&
-\dfrac{1}{m}\|\grad_p\pt_{r_i}\dt^m\|_{L^2(\cY;
\rho^0)}^2+\dfrac{1}{2\sqrt{m}}\langle h(r)p,
|\pt_{r_i}\dt^m|^2\rangle_{L^2(\cY; \rho^0)} \ .
\end{array} \end{equation*}

We now apply Lemma \ref{L:L2_Bounds4} and we also have
\begin{equation*}\begin{array}{ll}
\langle\cB \pt_{r_i}\dt^m, \pt_{r_i}\dt^m\rangle_{L^2(\cY; \rho^0)}
=&\play{\dfrac{1}{2}\int_{\cY}p\cdot
h(r)|\pt_{r_i}\dt^m|^2\rho^0(p,r)dpdr} \ .
\end{array} \end{equation*}

Furthermore, we can calculate
\begin{equation*}\begin{array}{ll}
\langle p\cdot\pt_{r_i}h(r)[\dt^m+1],
\pt_{r_i}\dt^m\rangle_{L^2(\cY; \rho^0)} =&\play{\int_{\cY}p\cdot
\pt_{r_i}h(r)(\dt^m+1)\pt_{r_i}\dt^m \rho^0(p,r)dpdr} \ ,
\end{array}
\end{equation*}

We can apply a straightforward generalization of Lemma \ref{L:L2_Bounds3} with
$h(r)$ replaced by $\pt_{r_i}h(r)$, as well as Young's inequality,
to estimate the right hand side of the above equation by
\begin{equation*}\begin{array}{ll}
\langle p\cdot\pt_{r_i}h(r)[\dt^m+1],
\pt_{r_i}\dt^m\rangle_{L^2(\cY; \rho^0)} \leq & K
\left(\|\dt^m\|_{L^2(\cY;\rho^0)}^2+\|\grad_p\dt^m\|_{L^2(\cY;
\rho^0)}^2 +\|\grad_r\dt^m\|_{L^2(\cY;
\rho^0)}^2+1\right)\\&+\dfrac{1}{2}\|\grad_p\pt_{r_i}\dt^m\|_{L^2(\cY;
\rho^0)}^2 \ ,
\end{array}
\end{equation*}
 where $K>0$ is a constant that depends only on $\sup\li_{r\in
\T^d}|\grad_rh(r)|$. We also have
\begin{equation*}\begin{array}{ll}
\langle p\cdot h(r)\pt_{r_i}\dt^m, \pt_{r_i}\dt^m\rangle_{L^2(\cY;
\rho^0)} =& \play{\int_{\cY} p\cdot h(r) |\pt_{r_i}\dt^m|^2
\rho^0(p,r) dpdr} \ .
\end{array}
\end{equation*}

Thus we get the identity
\begin{equation*}\begin{array}{ll}
-\dfrac{1}{m}\|\grad_p\pt_{r_i}\dt^m\|^2_{L^2(\cY;\rho^0)}=&
-\dfrac{1}{2\sqrt{m}}\langle ph(r),
|\pt_{r_i}\dt^m|^2\rangle_{L^2(\cY;
\rho^0)}\nonumber\\
&\quad-\dfrac{1}{\sqrt{m}}\langle p\cdot\pt_{r_i}h(r)[\dt^m+1],
\pt_{r_i}\dt^m\rangle_{L^2(\cY; \rho^0)} \ .
\end{array}
\end{equation*}

Making use of Lemma \ref{L:L2_Bounds3} and Young's inequality again we estimate
\begin{equation*}\begin{array}{ll}
\|\grad_p\pt_{r_i}\dt^m\|^2_{L^2(\cY;\rho^0)}\leq &
K\sqrt{m}\|\grad_p \dt^m\|^2_{L^2(\cY; \rho^0)}\\ & + K\sqrt{m}
(\|\dt^m\|^2_{L^2(\cY; \rho^0)}+\|\grad_p\dt^m\|^2_{L^2(\cY;
\rho^0)}+\|\grad_r\dt^m\|^2_{L^2(\cY;
\rho^0)}+1)\\&+\sqrt{m}\|\grad_p\pt_{r_i} \dt^m\|_{L^2(\cY;
\rho^0)}^2 \ ,
\end{array}
\end{equation*}
 where $K>0$ is a constant that depends only on $\sup\li_{r\in
\T^d}\max(|h(r)|, |\grad_r h(r)|)$. This implies the lemma.
\end{proof}

\begin{proof}[Proof of Lemma \ref{L:AuxiliaryBounds}]
We start with $T_{1}(\delta^{m})$. By Lemma \ref{L:L2_Bounds4} with
$f=\delta^{m}$ we have
\begin{align*}
 T_{1}(\delta^{m})&=\sqrt{m}\left[\left<\mathcal{B}\delta^{m},\delta^{m}\right>-\left<p h(r), \left(\delta^{m}+1\right)\delta^{m}\right>\right]\nonumber\\
 &=\sqrt{m}\left[\frac{1}{2}\left<ph(r),|\delta^{m}|^{2}\right>-\left<p h(r), \left(\delta^{m}+1\right)\delta^{m}\right>\right]\nonumber\\
 &=\sqrt{m}\left[-\frac{1}{2}\left<ph(r),|\delta^{m}|^{2}\right>+\left<p h(r), \delta^{m}\right>\right]
 \ . \end{align*}

Thus, by Lemma \ref{L:L2_Bounds3} with $f=g=\delta^{m}$ we have the
following bound
\begin{align}
\left| T_{1}(\delta^{m})\right|&\leq\sqrt{m} K\left[\left\|\delta^{m}\right\|\left\|\nabla_{p}\delta^{m}\right\|+\left\|\delta^{m}\right\|^{2}\right]\nonumber\\
&\leq\sqrt{m} \frac{K+2}{2}\left\|\delta^{m}\right\|^{2}
+\sqrt{m}\frac{K}{2}\left\|\nabla_{p}\delta^{m}\right\|^{2} \ .
\label{Eq:TermT1}
\end{align}

Next we derive an upper bound for $T_{2}(\delta^{m})=\left<A
\delta^{m}, A\cL^{1}\delta^{m}\right>$.  For this purpose we first
notice that
$$\begin{array}{ll}
\left<A\delta^{m},A\mathcal{B}\delta^{m}\right>&=\left<\nabla_{p}\delta^{m},\nabla_{p}\mathcal{B}\delta^{m}\right>\nonumber\\
&=\left<\nabla_{p}\delta^{m},\nabla_{p}\left(p\nabla_{r}\delta^{m}+b(r)\nabla_{p}\delta^{m}\right)\right>\nonumber\\
&=\left<\nabla_{p}\delta^{m},\mathcal{B}\nabla_{p}\delta^{m}\right>+\left<\nabla_{p}\delta^{m},\nabla_{r}\delta^{m}\right>\nonumber\\
&=\frac{1}{2}\left<ph(r),|\nabla_{p}\delta^{m}|^{2}\right>+\left<\nabla_{p}\delta^{m},\nabla_{r}\delta^{m}\right>
 \ , \end{array}$$
where in the last inequality we used Lemma \ref{L:L2_Bounds4}. Then,
using the equation for $\delta^{m}$,
(\ref{Eq:EquationInvariantDensity3}) and Lemma \ref{L:L2_Bounds3} we
have
\begin{align*}
&\left|T_{2}(\delta^{m})\right|=\left|\left<A \delta^{m}, A\cL^{1}\delta^{m}\right>\right|\nonumber\\
&=\left|\frac{(1+\sqrt{m})}{2}\left<ph(r),|\nabla_{p}\delta^{m}|^{2}\right>+(1+\sqrt{m})\left<\nabla_{p}\delta^{m},\nabla_{r}\delta^{m}\right>-\sqrt{m}\left<\nabla_{p}\delta^{m},h(r)(\delta^{m}+1)\right>\right.\nonumber\\
&\quad\left.-\sqrt{m}\left<ph(r),|\nabla_{p}\delta^{m}|^{2}\right>\right|\nonumber\\
&=\left|\frac{(1-\sqrt{m})}{2}\left<ph(r),|\nabla_{p}\delta^{m}|^{2}\right>+(1+\sqrt{m})\left<\nabla_{p}\delta^{m},\nabla_{r}\delta^{m}\right>-\sqrt{m}\left<\nabla_{p}\delta^{m},h(r)(\delta^{m}+1)\right>\right|\nonumber\\
&\leq \frac{(1-\sqrt{m})}{2} K \|\nabla_{p}\delta^{m}\|\|\nabla_{p}\nabla_{p}\delta^{m}\|+\frac{1+\sqrt{m}}{4\eta}\|\nabla_{p}\delta^{m}\|^{2}+\eta(1+\sqrt{m})\|\nabla_{r}\delta^{m}\|^{2}\nonumber\\
&\quad +\sqrt{m}K\left(\|\delta^{m}\|^{2}+\|\nabla_{p}\delta^{m}\|^{2}\right)\nonumber\\
&\leq \frac{(1-\sqrt{m})}{2} K \left(\eta \|\nabla_{p}\delta^{m}\|^{2}+\frac{1}{4\eta}\|\nabla_{p}\nabla_{p}\delta^{m}\|\right)+\frac{1+\sqrt{m}}{4\eta}\|\nabla_{p}\delta^{m}\|^{2}\nonumber\\
&\quad
+\eta(1+\sqrt{m})\|\nabla_{r}\delta^{m}\|^{2}+\sqrt{m}K\left(\|\delta^{m}\|^{2}+\|\nabla_{p}\delta^{m}\|^{2}\right)
 \ . \end{align*}

Next step now is to use Lemma \ref{L:L2_Bounds2a}. Doing so we get
the bound
\begin{align*}
&\left|T_{2}(\delta^{m})\right|=\left|\left<A \delta^{m}, A\cL^{1}\delta^{m}\right>\right|\nonumber\\
&\leq \frac{(1-\sqrt{m})}{2} K \left(\eta \|\nabla_{p}\delta^{m}\|^{2}+\frac{\sqrt{m}}{4\eta}K\left[1+ \left\|\delta^{m}\right\|^{2}+ \left\| \nabla_{p}\delta^{m}\right\|^{2}\right]\right)+\frac{1+\sqrt{m}}{4\eta}\|\nabla_{p}\delta^{m}\|^{2}\nonumber\\
&\quad +\eta(1+\sqrt{m})\|\nabla_{r}\delta^{m}\|^{2}+\sqrt{m}K\left(\|\delta^{m}\|^{2}+\|\nabla_{p}\delta^{m}\|^{2}\right)\nonumber\\
&\leq \left(\frac{K(1-\sqrt{m})}{2}\eta+ K\frac{(1-\sqrt{m})}{8\eta}\sqrt{m}+K\sqrt{m}+\frac{1+K\sqrt{m}}{4\eta}\right)\|\nabla_{p}\delta^{m}\|^{2}
+\nonumber\\
&\qquad+\eta(1+\sqrt{m})\|\nabla_{r}\delta^{m}\|^{2}+K\left(\frac{(1-\sqrt{m})}{4\eta}+1\right)\sqrt{m}\left(1+\left\|\delta^{m}\right\|^{2}\right)  \ .
\end{align*}

Use now Lemma \ref{L:L2_Bounds2} and then Lemma \ref{L:L2_Bounds3} to bound the term $\frac{1+K\sqrt{m}}{\eta}\|\nabla_{p}\delta^{m}\|^{2}$ by terms of the form
$K\frac{1+\sqrt{m}}{\eta}\sqrt{m}\left(\|\delta^{m}\|^{2}+\|\nabla_{p}\delta^{m}\|^{2}\right)$. Choosing then $\eta=\eta(m)$ such that $\eta(m)\rightarrow 0$ and
$\frac{\sqrt{m}}{\eta(m)}\rightarrow 0$, we get that for $m$
sufficiently small
\begin{align}
\left|T_{2}(\delta^{m})\right|&=\left|\left<A \delta^{m}, A\cL^{1}\delta^{m}\right>\right|\nonumber\\
&\leq \left(\eta(m)+\sqrt{m}+\frac{\sqrt{m}}{\eta(m)} \right)K \|\delta^{m}\|^{2}_{H^{1}}+
\left(\frac{\sqrt{m}}{\eta(m)}+\sqrt{m}\right)K\left(1+ \left\|\delta^{m}\right\|^{2}\right)
  \ .
\label{Eq:TermT2}
\end{align}

Next we derive an upper bound for
$T_{4}(\delta^{m})=\left<\mathcal{C} \delta^{m},
\mathcal{C}\cL^{1}\delta^{m}\right>$.  For this purpose we first
notice that
\begin{align*}
\left<\mathcal{C}\delta^{m},\mathcal{C}\mathcal{B}\delta^{m}\right>&=\left<\nabla_{r}\delta^{m},\nabla_{r}\mathcal{B}\delta^{m}\right>\nonumber\\
&=\left<\nabla_{r}\delta^{m},\nabla_{r}\left(p\nabla_{r}\delta^{m}+b(r)\nabla_{p}\delta^{m}\right)\right>\nonumber\\
&=\left<\nabla_{r}\delta^{m},\mathcal{B}\nabla_{r}\delta^{m}\right>+\left<\nabla_{r}\delta^{m},\nabla_{r}b(r)\nabla_{p}\delta^{m}\right>\nonumber\\
&=\frac{1}{2}\left<ph(r),|\nabla_{r}\delta^{m}|^{2}\right>+\left<\nabla_{r}\delta^{m},\nabla_{r}b(r)\nabla_{p}\delta^{m}\right>
 \ , \end{align*}
where in the last inequality we used Lemma \ref{L:L2_Bounds4}. Then,
using the equation for $\delta^{m}$
\begin{align*}
T_{4}(\delta^{m})&=\left<\mathcal{C} \delta^{m}, \mathcal{C}\cL^{1}\delta^{m}\right>\nonumber\\
&=\frac{(1+\sqrt{m})}{2}\left<ph(r),|\nabla_{r}\delta^{m}|^{2}\right>+(1+\sqrt{m})\left<\nabla_{r}\delta^{m},\nabla_{r}b(r)\nabla_{p}\delta^{m}\right>\nonumber\\
&\quad-\sqrt{m}\left<\nabla_{r}\delta^{m},p \nabla_{r}h(r)(\delta^{m}+1)\right>-\sqrt{m}\left<ph(r),|\nabla_{r}\delta^{m}|^{2}\right>\nonumber\\
&=\frac{(1-\sqrt{m})}{2}\left<ph(r),|\nabla_{r}\delta^{m}|^{2}\right>+(1+\sqrt{m})\left<\nabla_{r}\delta^{m},\nabla_{r}b(r)
\nabla_{p}\delta^{m}\right>\nonumber\\
&\quad-\sqrt{m}\left<\nabla_{r}\delta^{m},p\nabla_{r}h(r)(\delta^{m}+1)\right>
\ .
\end{align*}

Using Lemma \ref{L:L2_Bounds3} we subsequently obtain
\begin{align*}
|T_{4}(\delta^{m})|
&\leq \frac{(1-\sqrt{m})}{2} K \|\nabla_{r}\delta^{m}\|\|\nabla_{p}\nabla_{r}\delta^{m}\|+(1+\sqrt{m})\left|\left<\nabla_{r}\delta^{m},\nabla_{r}b(r) \nabla_{p}\delta^{m}\right>\right|\nonumber\\
&\quad+\sqrt{m}\left|\left<\nabla_{r}\delta^{m},p\nabla_{r}h(r)(\delta^{m}+1)\right>\right|\nonumber\\
&\leq \frac{(1-\sqrt{m})}{2} K \left(\eta \|\nabla_{r}\delta^{m}\|^{2}+\frac{1}{4\eta}\|\nabla_{p}\nabla_{r}\delta^{m}\|^{2}\right)\nonumber\\
&+(1+\sqrt{m})\left[\eta \left\|\nabla_{r}\delta^{m}\right\|^{2}+\frac{1}{4\eta}K\left\| \nabla_{p}\delta^{m}\right\|^{2}\right]+\sqrt{m}\left|\left<\nabla_{r}\delta^{m},p\nabla_{r}h(r)(\delta^{m}+1)\right>\right|\nonumber\\
&\leq \frac{(1-\sqrt{m})}{2} K \left(\eta \|\nabla_{r}\delta^{m}\|^{2}+\frac{1}{4\eta}\|\nabla_{p}\nabla_{r}\delta^{m}\|^{2}\right)\nonumber\\
&+(1+\sqrt{m})\left[\eta \left\|\nabla_{r}\delta^{m}\right\|^{2}+\frac{1}{4\eta}K\left\| \nabla_{p}\delta^{m}\right\|^{2}\right]\nonumber\\
&+
\sqrt{m}K\left[\left\|\nabla_{r}\delta^{m}\right\|\left\|\nabla_{p}\delta^{m}\right\|+\left\|\nabla_{p}\nabla_{r}\delta^{m}\right\|\left\|\delta^{m}\right\|+\left\|\nabla_{r}\delta^{m}\right\|^{2}\right]
\nonumber\\
&\leq \frac{(1-\sqrt{m})}{2} K \left(\eta \|\nabla_{r}\delta^{m}\|^{2}+\frac{1}{4\eta}\|\nabla_{p}\nabla_{r}\delta^{m}\|^{2}\right)\nonumber\\
&+(1+\sqrt{m})\left[\eta \left\|\nabla_{r}\delta^{m}\right\|^{2}+\frac{1}{4\eta}K\left\| \nabla_{p}\delta^{m}\right\|^{2}\right]\nonumber\\
&+
\sqrt{m}K\left[\frac{1}{2}\left\|\nabla_{r}\delta^{m}\right\|^{2}+\frac{1}{2}\left\|\nabla_{p}\delta^{m}\right\|^{2}+\frac{1}{2}\left\|\nabla_{p}\nabla_{r}\delta^{m}\right\|^{2}+\frac{1}{2}\left\|\delta^{m}\right\|^{2}+\left\|\nabla_{r}\delta^{m}\right\|^{2}\right]
\nonumber\\
&\leq \frac{(1-\sqrt{m})}{2} K \left(\eta \|\nabla_{r}\delta^{m}\|^{2}+\frac{1}{4\eta}\|\nabla_{p}\nabla_{r}\delta^{m}\|^{2}\right)\nonumber\\
&+(1+\sqrt{m})\left[\eta \left\|\nabla_{r}\delta^{m}\right\|^{2}+\frac{1}{4\eta}K\left\| \nabla_{p}\delta^{m}\right\|^{2}\right]\nonumber\\
&+
\sqrt{m}K\left[\left\|\nabla_{r}\delta^{m}\right\|^{2}+\left\|\nabla_{p}\delta^{m}\right\|^{2}+\left\|\nabla_{p}\nabla_{r}\delta^{m}\right\|^{2}+\left\|\delta^{m}\right\|^{2}\right]
 \ . \end{align*}

The constant $K$ may change from line to line, but it is always
independent of $m$. Using Lemma  \ref{L:L2_Bounds2b} and then Lemma
\ref{L:L2_Bounds2} we subsequently obtain
\begin{align*}
|T_{4}(\delta^{m})|
&\leq \frac{(1-\sqrt{m})}{2} K \left(\eta \|\nabla_{r}\delta^{m}\|^{2}+\frac{\sqrt{m}}{4\eta}K\left[\left\|\delta^{m}\right\|^{2}+ \left\| \nabla_{p}\delta^{m}\right\|^{2}+ \left\| \nabla_{r}\delta^{m}\right\|^{2}\right]\right)\nonumber\\
&+(1+\sqrt{m})\left[\eta \left\|\nabla_{r}\delta^{m}\right\|^{2}+\frac{1}{4\eta}K\left\| \nabla_{p}\delta^{m}\right\|^{2}\right]\nonumber\\
&+
\sqrt{m}K\left[\left\|\nabla_{r}\delta^{m}\right\|^{2}+\left\|\nabla_{p}\delta^{m}\right\|^{2}+\sqrt{m}\left[\left\|\delta^{m}\right\|^{2}+
\left\| \nabla_{p}\delta^{m}\right\|^{2}+ \left\|
\nabla_{r}\delta^{m}\right\|^{2}\right]+\left\|\delta^{m}\right\|^{2}\right]
 \ .  \end{align*}

Finally, choosing $\eta=\eta(m)$ such that $\eta(m)\rightarrow 0$
and $\frac{\sqrt{m}}{\eta(m)}\rightarrow 0$, we get that for $m$
sufficiently small and for some constant $K<\infty$
\begin{align}
\left|T_{4}(\delta^{m})\right|&=\left|\left<\mathcal{C} \delta^{m}, \mathcal{C}\cL^{1}\delta^{m}\right>\right|\nonumber\\
&\leq K\left[\left(\eta(m)+\frac{\sqrt{m}}{\eta(m)}\right)\|\delta^{m}\|^{2}_{H^{1}}+\frac{\sqrt{m}}{\eta(m)}\left\|\delta^{m}\right\|^{2}\right]\ . \label{Eq:TermT4}
\end{align}

It remains to consider the cross-term
\begin{align*}
T_{3}(\delta^{m})&=\left<A\cL^{1}\delta^{m},\mathcal{C}\delta^{m}\right>+\left<A\delta^{m},\mathcal{C}\cL^{1}\delta^{m}\right>
 \ . \end{align*}

Recalling (\ref{Eq:EquationInvariantDensity2}) we have the following
calculations
\begin{align*}
T_{3}(\delta^{m})&=\left<A\cL^{1}\delta^{m},\mathcal{C}\delta^{m}\right>+\left<A\delta^{m},\mathcal{C}\cL^{1}\delta^{m}\right>\nonumber\\
&=\left<\nabla_{r}\delta^{m},(1+\sqrt{m})\mathcal{B}\nabla_{p}\delta^{m}+(1+\sqrt{m})\nabla_{r}\delta^{m}-\sqrt{m}h(r)(\delta^{m}+1)-\sqrt{m}h(r)p\nabla_{p}\delta^{m}\right>\nonumber\\
&+\left<\nabla_{p}\delta^{m},(1+\sqrt{m})\mathcal{B}\nabla_{r}\delta^{m}+(1+\sqrt{m})\nabla_{r}b(r)\nabla_{p}\delta^{m}-\sqrt{m}\nabla_{r}h(r) p (\delta^{m}+1)\right.\nonumber\\
&\qquad\left.-\sqrt{m}h(r)p\nabla_{r}\delta^{m}\right>\nonumber\\
&=(1+\sqrt{m})\left[\left<\nabla_{p}\delta^{m}, \mathcal{B}\nabla_{r}\delta^{m}\right>+\left<\nabla_{r}\delta^{m}, \mathcal{B}\nabla_{p}\delta^{m}\right>\right]\nonumber\\
&\quad +(1+\sqrt{m})\left[\left<\nabla_{p}\delta^{m},\nabla_{r}b \nabla_{p}\delta^{m}\right>+\left\|\nabla_{r}\delta^{m}\right\|^{2}\right]\nonumber\\
&\quad -\sqrt{m}\left[\left<\nabla_{r}\delta^{m}, p h(r)\nabla_{p}\delta^{m}\right>+\left<\nabla_{p}\delta^{m}, p h(r)\nabla_{r}\delta^{m}\right>\right]\nonumber\\
&\quad -\sqrt{m}\left[\left<\nabla_{p}\delta^{m}, \nabla_{r}h(r) p (\delta^{m}+1)\right>+\left<\nabla_{r}\delta^{m}, h(r)(\delta^{m}+1)\right>\right]\nonumber\\
&=(1+\sqrt{m})\left[\left<\nabla_{p}\delta^{m}, \cL^{1}\nabla_{r}\delta^{m}\right>+\left<\nabla_{r}\delta^{m}, \cL^{1}\nabla_{p}\delta^{m}\right>\right]\nonumber\\
&\quad-(1+\sqrt{m})\left[\left<\nabla_{p}\delta^{m}, \nabla_{p}\nabla_{r}\delta^{m}\right>+\left<\nabla_{r}\delta^{m}, \nabla_{p}\nabla_{p}\delta^{m}\right>\right]\nonumber\\
&\quad +(1+\sqrt{m})\left[\left<\nabla_{p}\delta^{m},\nabla_{r}b \nabla_{p}\delta^{m}\right>+\left\|\nabla_{r}\delta^{m}\right\|^{2}\right]\nonumber\\
&\quad -\sqrt{m}\left[\left<\nabla_{r}\delta^{m}, p h(r)\nabla_{p}\delta^{m}\right>+\left<\nabla_{p}\delta^{m}, p h(r)\nabla_{r}\delta^{m}\right>\right]\nonumber\\
&\quad -\sqrt{m}\left[\left<\nabla_{p}\delta^{m}, \nabla_{r}h(r) p
(\delta^{m}+1)\right>+\left<\nabla_{r}\delta^{m},
h(r)(\delta^{m}+1)\right>\right] \ .
\end{align*}

Using now Lemma \ref{L:L2_Bounds1} on the first term of the right
hand side of the last display we obtain
\begin{align*}
T_{3}(\delta^{m})
&=-2(1+\sqrt{m})\left<\nabla_{p}\nabla_{p}\delta^{m},
\nabla_{p}\nabla_{r}\delta^{m}\right>
\nonumber\\
&\quad-(1+\sqrt{m})\left[\left<\nabla_{p}\delta^{m}, \nabla_{p}\nabla_{r}\delta^{m}\right>+\left<\nabla_{r}\delta^{m}, \nabla_{p}\nabla_{p}\delta^{m}\right>\right]\nonumber\\
&\quad +(1+\sqrt{m})\left[\left<\nabla_{p}\delta^{m},\nabla_{r}b \nabla_{p}\delta^{m}\right>+\left\|\nabla_{r}\delta^{m}\right\|^{2}\right]\nonumber\\
&\quad -\sqrt{m}\left[\left<\nabla_{r}\delta^{m}, p h(r)\nabla_{p}\delta^{m}\right>+\left<\nabla_{p}\delta^{m}, p h(r)\nabla_{r}\delta^{m}\right>\right]\nonumber\\
&\quad -\sqrt{m}\left[\left<\nabla_{p}\delta^{m}, \nabla_{r}h(r) p
(\delta^{m}+1)\right>+\left<\nabla_{r}\delta^{m},
h(r)(\delta^{m}+1)\right>\right]  \ .
\end{align*}

Next we bound terms from above. Using Lemma \ref{L:L2_Bounds3}, we
have for $\eta>0$
\begin{align*}
|T_{3}(\delta^{m})|
&\leq (1+\sqrt{m})K\left[\left\|\nabla_{p}\nabla_{p}\delta^{m} \right\|^{2}+\left\|\nabla_{p}\nabla_{r}\delta^{m} \right\|^{2}\right]\nonumber\\
&\quad+\sqrt{m}\left[ \eta \left|\nabla_{p}\delta^{m}\right\|^{2}+\frac{1}{4\eta}\left\| \nabla_{p}\nabla_{r}\delta^{m}\right\|^{2}+\eta \left\|\nabla_{r}\delta^{m}\right\|^{2}+\frac{1}{4\eta} \left\|\nabla_{p}\nabla_{p}\delta^{m}\right\|^{2}\right]\nonumber\\
&\quad +\frac{(1+\sqrt{m})}{2}K\left[\left\|\nabla_{r}\delta^{m}\right\|\left\|\nabla_{p}\nabla_{p}\delta^{m}\right\|+\left\|\nabla_{p}\delta^{m}\right\|\left\|\nabla_{p}\nabla_{r}\delta^{m}\right\|\right]\nonumber\\
&\quad +(1+\sqrt{m})\left[K \left\|\nabla_{p}\delta^{m}\right\|^{2}+\left\|\nabla_{r}\delta^{m}\right\|^{2}\right]\nonumber\\
&\quad +\sqrt{m}K\left[1+\left\|\delta^{m}\right\|^{2}
+\left\|\nabla_{p}\delta^{m}\right\|^{2}+
\left\|\nabla_{r}\delta^{m}\right\|^{2}\right] \ .
\end{align*}

The constant $K$ may change from line to line. Using Lemmas
\ref{L:L2_Bounds2a} and \ref{L:L2_Bounds2b} we obtain
\begin{align*}
|T_{3}(\delta^{m})|
&\leq \sqrt{m}(1+\sqrt{m})K\left[1+\left\|\delta^{m}\right\|^{2} +\left\|\nabla_{p}\delta^{m}\right\|^{2}+ \left\|\nabla_{r}\delta^{m}\right\|^{2}\right]\nonumber\\
&\quad+\sqrt{m}\left[ \eta \left|\delta^{m}\right\|^{2}_{H^{1}}+\frac{\sqrt{m}}{4\eta}\left[1+ \left\|\delta^{m}\right\|^{2} +\left\|\delta^{m}\right\|^{2}_{H^{1}}\right]\right]\nonumber\\
&\quad +(1+\sqrt{m})K\left[\eta \left|\delta^{m}\right\|^{2}_{H^{1}}+\frac{\sqrt{m}}{4\eta}\left[1+ \left\|\delta^{m}\right\|^{2} +\left\|\delta^{m}\right\|^{2}_{H^{1}}\right]\right]\nonumber\\
&\quad +(1+\sqrt{m})\left[K\left\|\nabla_{p}\delta^{m}\right\|^{2}+\left\|\nabla_{r}\delta^{m}\right\|^{2}\right]\nonumber\\
&\quad +\sqrt{m}K\left[\left\|\delta^{m}\right\|^{2} +\left\|\nabla_{p}\delta^{m}\right\|^{2}+ \left\|\nabla_{r}\delta^{m}\right\|^{2}\right] \ . \nonumber\\
\end{align*}

Applying then Lemma \ref{L:L2_Bounds2} to estimate the term
$\left\|\nabla_{p}\delta^{m}\right\|^{2}$ on the fourth line of the
last display, we obtain
 the following bound
\begin{align*}
|T_{3}(\delta^{m})| &\leq
\left(\sqrt{m}+\frac{\eta}{\sqrt{m}}\right)K\left[1+\left\|\delta^{m}\right\|^{2}
+\left\|\delta^{m}\right\|^{2}_{H^{1}}\right]+ \eta K
\left\|\delta^{m}\right\|^{2}_{H^{1}}+(1+\sqrt{m})
\left\|\nabla_{r}\delta^{m}\right\|^{2} \ . \nonumber
\end{align*}

Finally, choosing $\eta=\eta(m)$ such that $\eta(m)\rightarrow 0$
and $\frac{\sqrt{m}}{\eta(m)}\rightarrow 0$, we get that for $m$
sufficiently small and for some constant $K<\infty$
\begin{align*}
|T_{3}(\delta^{m})| &\leq
(\sqrt{m}+\frac{\sqrt{m}}{\eta(m)})K\left[1+\left\|\delta^{m}\right\|^{2}
+\left\|\delta^{m}\right\|^{2}_{H^{1}}\right]+ \eta(m)K \left\|\delta^{m}\right\|^{2}_{H^{1}} +(1+\sqrt{m})
\left\|\nabla_{r}\delta^{m}\right\|^{2} \ . 
\end{align*}

This concludes the proof of the lemma.
\end{proof}

\section{Proofs of Lemmas in Section \ref{S:CellProblems}.}\label{S:CellProblemsProof}

\begin{proof}[Proof of Lemma \ref{mZeroPoincareInequality}]
This can be shown by using Theorem 4.2.5 in
\cite{BakeryGentilLedoux}. Let $(P_t)_{t\geq 0}$ be the Markov
semigroup corresponding to generator $\mathcal{L}^{1}$ on
$\mathcal{Y}$.

By Lemma \ref{L:L2_Bounds1} with $m=1$, we obtain for the first term
(recall that $\rho^{0}(p,r)dpdr$ is the invariant measure
corresponding to the operator $\cL^{1}$) that the Dirichlet form
associated with $(P_t)_{t\geq 0}$ can be calculated as follows
\begin{align*}
\mathcal{E}(f)&=\left<-\cL^{1}f, f\right>_{L^{2}(\mathcal{Y};
\rho^{0})}=\left\|\nabla_{p}f\right\|^{2}_{L^{2}(\mathcal{Y};
\rho^{0})} \ .
\end{align*}

Thus by Theorem 4.2.5 of \cite{BakeryGentilLedoux} the validity of
Poincar\'{e} inequality is equivalent to exponential convergence to
equilibrium of the semigroup $(P_t)_{t\geq 0}$:
$$\int_{\mathcal{Y}}\left(P_t f-\int_{\mathcal{Y}}(P_t
f)\rho^0(p,r)dpdr\right)^2\rho^0(p,r)dpdr\leq c(f)e^{-2t/\kappa} \
.$$ for some constant $\kappa>0$. The above inequality is true since
$\cL_q^1$ admits a spectral gap (see \cite{EckmannHairer}).
\end{proof}

\begin{proof}[Proof of Lemma \ref{mZeroL2bound1}]
We make use of our equation (\ref{Eq:ErrorEquation}),
(\ref{Eq:mZeroIntegrationByParts}) as well as Lemma
\ref{L:L2_Bounds3} and we get
\begin{align}
\left\|\dfrac{1}{\sqrt{m}}\grad_p
\Psi_{\ell,3}^m\right\|^2_{L^2(\cY;
\rho^0)}-K\|\Psi_{\ell,3}^m\|_{L^2(\cY;
\rho^0)}\left\|\dfrac{1}{\sqrt{m}}\grad_p
\Psi_{\ell,3}^m\right\|_{L^2(\cY;\rho^0)}\leq \sqrt{m}\|\cB
\Psi_{\ell,2}\|_{L^2(\cY; \rho^0)}\|\Psi_{\ell,3}^m\|_{L^2(\cY;
\rho^0)} \ , \label{Eq:BasicEstimate}
\end{align}
for some constant $K>0$ independent of $m$.

We apply Lemma \ref{mZeroPoincareInequality}, using the fact that
$\play{\int_{\cY}\Psi_{\ell,3}^m(p,r)\rho^m(p,r)dpdr=0}$, and we
have
\begin{align}
&\|\Psi_{\ell,3}^m\|_{L^2(\cY; \rho^0)}^2 & \nonumber
\\
&\leq
\left\|\Psi_{\ell,3}^m-\int_{\cY}\Psi_{\ell,3}^m(p,r)\rho^0(p,r)
dpdr\right\|_{L^2(\cY;
\rho^0)}^2+\left(\int_{\cY}\Psi_{\ell,3}^m(p,r)\rho^0(p,r)dpdr\right)^2
\nonumber
\\
&= \left\|\Psi_{\ell,3}^m-\int_{\cY}\Psi_{\ell,3}^m(p,r)\rho^0(p,r)
dpdr\right\|_{L^2(\cY;
\rho^0)}^2+\nonumber\\
&\qquad+\left(\int_{\cY}\Psi_{\ell,3}^m(p,r)\rho^0(p,r)dpdr
-\int_{\cY}\Psi_{\ell,3}^m(p,r)\rho^m(p,r)dpdr\right)^2 \nonumber
\\
&\leq
\kp\|\grad_p\Psi_{\ell,3}^m\|_{L^2(\cY;\rho^0)}^2+\left(\int_{\cY}\Psi_{\ell,3}^m(p,r)\dt^m(p,r)\rho^0(p,r)dpdr\right)^2
\nonumber
\\
&\leq \kp\|\grad_p\Psi_{\ell,3}^m\|_{L^2(\cY; \rho^0)}^2+
\|\Psi_{\ell,3}^m\|_{L^2(\cY; \rho^0)}^2\|\dt^m\|_{L^2(\cY;
\rho^0)}^2 \ . \nonumber
\end{align}

Since we have $\lim\li_{m\ra 0} \|\dt^m\|_{L^2(\cY;\rho^0)}=0$, we
can choose $m$ small enough so that
\begin{align}
\|\Psi_{\ell,3}^m\|_{L^2(\cY; \rho^0)}\leq 2\kp \|\grad_p
\Psi_{\ell,3}^m\|_{L^2(\cY; \rho^0)} \ .
\label{Eq:noncenteredPoincare}
\end{align}

Combining (\ref{Eq:noncenteredPoincare}) and
(\ref{Eq:BasicEstimate}) we see that we have
\begin{align}
(1-2\kp
K\sqrt{m})\left\|\dfrac{1}{\sqrt{m}}\grad_p\Psi_{\ell,3}^m\right\|^2_{L^2(\cY;
\rho^0)}\leq \sqrt{m}\|\cB \Psi_{\ell,2}\|_{L^2(\cY;
\rho^0)}\|\Psi_{\ell,3}^m\|_{L^2(\cY; \rho^0)} \ . \nonumber
\end{align}

Using (\ref{Eq:noncenteredPoincare}) again we see that
\begin{align}
(1-2\kp
K\sqrt{m})\left\|\dfrac{1}{\sqrt{m}}\grad_p\Psi_{\ell,3}^m\right\|^2_{L^2(\cY;
\rho^0)}\leq 2\kp m\|\cB \Psi_{\ell,2}\|_{L^2(\cY;
\rho^0)}\left\|\dfrac{1}{\sqrt{m}}\grad_p\Psi_{\ell,3}^m\right\|_{L^2(\cY;
\rho^0)} \ . \nonumber
\end{align}

This means that we have the bound
\begin{align}
\left\|\dfrac{1}{\sqrt{m}}\grad_p\Psi_{\ell,3}^m\right\|_{L^2(\cY;
\rho^0)}\leq \dfrac{2\kp m}{1-2\kp K\sqrt{m}}\|\cB
\Psi_{\ell,2}\|_{L^2(\cY; \rho^0)} \ . \nonumber
\end{align}

Now apply (\ref{Eq:noncenteredPoincare}) again we obtain the bound
\begin{align}
\|\Psi_{\ell,3}^m\|_{L^2(\cY; \rho^0)}\leq \dfrac{4\kp^2
m^{3/2}}{1-2\kp K\sqrt{m}}\|\cB \Psi_{\ell,2}\|_{L^2(\cY; \rho^0)} \
. \nonumber
\end{align}

This proves the lemma.
\end{proof}

\begin{proof}[Proof of Lemma \ref{mZeroL4bound}]

Let us write $\Psi$ in place of $\Psi^m_{\ell,3}$ for similicity of
notations. We set $f=\Psi^2$ and we look for the equation that $f$
satisfies:
\begin{align}
&\cL_q^m f
\nonumber \\
&=\dfrac{1}{m}\cA f+\dfrac{1}{\sqrt{m}}\cB f
\nonumber \\
&=\dfrac{1}{m}(-p\cdot\grad_p f+\Dt_p
f)+\dfrac{1}{\sqrt{m}}(b(r)\cdot \grad_p f+p\cdot \grad_r f)
\nonumber \\
&=\dfrac{1}{m}(-p\cdot(2\Psi\grad_p\Psi)+2|\grad_p
\Psi|^2+2\Psi\Dt_p\Psi)+\dfrac{1}{\sqrt{m}}(b(r)\cdot
2\Psi\grad_p\Psi+p\cdot 2\Psi\grad_r\Psi)
\nonumber \\
&=2\Psi(\cL_q^m \Psi)+\dfrac{2}{m}|\grad_p\Psi|^2 \nonumber \ .
\end{align}

Using the equation (\ref{Eq:ErrorEquation}) we see that
\begin{align}
&\cL_q^m f=-2\sqrt{m}\Psi\cB\Psi_{\ell,2}+\dfrac{2}{m}|\grad_p
\Psi|^2 \ . \label{Eq:EquationOfPsiSquare}
\end{align}

Making use of Lemma \ref{L:L2_Bounds1} we have
\begin{align}
\langle\cL_q^m f, f\rangle_{L^2(\cY;\rho^0)}=-\dfrac{1}{m}\|\grad_p
f\|^2_{L^2(\cY;\rho^0)}+\dfrac{1}{2\sqrt{m}}\langle h(r)\cdot p,
f^2\rangle_{L^2(\cY;\rho^0)} \ . \nonumber
\end{align}

This gives
\begin{align}
\|\grad_p f\|^2_{L^2(\cY;\rho^0)}=\dfrac{\sqrt{m}}{2}\langle
h(r)\cdot p, f^2\rangle_{L^2(\cY;\rho^0)}-m \langle\cL_q^m f,
f\rangle_{L^2(\cY;\rho^0)} \ . \label{Eq:GradientOfPsiSquare}
\end{align}

Making use of (\ref{Eq:EquationOfPsiSquare}),
(\ref{Eq:GradientOfPsiSquare}), the fact that $f\geq 0$ and Lemma
\ref{L:L2_Bounds3}  we get, for some constant $K>0$ independent of $m$ that may
vary from line to line,
\begin{align}
&\|\grad_p f\|^2_{L^2(\cY;\rho^0)}
\nonumber \\
&=\dfrac{\sqrt{m}}{2}\langle h(r)\cdot p,
f^2\rangle_{L^2(\cY;\rho^0)}+2m^{3/2}\langle \Psi \cB\Psi_{\ell,2},
f\rangle_{L^2(\cY;\rho^0)}-2\langle|\grad_p\Psi|^2,
f\rangle_{L^2(\cY;\rho^0)}
\nonumber \\
&\leq \dfrac{\sqrt{m}}{2}\langle h(r)\cdot p,
f^2\rangle_{L^2(\cY;\rho^0)}+2m^{3/2}\langle \Psi\cB\Psi_{\ell,2},
f\rangle_{L^2(\cY;\rho^0)}
\nonumber \\
&\leq \dfrac{\sqrt{m}}{2}K\|f\|_{L^2(\cY;\rho^0)}\|\grad_p
f\|_{L^2(\cY;\rho^0)}+m^{3/2}(\|\Psi\cB\Psi_{\ell,2}\|_{L^2(\cY;\rho^0)}^2+\|f\|_{L^2(\cY;\rho^0)}^2)
\nonumber \\
&\leq \dfrac{\sqrt{m}}{2}K\|f\|_{L^2(\cY;\rho^0)}\|\grad_p
f\|_{L^2(\cY;\rho^0)}+m^{3/2}(\|\Psi^2\|_{L^2(\cY;\rho^0)}^2\|(\cB\Psi_{\ell,2})^2\|_{L^2(\cY;\rho^0)}^2+\|f\|_{L^2(\cY;\rho^0)}^2)
\nonumber \\
&\leq K[\dfrac{\sqrt{m}}{2}\|f\|_{L^2(\cY;\rho^0)}\|\grad_p
f\|_{L^2(\cY;\rho^0)}+m^{3/2}\|f\|_{L^2(\cY;\rho^0)}^2] \ .
\label{Eq:CrucialBound}
\end{align}

Now we apply Lemma \ref{mZeroPoincareInequality} and we see that for
some $\kp>0$ we have
\begin{align}
&\|f\|_{L^2(\cY;\rho^0)}
\nonumber \\
&=\left\|f-\int_{\cY}f(p,r)\rho^0(p,r)dpdr+\int_{\cY}f(p,r)\rho^0(p,r)dpdr\right\|_{L^2(\cY;\rho^0)}
\nonumber \\
&\leq
\left\|f-\int_{\cY}f(p,r)\rho^0(p,r)dpdr\right\|_{L^2(\cY;\rho^0)}+\left|\int_{\cY}f(p,r)\rho^0(p,r)dpdr\right|
\nonumber \\
&\leq \kp \|\grad_p
f\|_{L^2(\cY;\rho^0)}+\|\Psi\|^2_{L^2(\cY;\rho^0)} \ . \nonumber
\end{align}

In the last step we used the fact that $f=\Psi^2$. Now we apply
Lemma \ref{mZeroL2bound1} and we see that
$\|\Psi\|_{L^2(\cY;\rho^0)}^2\leq Km^3$ for some constant $K>0$
independent of $m$. Thus we see that
\begin{align}
\|f\|_{L^2(\cY;\rho^0)}\leq K[\|\grad_p f\|_{L^2(\cY;\rho^0)}+m^3]
\label{Eq:noncenteredPoincareWithCorrection}
\end{align}

Combining (\ref{Eq:CrucialBound}) and
(\ref{Eq:noncenteredPoincareWithCorrection}) we see that
\begin{align}
\|\grad_p f\|^2_{L^2(\cY;\rho^0)}\leq K[\sqrt{m}\|\grad_p
f\|_{L^2(\cY;\rho^0)}^2+m^3\|\grad_p
f\|_{L^2(\cY;\rho^0)}+m^{3/2}\|\grad_p
f\|_{L^2(\cY;\rho^0)}^2+m^{3/2+6}] \ . \nonumber
\end{align}

This gives $\lim\li_{m\ra 0}\|\grad_p f\|^2_{L^2(\cY;\rho^0)}=0$.
Apply (\ref{Eq:noncenteredPoincareWithCorrection}) again we see that
the claim of the Lemma follows.
\end{proof}

\begin{proof}[Proof of Lemma \ref{mNonZeroIntegrationByParts}]
The proof of this lemma is completely analogous to that of Lemma \ref{L:L2_Bounds1} and thus it is omitted.
\end{proof}

\end{document}